%% file: FHN_strong_3.tex
\documentclass[11pt]{amsart}
\input{macro.tex}

\usetikzlibrary{fit}
\usetikzlibrary{arrows}
\numberwithin{equation}{section}
\hypersetup{linkbordercolor=red}

\title{
Large coupling in a FitzHugh-Nagumo neural network:\\
quantitative and strong convergence results}
\date {\today}

\begin{document}
\maketitle

\centerline{\scshape Alain Blaustein}
{\footnotesize
    \centerline{Universit\'e Paul Sabatier, 
Institut de Math\'ematiques de Toulouse,}
\centerline{
118 route de Narbonne - F-31062 TOULOUSE Cedex 9, France;}
    \centerline{\textit{email}:  \email{alain.blaustein@math.univ-toulouse.fr}}
}

\bigskip

\begin{abstract}
We consider a mesoscopic model for a spatially extended FitzHugh-Nagumo neural network and prove that in the regime where short-range interactions dominate, the probability density of the potential throughout the network concentrates into a Dirac distribution whose center of mass solves the classical non-local reaction-diffusion
FitzHugh-Nagumo system.
In order to refine our comprehension of this regime, we focus on the blow-up profile of this concentration phenomenon.
Our main purpose here consists in deriving two
quantitative and strong convergence estimates proving that the profile is Gaussian: the first one in a $L^1$ functional framework and the second in a weighted $L^2$ functional setting. We develop original 
relative entropy techniques to prove the first result whereas our second result relies on propagation of regularity.
\end{abstract}

\maketitle
\vspace{0.5cm}
\\

\tableofcontents

\section{Introduction}
\label{sec:1}
\subsection{Physical model and motivations}
Over the last century, mathematical models were built in order to describe biological neural activity, laying the groundwork for computational neuroscience. We mention the pioneer work A. Hodgkin and A. Huxley \cite{HH} who derived a precise model for the voltage dynamics of a nerve cell submitted to an external input. However a general and precise description of cerebral activity seems out of reach, due to the number of neurons, the complexity of their behavior and the intricacy of their interactions. Therefore, numerous simplified models arose from neuroscience over the last decade allowing to recover some of the behaviors observed in regimes or situations of interest. They may usually be interpreted as the mean-field limit of stochastic microscopic models. We mention integrate-and-fire neural networks \cite{CCP,CPSS,roux}, time-elapsed neuronal models \cite{CCDR,CHE1,CHE,PPS} and also voltage-conductance firing models \cite{PPS,Perthame/Salort13}. In this article, we study a FitzHugh-Nagumo neural field represented by its distribution $\mu(t,\bx,\bu)$ depending on time $t$, position $\bx \in K$ with $K$ a compact set of $\R^d$, and $\bu=(v,w)\in\R^2$ where $v$ stands for the membrane potential and $w$ is an adaptation variable. The distribution $\mu$ is normalized by the total density $\rho_0(\bx)$ of neurons at position $\bx$. Therefore $\mu$ is a non-negative function taken in 
$
\ds
\scC^0
\left(\,
\R^+
\times 
K\,,\,
L^1
\left(
\R^2
\right)
\right)
$ which verifies
\[
\int_{\R^2}
\,
\mu(t,\bx,\bu)\,\dD \bu
\,=\,
1
,\quad
\forall
\left(
t,\bx
\right)
\,\in\,
\R^+
\times 
K\,,\]
and which solves the following McKean-Vlasov equation (see \cite{limite-hydro-crevat,mfieldlimitcrevat,mlimit3} for other instances of such model)
\begin{equation*}
	\ds\partial_t \, \mu
	\,+\,
	\partial_v 
	\left( 
	\left(N(v)
	-w
	- \mathcal{K}_{\Phi}[\rho_0\,\mu]\right) \mu\right) \,+\,
	\partial_w
	\left( 
	A\left(v,w\right)\mu\right) 
	\,-\,
	\partial_v^2 \mu\,=\,0\,,
\end{equation*}
where the non-linear term $\mathcal{K}_{\Phi}[\rho_0\,\mu]\, \mu$ is induced by non-local electrostatic interactions: we suppose that neurons interact through Ohm's law and that the conductance between two neurons is given by an interaction kernel $\Phi:K^2\rightarrow \R$ which depends on their position, this yields
\[
\mathcal{K}_{\Phi}[\rho_0\,\mu](t,\bx,v)
\,=\,
\int_{K\times\R^{2}}
\Phi(\bx,\bx')\, (v-v')\,\rho_0(\bx')\mu(t,\,\bx',\bu')
\dD\bx'\,\dD \bu'\,.
\]
The other terms in the McKean-Vlasov equation are associated to the individual behavior of each neuron, driven here by the model of R. FitzHugh and J. Nagumo in \cite{FHN1,FHN2}. On the one hand, the drift $N\in \scC^2\left(\R\right)$ is a confining non-linearity: setting
$
\ds
\omega(v)
\,=\,
N(v)/v
$ we suppose
\begin{subequations}
	\begin{numcases}{}
		\label{hyp1:N}
		\ds
		\limsup_{|v|\,\rightarrow\,+ \infty}\,
		\omega(v)
		\,=\,-\infty\,,\\[0,9em]
		\label{hyp2:N}
		\ds
		\sup_{|v|\,\geq\,1}\,
		\frac{
			\left|\omega(v)\right|
		}{|v|^{p-1}}
		\,<\, + \infty\,,
	\end{numcases}
\end{subequations}
for some $p \geq 2$. For instance, these assumptions are met by the original model proposed by R. FitzHugh and J. Nagumo, where $N$ is a cubic non-linearity
\[
N(v)\,=\, v \,-\, v^3\,.
\] 
On the other hand, $A$ drives the dynamics of the adaptation variable, it is given by
\begin{equation*}
	A(v,w) \,=\, a\,v \,-\, b\,w \,+\, c\,,
\end{equation*} 
where $a$, $c \in \R$ and $b>0$. We also add a diffusion term with respect to $v$ in order to take into account random fluctuations of the voltage. 
This type of model has been rigorously derived as the mean-field limit of microscopic model in \cite{mlimit1, mlimit2, mfieldlimitcrevat, mlimit4, mlimit3}. Well posedeness of the latter equation is well known and will not be discussed here. We refer to \cite[Theorem $2.3$]{BF} for a precise discussion over that matter.
\subsection{Regime of strong short-range interactions}\label{formal derivation} In this article, we consider a situation where $\Phi$ decomposes as follows
\[
\Phi(\bx,\bx')\,=\,\Psi(\bx,\bx')\,+\,\frac{1}{\eps}\,\delta_0(\bx-\bx')\,,
\]
where the Dirac mass $\delta_0$ accounts for short-range interactions whereas the interaction kernel $\Psi$ models long-range interactions: it is "smoother" than $\delta_0$ since we suppose 
\begin{equation}
	\label{hyp2:psi}
	\Psi
	\,
	\in
	\,
	\scC^0
	\left(
	K_{\bx},\,
	L^1
	\left(
	K_{\bx'}
	\right)
	\right)\;\quad\text{and}\quad
	\sup_{\bx\in K} \int_{K}\left|\Psi(\bx',\bx)\right|\,+\,\left|\Psi(\bx,\bx')\right|^r \,\dD\bx' \,<\, +\infty\,,
\end{equation}
for some $r>1$ (we denote $r'$ its conjugate: $r' = (r-1)/r$). We point out that our assumptions on $\Psi$ are quite general, this is in line with other works which put a lot of effort into considering general interactions \cite{JPS}. \\
The scaling parameter $\eps>0$ represents the magnitude of short-range interactions; we focus on the regime where they dominate, that is, when $\eps \ll 1$. From these assumptions, the equation on $\mu$ can be rewritten as
\begin{equation}
	\label{kinetic:eq}
	\ds\partial_t \, \mu^\eps
	\,+\,
	\partial_v 
	\left( 
	\left(N(v)
	-w
	- \mathcal{K}_{\Psi}[\rho_0^\eps\,\mu^\eps]\right) \mu^\eps\right) \,+\,
	\partial_w
	\left( 
	A\left(v,w\right)\mu^\eps\right) 
	\,-\,
	\partial_v^2 \mu^\eps\,=\,\frac{\rho_0^\eps}{\eps}\,\partial_v \left(
	\left( v-\cV^\eps\right) \mu^\eps\right)
	,
\end{equation}
where the averaged voltage and adaptation variables $\ds\, \mathcal{U}^\eps 
= 
\left(\,
\mathcal{V}^\eps ,
\mathcal{W}^\eps\, 
\right)
$ at a spatial location $\bx$ are defined as
\begin{equation}\label{macro:q}
	\left\{
	\begin{array}{ll}
		\displaystyle \cV^\eps(t,\bx) 
		&\,=\,
		\ds\int_{\R^2}v~\mu^\eps(t,\bx,\bu)\,\dD\bu\,,
		\\[1.1em]
		\displaystyle \cW^\eps(t,\bx) 
		&\,=\,
		\ds\int_{\R^2}w~\mu^\eps(t,\bx,\bu)\,\dD\bu\,.
	\end{array}
	\right.
\end{equation}
Previous works already went through the analysis of the asymptotic $\eps \ll 1$ and it was proven that in this regime the voltage distribution undergoes a concentration phenomenon. We mention \cite{limite-hydro-crevat} in which is investigated this asymptotic in a deterministic setting using relative entropy methods and also \cite{HJ quininao/touboul} in which authors study this model in a spatially homogeneous framework following a Hamilton-Jacobi approach. 
These works conclude that as $\eps$ vanishes,
$\mu^\eps$ converges as follows
\begin{equation*}
	\mu^\eps(t,\bx,\bu)
	\;\underset{\eps \rightarrow 0}{\longrightarrow}
	\; \delta_{\cV(t,\bx)}(v)
	\otimes \bar{\mu}(t,\bx,w)\,,
\end{equation*}
where the couple $\ds\left(\cV,\bar{\mu}
\right)\,$ solves\\[-0.3em]
\begin{equation}
	\label{macro:eq}
	\left\{
	\begin{array}{l}
		\displaystyle \partial_t\,\cV
		\,=\, 
		N(\cV)
		\,-\,
		\cW
		\,-\,
		\mathcal{L}_{\rho_0}[\,\cV\,]
		\,,\\[0.9em]
		\displaystyle \partial_t\, \bar{\mu}
		\,+\,
		\partial_w
		\left(
		A(\cV\,,\,w)\, \bar{\mu}
		\right) \,=\, 0\,,
	\end{array}
	\right.
\end{equation}
with
$$
\cW
\,=\,
\int_{\R}
w\,
\bar{\mu}
\left(t\,,\,\bx\,,\,w\right)\,\dD w\,,
$$
and where $\ds\mathcal{L}_{\rho_0}
\left[\,\cV\,
\right]$ is a non local operator given by
\[
\mathcal{L}_{\rho_0}
\left[\,\cV\,
\right]
=
\cV \,\Psi*_r\rho_0 \,-\, \Psi*_r(\rho_0 \cV)\;,
\]
where $*_r$ is a shorthand notation for the convolution on the right side of any function $g$ with $\Psi$
\[
\Psi*_r g(\bx)
\;=\,
\int_{K}
\Psi(\bx,\bx')\,g(\bx')\,\dD \bx'\,.
\]
Then, the concentration profile of $\mu^\eps$ around $\delta_{\cV}$ was investigated in \cite{BF}. The strategy consists in considering the following re-scaled version $\nu^\eps$ of $\mu^\eps$
\begin{equation}\label{def:nueps}
	\mu^\eps
	\left(t,\,\bx,\,\bu\right) 
	\,=\,
	\frac{1}{\theta^\eps}\,
	\nu^\eps\left(t,\,\bx,\,
	\frac{v-\mathcal{V}^\eps}{\theta^\eps},\,
	w-\mathcal{W}^\eps
	\right),
\end{equation}
where $\theta^\eps$ shall be interpreted as the concentration rate of $\mu^\eps$ around its mean value $\mathcal{V}^\eps$. Under the scaling $\theta^\eps\,=\,\sqrt{\eps}$, it is proven that
\begin{align}\label{limit nu}
	\nu^\eps\;
	&\underset{\eps\rightarrow 0}{\sim}\;
	\nu\,:=\,
	\cM_{\rho_0^\eps}\otimes \bar{\nu}\,,
\end{align}
where $\bar\nu$ solves the following linear transport equation
\begin{equation}\label{bar nu:eq}
	\partial_t\, \bar\nu
	\,-\,b\,
	\partial_{w}
	\left( 
	w \,\bar\nu\right) \,=\, 0\,,
\end{equation}
and where the Maxwellian $\cM_{\rho_0^\eps}$
is defined as
\begin{equation*}
	\mathcal{M}_{\rho_0^\eps(\bx) }(v)=
	\sqrt{\frac{\rho_0^\eps(\bx)}{2\pi}}\exp
	\left(
	-\,\rho_0^\eps(\bx) \,\frac{|v|^2}{2}
	\right).
\end{equation*}
The latter convergence translates on $\mu^\eps$ as follows
\begin{equation}\label{expansion mu eps order 1}
	\mu^\eps(t,\bx,\bu) \underset{\eps \rightarrow 0}{\sim}
	\mathcal{M}_{
		\rho_0\,
		\left|
		\theta^\eps
		\right|^{-2}
	}
	\left(
	v\,-\,\cV
	\right)
	\otimes
	\ols{\mu}(t,\bx,\bu)
	\,,
\end{equation}
with $\theta^\eps\,=\,\sqrt{\eps}$. More precisely, it was proven in \cite{BF} that \eqref{expansion mu eps order 1} occurs up to an error of order 
$
\ds
\eps
$ in the sense of weak convergence in some probability space. Our goal here is to strengthen the results obtained in \cite{BF} by providing strong convergence estimates for \eqref{expansion mu eps order 1}. The general strategy consists in deducing  \eqref{expansion mu eps order 1} from \eqref{limit nu}. The main difficulties to achieve this is twofold. On the one hand, since the norms associated to strong topology are usually not scaling invariant, the time homogeneous scaling $\theta^\eps
\,=\,
\sqrt{\eps}\,$ comes down to considering well-prepared initial conditions. Therefore we find an appropriate scaling $\theta^\eps$ which enables to treat general initial condition. On the other hand, the proof is made challenging by the cross terms between $v$ and $w$ in \eqref{kinetic:eq}. This issue is analogous to the difficulty induced by the free transport operator in the context of kinetic theory \cite{El Ghani/ Masmoudi,Herda,Herda/Rodrigues}. In our context, we propagate regularity in order to overcome this difficulty and obtain error estimates. \\

This article is organized as follows. We start with Section \ref{sec:main:th}, in which we carry out an heuristic in order to derive the appropriate scaling $\theta^\eps$ and then state our two main results. We first provide convergence estimates for $\mu^\eps$ in a $L^1$ setting,
which is the natural space to consider for such type of conservative problem (see Theorem \ref{th:12}). This result is a direct consequence of the convergence of the re-scaled distribution $\nu^\eps$ (see Theorem \ref{th1}) which is obtained in Section \ref{sec1b}. Then we propose convergence estimates in a weighted $L^2$ setting (see Theorem \ref{th21}). Once again this result is the consequence of the convergence of $\nu^\eps$ (see Theorem \ref{th:2}) provided in Section \ref{sec2}. We emphasize that the latter result allows us to recover the optimal convergence rates obtained in \cite{BF} and to achieve pointwise convergence estimates with respect to time. This analysis is in line with \cite{mlimit3}, which focuses on the regime of weak interactions between neurons (this corresponds to the asymptotic $\ds\eps \rightarrow +\infty$ in equation \eqref{kinetic:eq}).

\section{Heuristic and main results}\label{sec:main:th}
As mentioned before the time homogeneous scaling $\theta^\eps
\,=\,
\sqrt{\eps}\,$ in the definition \eqref{def:nueps} of $\nu^\eps$ comes down to considering well-prepared initial conditions. We seek for a stronger result which also applies for ill-prepared initial conditions.
To overcome this difficulty, our strategy consists in adding the following constraint on the concentration rate
\[
\theta^\eps
\left(
t\,=\,0
\right)\,=\,1\,.\] 
In this setting, the equation on $\nu^\eps$ is obtained performing the following change of variable
\begin{equation}
	\label{change:var}
	(t\,,\,v\,,\,w)\mapsto \left(t\,,\,\frac{v-\cV^\eps}{\theta^\epsilon}\,,\, w-\cW^\eps\right)
\end{equation}
in equation \eqref{kinetic:eq}. Following computations detailed in \cite{BF}, it yields
\begin{equation}\label{nu:eq:-1}
	\partial_t\, \nu^\eps
	\,+\,
	\mathrm{div}_{\bu}
	\left[
	\,\mathbf{b}^\eps_0\,\nu^\eps \right]
	\,=\,
	\frac{1}{|\theta^\eps|^2}\,
	\partial_v
	\left[
	\left(
	\frac{1}{2} \,\frac{\dD}{\dD t}\,
	\left|\theta^\eps\right|^2
	+ 
	\frac{\rho_0^\eps}{\eps}
	\,
	\left|\theta^\eps\right|^2
	\right)
	v \,\nu^\eps
	+
	\partial_v \nu^\eps
	\right]\,,
\end{equation}
where $\mathbf{b}^\eps_0$ is given by
\begin{equation}\label{def:b0}
	\ds
	\mathbf{b}^\eps_0
	\left(
	t,\bx,\bu
	\right)
	\,=\,
	\begin{pmatrix}
		\ds\,
		\left(
		\theta^\eps
		\right)^{-1}
		B^{\eps}_0
		\left(t,\bx,
		\theta^\eps\, v, w
		\right)\,\\[0,9em]
		\ds
		A_0
		\left(\theta^\eps v, w
		\right)
	\end{pmatrix}
\end{equation}
and $B^{\eps}_0$ is defined as
\[
B^{\eps}_0(t\,,\,\bx\,,\,\bu)
\,=\,
N(\cV^\eps
\,+\,
v)
\,-\,
N(\cV^\eps) 
\,-\,
w
\,-\,
v\, \Psi *_r \rho^\eps_0(\bx)
\,-\,
\mathcal{E}
\left(
\mu^\eps
\right),
\]
with $
\ds\cE
\left(\mu^\eps
\right)$ the following error term
\begin{equation}\label{error}
	\cE
	\left(
	\mu^\eps\left(t,\bx,\,\cdot\,
	\right)
	\right)
	\;=\,
	\;\int_{\R^2} N(v)\,
	\mu^\eps
	\left(t,\bx,\bu
	\right)\,
	\dD \bu\,-\, N\left(\cV^\eps\left(t,\bx
	\right)\right)\,,
\end{equation}
and where $A_0$ is the linear version of $A$
\begin{equation*}
	A_0(\bu) = A(\bu)-A(\mathbf{0})\,.
\end{equation*}
Considering the leading order in \eqref{nu:eq:-1} and since we expect concentration with Gaussian profile $\mathcal{M}_{\rho_0^\eps }$, $\theta^\eps$ should verify 
\begin{equation*}
	\left\{
	\begin{array}{l}
		\ds
		\frac{1}{2} \,\frac{\dD}{\dD t}\,
		\left|\theta^\eps\right|^2
		\,+\, 
		\frac{\rho_0^\eps}{\eps}\,
		\left|\theta^\eps\right|^2
		\,=\,
		\rho_0^\eps\,,
		\\[1.1em]
		\ds \theta^\eps
		\left(
		t\,=\,0
		\right)\,=\,1\,,
	\end{array}\right.
\end{equation*}
whose solution is given by the following explicit formula
\begin{equation}\label{expression for theta eps}
	\theta^\eps(t,\,\bx)^2
	\,=\,
	\eps
	\left(1 \,-\,
	\exp{
		\left(-
		(\,2\,\rho_0^\eps(\bx)\, t\,)\,/\,\eps
		\right)}
	\right)
	\,+\,
	\exp{
		\left(-(\,2\,\rho_0^\eps(\bx)\, t\,)\,/\,\eps
		\right)}\,
	.
\end{equation}
Therefore, we obtain a time dependent $\theta^\eps$, which is of order $\sqrt{\eps}$, up to an exponentially decaying correction to authorize ill-prepared initial conditions. With this choice, the equation on $\nu^\eps$ rewrites:
\begin{equation}\label{nu:eq}
	\partial_t\, \nu^\eps
	\,+\,
	\mathrm{div}_{\bu}
	\left[
	\,\mathbf{b}^\eps_0\,\nu^\eps\, \right]
	\,=\,
	\frac{1}{
		\left|\theta^\eps\right|^2}\,
	\cF_{\rho_0^\eps}
	\left[\,
	\nu^\eps\,
	\right]\,,
\end{equation}
where
$
\displaystyle
\mathbf{b}^\eps_0
$ is given by \eqref{def:b0} and the Fokker-Planck operator is defined as
\[
\cF_{\rho_0^\eps}
\left[\,
\nu^\eps\,
\right]\,
=\,
\partial_v
\left[\,\rho_0^\eps\,
v \,\nu^\eps
+
\partial_v \, \nu^\eps\,
\right]\,.
\] 

Let us now precise our assumptions on the initial data. We suppose the following uniform boundedness condition on the spatial distribution $\rho_0^\eps$
\begin{equation}\label{hyp:rho0}
	\rho_0^\eps \in 
	\scC^0
	\left(\,
	K\,
	\right)
	\quad\mathrm{and}\quad
	m_* \leq \rho_0^\varepsilon \leq 1/m_*\,,
\end{equation} 
as well as moment assumptions on the initial data
\begin{subequations}
	\begin{numcases}{}
		\label{hyp1:f0}
		\sup_{\bx \in K}
		\,\int_{\R^2}
		|\bu|^{2p} \mu^\eps_0
		\left(\bx,\bu
		\right)\,
		\dD \bu
		\,\leq\, m_p\,,\\[0,9em]
		\label{hyp2:f0}
		\ds
		\int_{K \times \, \R^2}\,
		|\bu|^{2p r'}
		\,\rho_0^\eps(\bx)\,\mu^\eps_0
		\left(\bx,\bu
		\right)
		\dD \bu\,\dD \bx
		\,\leq\, \ols{m}_p\,,
	\end{numcases}
\end{subequations}
where $p$ and $r'$ are given in \eqref{hyp2:N} and \eqref{hyp2:psi}, for constants $m_*,m_p,\ols{m}_p$ uniform with respect to $\eps$. \\

Since well posedeness of the mean-field equation \eqref{kinetic:eq} and the limiting model \eqref{macro:eq} is well known, we do not discuss it here and refer to \cite[Theorems $2.3$ and $2.6$]{BF} for a precise discussion on that matter. To apply these results, we suppose the following assumptions which are not uniform with respect to $\eps$ on $\mu^\eps_0$
\begin{equation}
	\label{hyp3:f0}
	\left\{
	\begin{array}{l}
		\ds\sup_{\bx \in K}
		\int_{\R^2}
		e^{|\bu|^2/2}\,
		\mu^\eps_{0}(\bx,\bu)\dD \bu
		<\, +\infty\,,
		\\[1.1em]
		\ds \sup_{\bx \in K}
		\left\|
		\nabla_{\bu}
		\sqrt{\mu^\eps_{0}}(\bx,\cdot)\,
		\right\|_{L^2(\R^2)}\,
		<\, +\infty\,,
	\end{array}\right.
\end{equation}
and for the limiting problem \eqref{macro:eq}, we suppose
\begin{equation}\label{hyp:macro}
	\left(
	\mathcal{V}_0\,,\,
	\bar{\mu}_0
	\right)
	\in 
	\scC^0
	\left(
	K
	\right)
	\times
	\scC^0
	\left(
	K\,,\,
	L^{1}
	\left(
	\R
	\right)
	\right)\,.
\end{equation}
\\~\\
All along our analysis, we denote by $\tau_{w_0}$ the translation by $w_0$ with respect to the $w$-variable, for any given $w_0$ in $\R$
\[
\tau_{w_0}\,
\nu
\left(t,\bx,v,w
\right)
\,=\,
\nu
\left(t,\bx,v,w
+w_0
\right)\,.
\]
\subsection{$L^1$ convergence result}
In the following result, we provide explicit convergence rates for $\nu^\eps$ towards the asymptotic concentration profile of the neural network's distribution $\mu^\eps$ in the regime of strong interactions in a $L^1$ setting. We will use the notation
\[
L^{\infty}_{\bx}
L^1_{\bu}
\,:=\,
L^{\infty}
\left(
K\,,\,
L^1
\left(
\R^2
\right)
\right),\quad\text{and}\quad L^{\infty}_{\bx}
L^1_{w}
\,:=\,
L^{\infty}
\left(
K\,,\,
L^1
\left(
\R
\right)
\right)\,.
\]
We prove that the profile of concentration with respect to $v$ is Gaussian and we also characterize the limiting distribution with respect to the adaptation variable $w$. We denote by
$
\ds H
$ the Boltzmann entropy, defined for all function $\mu:\R^2\rightarrow \R^+$ as follows
\begin{equation*}
	H\left[\,\mu\,\right]\,=\,\int_{\R^2}
	\mu
	\,\ln{
		\left(
		\mu
		\right)
	}
	\,\dD\bu\,.
\end{equation*}

\begin{theorem}\label{th1}
	Under assumptions \eqref{hyp1:N}-\eqref{hyp2:N} on the drift $N$, assumption \eqref{hyp2:psi} on the interaction kernel $\Psi$,
	consider the unique sequence of solutions $(\mu^\eps)_{\eps\,>\,0}$ to \eqref{kinetic:eq} with initial conditions satisfying assumptions \eqref{hyp:rho0}-\eqref{hyp3:f0} and
	the solution $\ols{\nu}$ to equation \eqref{bar nu:eq} with an initial condition $\ols{\nu}_0$ such that
	\begin{equation}\label{hyp bar nu 0 L 1}
		\ols{\nu}_0 \in
		L^{\infty}
		\left(
		K\,,\,
		W^{2\,,\,1}
		\left(
		\R
		\right)
		\right)\,,\quad
		\textrm{and}\quad\sup_{\bx\in K}
		\int_{\R}\left|w\,\partial_w\,\ols{\nu}_0(\bx,w)\right|\dD w<+\infty
		\,.
	\end{equation}
	Moreover, suppose that there exists a positive constant $m_{1}$ such that for all $\left(\gamma,w_0\right)\in\left(\R^*\right)^2$
	\begin{equation}\label{hyp1 nu L 1}
		\sup_{\eps\,>\,0}\,
		\left(
			\frac{1}{|\gamma|}		
			\left\|
			\nu^\eps_0\,-\,
			\tau_{\gamma v}\,\nu^\eps_0\,
			\right\|_{
				L^{\infty}_{\bx}
				L^1_{\bu}}
		+
		\frac{1}{|w_0|}
		\left\|
		\nu^\eps_0\,-\,
		\tau_{w_0}\,\nu^\eps_0\,
		\right\|_{
			L^{\infty}_{\bx}
			L^1_{\bu}}\right)
		\,\leq\,
		m_{1}
		\,,
	\end{equation}
	and a positive constant $m_2$ such that
	\begin{equation}\label{hyp2 nu L 1}
		\sup_{\eps \,>\, 0}\,
		\left\|\,
		H
		\left[
		\,\nu^\eps_0\,
		\right]
		\,
		\right\|_
		{
			L^{\infty}
			\left(
			K
			\right)
		}
		\,\,\leq\,\,
		m_2^2
		\,.
	\end{equation}
	Then, there exists a positive constant $C$ independent of $\eps$ such that for all $\eps$ less than $1$, it holds
	\[
	\left\|
	\,\nu^\eps
	\,-\,
	\mathcal{M}_{\rho^\eps_0}\otimes
	\ols{\nu}\,
	\right\|_{
		L^{\infty}\left(K\,,\,L^1\left([\,0,\,t\,]\times \R^2\right)\right)
	}
	\,\leq\,
	2\,\sqrt{2}
	\,t\,
	\left\|\,
	\ols{\nu}^\eps_0
	\,-\,
	\ols{\nu}_0
	\right\|^{1/2}_{L^{\infty}_{\bx}
		L^{1}_{w}}
	\,+\,\sqrt{\eps}\,
	\left(
	4\,\sqrt{t}\,
	m_2
	\,+\,
	C\,e^{b\,t}\right)\,,
	\]
	for all time $t\geq0$.
	In particular, under the compatibility assumption
	\[
	\left\|\,
	\ols{\nu}^\eps_0
	\,-\,
	\ols{\nu}_0
	\right\|^{1/2}_{L^{\infty}_{\bx}
		L^{1}_{w}}\,
	\underset{\eps \rightarrow 0}{=}\,
	O\left(\,
	\sqrt{\eps}\,\right)\,,
	\]
	it holds
	
	\[
	\sup_{t\,\in \,\R^+}
	\left(\,
	e^{-b\,t}
	\left\|
	\,\nu^\eps
	\,-\,
	\mathcal{M}_{\rho^\eps_0}\otimes
	\ols{\nu}\,
	\right\|_{
		L^{\infty}\left(K\,,\,L^1\left([\,0,\,t\,]\times \R^2\right)\right)
	}
	\,\right)
	\,
	\underset{\eps \rightarrow 0}{=}\,
	O\left(\,\sqrt{\eps}\,\right)\,.\\[0,4em]
	\]
	In this result, the constant $C$ only depends on $m_1$, $m_*$,~$m_p$~and~$\ols{m}_p$ (see assumptions \eqref{hyp1 nu L 1}, \eqref{hyp:rho0}-\eqref{hyp2:f0}) and the data of the problem $\ols{\nu}_0$, $N$, $\Psi$ and $A_0$.
\end{theorem}
The proof of this result is divided into two steps. First, we prove that $\nu^\eps$ converges towards the following local equilibrium of the Fokker-Planck operator
\[
\cM_{\rho_0^\eps}\otimes \bar{\nu}^\eps\,,
\]
where $\bar{\nu}^\eps$ is the marginal of $\nu^\eps$ with respect to the re-scaled adaptation variable
\begin{equation*}
	\bar{\nu}^\eps(t\,,\,\bx\,,\,w)
	\,=\,
	\int_{\R}\nu^\eps
	(t\,,\,\bx\,,\,\bu)\,\dD v\,,
\end{equation*}
and solves the following equation,
obtained after integrating  equation \eqref{nu:eq} with respect to $v$
\begin{equation} \label{bar:nu eps:eq}
	\displaystyle \partial_t\, \bar{\nu}^\eps
	\,-\,b\,
	\partial_{w}
	\left( 
	w\, \bar{\nu}^\eps
	\right)\,=\,
	-a\,\theta^\eps\,\partial_{w}\int_\R v\,\nu^\eps(\,t\,,\,\bx\,,\,\bu\,)\,\dD v\,.
\end{equation} 
The argument relies on a rather classical free energy estimate.
However, the analysis becomes more intricate when it comes to the convergence of the marginal $\ols{\nu}^\eps$. As already mentioned, the proof of convergence is made challenging by cross terms between $v$ and $w$ in equation \eqref{nu:eq} inducing in equation \eqref{bar:nu eps:eq} the following term which involves derivatives of $\nu^\eps$
\[
\partial_{w}\int_\R v\, \nu^\eps(\,t\,,\,\bx\,,\,\bu\,)\,\dD v\,.
\]
To overcome this difficulty, we perform a change of variable which cancels the latter source term and then conclude by proving a uniform equicontinuity estimate.\\
To conclude this discussion, we point out that the following condition on the initial data is sufficient in order to meet assumption \eqref{hyp1 nu L 1} 
	\begin{equation*}
		\sup_{\eps\,>\,0}\,
		\left\|
		(1+|v|)\partial_w\nu^\eps_0\,
		\right\|_{
			L^{\infty}_{\bx}
			L^1_{\bu}}
		\,\leq\,
		m_{1}
		\,.
	\end{equation*}
	Indeed, for all $(\bx,v,w_1)\in K\times\R^2$, it holds
	\[
	\int_{\R}
	\left|\nu^\eps_0
	- \tau_{w_1}\nu^\eps_0
	\right|\left(\bx,\bu\right)\,\dD w
	\leq
	|w_1|
	\int_{\R}
	\left|\partial_w\nu^\eps_0\left(\bx,\bu\right)\right|\dD w\,.
	\]
	Therefore, taking the sum between the latter estimate with $w_1=\gamma v$, divided by $|\gamma|$ and with $w_1=w_0$, divided by $|w_0|$, integrating with respect to $v\in\R$ and taking the supremum over all $\bx\in K$ it yields
	\[
	\frac{1}{|\gamma|}		
	\left\|
	\nu^\eps_0\,-\,
	\tau_{\gamma v}\,\nu^\eps_0\,
	\right\|_{
		L^{\infty}_{\bx}
		L^1_{\bu}}
	+
	\frac{1}{|w_0|}
	\left\|
	\nu^\eps_0\,-\,
	\tau_{w_0}\,\nu^\eps_0\,
	\right\|_{
		L^{\infty}_{\bx}
		L^1_{\bu}}\leq
	\left\|
	(1+|v|)\partial_w\nu^\eps_0\,
	\right\|_{
		L^{\infty}_{\bx}
		L^1_{\bu}}\,,
	\]
	for all $\left(\gamma,w_0\right)\in\left(\R^*\right)^2$.\\

We now interpret Theorem \ref{th1} on the solution $\mu^\eps$ to equation \eqref{kinetic:eq} in the regime of strong interactions
\begin{theorem}\label{th:12}
	Under the assumptions of Theorem \ref{th1}
	consider the unique sequence of solutions $(\mu^\eps)_{\eps\,>\,0}$ to \eqref{kinetic:eq} as well as
	the unique solution $\ds\left(\cV\,,\,\ols{\mu}\right)$ to equation \eqref{macro:eq} with an initial condition $\ols{\mu}_0$ which fulfills assumption \eqref{hyp:macro}-\eqref{hyp bar nu 0 L 1}.  Furthermore, suppose the following compatibility assumption to be fulfilled
	\begin{equation}\label{th12:compatibility assumption 1}
		\left\|\,
		\mathcal{U}_0
		\,-\,
		\mathcal{U}_0^\eps\,
		\right\|_{L^{\infty}(K)}
		\,+\,
		\|\,
		\rho_0
		-
		\rho_0^\eps\,
		\|_{L^{\infty}(K)}\,+\,
		\left\|\,
		\ols{\mu}^\eps_0
		\,-\,
		\ols{\mu}_0
		\right\|^{1/2}_{L^{\infty}_{\bx}
			L^{1}_{w}}\,
		\underset{\eps \rightarrow 0}{=}
		\,
		O\left(\,\sqrt{\eps}\,\right)
		\,.
	\end{equation}
	There exists 
	$\ds
	\left(
	C\,,\,\eps_0
	\right)
	\in 
	\left(
	\R^*_+
	\right)^2
	$
	such that for all $\eps$ less than $\eps_0$, it holds
	\[
	\left\|
	\,\mu^\eps
	\,-\,
	\mu\,
	\right\|_{
		L^{\infty}\left(K\,,\,L^1\left([\,0,\,t\,]\times \R^2\right)\right)
	}
	\,\leq\,
	C\,e^{C\,t}\,\sqrt{\eps}\,
	\,,
	\quad
	\forall\,
	t\,\in\,\R^+\,,
	\]
	where the limit $\mu$ is given by
	\[
	\mu
	\,:=\,
	\mathcal{M}_{
		\rho_0\,
		\left|
		\theta^\eps
		\right|^{-2}
	}
	\left(
	v\,-\,\cV
	\right)
	\otimes
	\ols{\mu}\,.
	\]
	In this result, the constant $C$ and $\eps_0$ only depend on the implicit constant in assumption \eqref{th12:compatibility assumption 1}, on the constants $m_1$, $m_2$ $m_*$,~$m_p$~and~$\ols{m}_p$ (see assumptions
	\eqref{hyp1 nu L 1}-\eqref{hyp2 nu L 1} and
	\eqref{hyp:rho0}-\eqref{hyp2:f0}) and  on the data of the problem $\ols{\mu}_0$, $N$, $\Psi$ and $A_0$.
\end{theorem}
\begin{proof}
	Since the norm $\ds
	\|\,\cdot\,\|_{
		L^1
		\left(\R^2\right)}$ is unchanged by the re-scaling \eqref{change:var}, this theorem is a straightforward consequence of Theorem \ref{th1} and Proposition \ref{th:preliminary}, which ensures the convergence of the macroscopic quantities
	$
	\ds
	\left(
	\cV^\eps,
	\cW^\eps
	\right)
	$
	.
\end{proof}
On the one hand we obtain $L^1$ in time convergence result, which is a consequence of our method, which relies on a free energy estimate for solutions to \eqref{nu:eq}. This is somehow similar to what is obtained in various classical kinetic models. Let us mention for instance the diffusive limit for collisional Vlasov-Poisson \cite{El Ghani/ Masmoudi,Herda,Masmoudi/Tayeb}. On the other hand, we obtain the convergence rate $\displaystyle
O(\sqrt{\eps})
$ instead of the optimal convergence rate, which should be $\displaystyle
O(\eps)
$ as rigorously proven for weak convergence metrics (see \cite{BF}, Theorem $2.7$). This is due to the fact that we use the Csiz\'ar-Kullback inequality
to close our $L^1$ convergence estimates. Therefore it seems quite
unlikely to recover the optimal convergence rate in a $L^1$
setting. This motivates our next result, on  the $L^2$
convergence (Theorems \ref{th:2} and \ref{th21}), in which pointwise in time convergence is achieved.\\

\subsection{Weighted $L^2$ convergence result}
In this section, we provide a pair of result analog to Theorems \ref{th1} and \ref{th:12} this time in a weighted $L^2$ setting. Since our approach relies on propagating $w$-derivatives, we introduce the following functional framework
\begin{equation*}
	\scH^{k}
	\left(m ^\eps\right)
	\,=\,L^{\infty}\left(K_{\bx},H^{k}_{w}\left(m ^\eps_{\bx}\right)\right)\,,
\end{equation*}
equipped with the norm 
\begin{equation*}
	\left\|
	\,\nu\,
	\right\|_{\scH^k(m ^\eps)}
	\,=\,
	\sup_{\bx \in K}\,
	\left\{
	\left\|\,
	\nu
	\left(\bx\,,\,\cdot\,
	\right)\,
	\right\|
	_{H^{k}_{w}
		\left(m^\eps_{\bx}
		\right)}
	\right\}\,,
\end{equation*}
where $\ds H^{k}_{w}(m^\eps_{\bx})$ denotes the weighted Sobolev space with index $k\in\N$ whose norm is given by
\[
\left\|\,
\nu\,
\right\|^2_
{
	H^{k}_{w}(m^\eps_{\bx})
}
\,=\,
\sum_{\,l\,\leq\,k}\,
\int_{\R^2}
\left|\partial_w^{\,l}\, \nu(\bu)\right|^2\,m^\eps_{\bx}(\bu)\dD \bu\,,
\]
and where the weight $m ^\eps_{\bx}$ is given by
\begin{equation}\label{def L 2 m}
	m^\eps_{\bx}(\bu)
	\,=\,
	\frac{2\,\pi}{
		\sqrt{
			\rho_0^\eps\,\kappa
		}
	}\,
	\exp{
		\left(\,
		\frac{1}{2}
		\left(\,
		\rho_0^\eps(\bx)\,|v|^2
		\,+\,
		\kappa\,
		|w|^2\,
		\right)
		\right)
	}
	\,,
\end{equation} 
for some exponent $\kappa > 0$ which will be
prescribed later. We also introduce the associated weight
with respect to the adaptation variable
\[
\ols{m}(w)
\,=\,
\sqrt{
	\frac{2\,\pi}{\kappa
}}\,
\exp{
	\left(\,
	\frac{\kappa}{2}
	\,
	|w|^2
	\right)
}\,.
\]
and denote by 
$\ds \scH^{k}
\left(\ols{m}\right)$ the corresponding functional space associated to  the marginal $\bar{\nu}$, depending only on
$(\bx,w)\in K\times \R$.

Hence, the following result tackles the convergence of $\nu^\eps$ in the $L^2$ weighted setting
\begin{theorem}\label{th:2}
	Under assumptions \eqref{hyp1:N}-\eqref{hyp2:N} on the drift $N$ and the additional assumption
	\begin{equation}\label{hyp3:N}
		\sup_{|v|\,\geq\,1}
		\left(
		v^2\,\omega(v)
		\,-\,
		C_0 \,N'(v)
		\right)
		\,<\,
		+\infty\,,
	\end{equation}
	for all positive constant $C_0>0\,$, supposing assumption \eqref{hyp2:psi} on the interaction kernel $\Psi$,
	consider the unique sequence of solutions $(\mu^\eps)_{\eps\,>\,0}$ to \eqref{kinetic:eq} with initial conditions satisfying assumptions \eqref{hyp:rho0}-\eqref{hyp3:f0} and
	the solution $\ols{\nu}$ to equation \eqref{bar nu:eq} with an initial condition $\ols{\nu}_0$. Furthermore, consider an exponent $\kappa$ which verifies the condition
	\begin{equation}\label{condition kappa}
		\kappa\;\in\;
		\left(\,\frac{1}{2\,b}
		\,,\,
		+\infty\,
		\right)
		\;,
	\end{equation}
	and consider a rate $\alpha_*$ lying in $\displaystyle 
	\left(\,
	0\,,\,1-(2b\kappa)^{-1}\,
	\right)\,
	$. There exists a positive constant $C$ independent of $\eps$ such that for all $\eps$ between $0$ and $1$ the following results hold true
	
	\begin{enumerate}
		\item\label{item 1 th2} consider $k$ in $\ds\{0\,,\,1\}$ and suppose that the sequence $
		\displaystyle
		\left(\nu^\eps_0\right)_{\eps\,>\,0}$ verifies
		\begin{equation}\label{hyp4:f0}
			\sup_{\eps \,>\,0}\,
			\left\|
			\,\nu^\eps_0\,
			\right\|_{\scH^{k+1}(m ^\eps)}\,
			\,<\,
			+\,\infty\,,
		\end{equation}
		and that $\ols{\nu}_0$ verifies
		\begin{equation}\label{hyp bar nu th2}
			\ols{\nu}_0\in \scH^{\,k}(\ols{m})\,.
		\end{equation}
		Then for all time $t$ in $\R^+$ it holds
		\begin{eqnarray*}
			&&\left\|\,
			\nu^\eps(t)
			\,-\,
			\nu(t)\,
			\right\|_{\scH^{k}
				\left(
				m ^\eps
				\right)}
			\\ &&\leq
			e^{Ct}
			\left(
			\|\,\ols{\nu}_0^\eps\,-\,\ols{\nu}_0\,\|_{
				\scH^k(\ols{m})} \,+\,
			C\,\left\|\,
			\nu^\eps_{0}\,
			\right\|_
			{
				\scH^{k+1}
				\left(
				m ^\eps
				\right)
			}
			\left(
			\sqrt{\eps}
			\,+\,
			\min\left\{1\,,\,
			e^{
				-\alpha_*\frac{t}{\eps}
			}
			\eps^{-\frac{\alpha_*}{2 m_*}
			}
			\right\}
			\right)
			\right)\,,
		\end{eqnarray*}
		where the asymptotic profile $\ds\nu$ is given by \eqref{limit nu};
		\item\label{item2 th21} suppose assumption \eqref{hyp4:f0} with index $k\,=\,1$ and  assumption \eqref{hyp bar nu th2} with index $k\,=\,0$, it holds for all time $t\geq0$		\begin{equation*}
			\|\,\ols{\nu}^\eps(t)\,-\,\ols{\nu}(t)\,\|_{\scH^{0}(\ols{m})}
			\,\leq\,
			e^{Ct}\,
			\left(
			\|\,\ols{\nu}_0^\eps\,-\,\ols{\nu}_0\,\|_{
				\scH^0(\ols{m})}
			\,+\,
			C\,\left\|\,
			\nu^\eps_{0}\,
			\right\|_
			{
				\scH^{2}
				\left(
				m ^\eps
				\right)
			}\eps\,\sqrt{
				\left|\,\ln{\eps}\,\right|
				\,+\,1
			}\,
			\right)\,.
		\end{equation*}
	\end{enumerate}
	In this theorem, the positive constant $C$ only depends on $\kappa$, $\alpha_*$, $m_*$, $m_p$, $\ols{m}_p$ (see assumptions \eqref{hyp:rho0}, \eqref{hyp1:f0} and \eqref{hyp2:f0}) and on the data of the problem:~$N$,~$A_0$~and~$\Psi$.
\end{theorem}
The proof of this result is provided in Section \ref{sec2} and relies on regularity estimates for the solution $\nu^\eps$ to equation \eqref{nu:eq}. These regularity estimates allow us to bound the source term which appears in the right hand side of equation \eqref{bar:nu eps:eq}.\\

We now interpret the latter theorem in terms
$\mu^\eps$. Let us emphasize that since $m^\eps$ defined by \eqref{def L 2 m} depends on $\eps$ through the spatial distribution $\rho_0^\eps$, we introduce weights which do not depend on $\eps$ anymore and which are meant to upper and lower bound $m^\eps$. We consider $
\ds
\left(\bx,\bu
\right)$ lying in 
$
\ds
K\times\R^2
$ and define
\begin{equation*}
	\left\{
	\begin{array}{l}
		\displaystyle  m^-_{\bx}(\bu)
		\,=
		\,
		\left(\,
		\rho_0(\bx)\,\kappa\,
		\right)^{-\frac{1}{2}}
		\,
		\exp{
			\left(\,
			\frac{1}{8}
			\left(\,
			\rho_0(\bx)\,|v|^2
			\,+\,
			\kappa\,
			|w|^2\,
			\right)
			\right)
		}
		\,,\\[1.1em]
		\displaystyle  \ols{m}^-(w)
		\,=
		\,\kappa^{-\frac{1}{2}}
		\,
		\exp{
			\left(\,
			\frac{\kappa}{8}\,\,
			|w|^2\,
			\right)
		}
		\,,\\[1.1em]
		\displaystyle  m^+_{\bx}(\bu)
		\,=
		\,
		\left(\,
		\rho_0(\bx)\,\kappa\,
		\right)^{-\frac{1}{2}}
		\,
		\exp{
			\left(\,
			2
			\left(\,
			\rho_0(\bx)\,|v|^2
			\,+\,
			\kappa\,
			|w|^2\,
			\right)
			\right)
		}
		\,,\\[1.1em]
		\displaystyle  \ols{m}^+(w)
		\,=
		\,\kappa^{-\frac{1}{2}}
		\,
		\exp{
			\left(\,
			2
			\,
			\kappa\,
			|w|^2\,
			\right)
		}
		\,.
	\end{array}
	\right.\\[0.8em]
\end{equation*}
With these notations, our result reads as follows
\begin{theorem}\label{th21}
	Under the assumptions of Theorem \ref{th:2} consider the unique sequence of solutions $(\mu^\eps)_{\eps\,>\,0}$ to \eqref{kinetic:eq} and the solution $\ds
	\left(\cV\,,\,\ols{\mu}
	\right)$
	to \eqref{macro:eq} with initial condition $\ds
	\left(\cV\,,\,\ols{\mu}
	\right)$ satisfying \eqref{hyp:macro}. The following results hold true
	\begin{enumerate}
		\item\label{cv mu eps H 0}
		Consider $k$ in $\ds \left\{0\,,\,1\right\}$ and suppose
		\begin{equation}\label{hypf0:mu H k + 1}
			\sup_{\eps \,>\,0}\,
			\left\|
			\,\mu^\eps_0\,
			\right\|_{\scH^{k+1}(m ^+)}\,
			\,<\,
			+\,\infty\,,
		\end{equation}
		as well as the following compatibility assumption
		\begin{equation}\label{th:2 compatibility assumption 1}
			\left\|\,
			\mathcal{U}_0
			\,-\,
			\mathcal{U}_0^\eps\,
			\right\|_{L^{\infty}(K)}
			\,+\,
			\|\,
			\rho_0
			-
			\rho_0^\eps\,
			\|_{L^{\infty}(K)}\,+\,
			\|\,\ols{\mu}_0^\eps\,-\,\ols{\mu}_0\,\|_{
				\scH^{k}(\ols{m}^+)}
			\,
			\underset{\eps \rightarrow 0}{=}
			\,
			O\left(\,\eps\,\right)
			\,.
		\end{equation}
		Moreover, suppose that there exists a constant $C$ such that
		\begin{equation}\label{continuite L 2 bar mu}
			\sup_{\eps\,>\,0}\,
			\left\|\,
			\ols{\mu}_0
			\,-\,
			\tau_{\,
				w_0
			}\,
			\bar{\mu}_0
			\,
			\right\|_{\mathcal{H}^{k}
				\left(
				\ols{m}^{+}
				\right)}
			\,
			\leq
			\,
			C\,\left|\,w_0\,\right|\,,
			\quad
			\forall\, w_0\in \R
			\,.
		\end{equation}
		Then, for all $i\in\N$ and under the constraint $\alpha_*<\min{\left\{m_*/2\,,\,1-(2b\kappa)^{-1}\right\}}$, there exists 
		$\ds
		(C_i\,,\,\eps_0)\in 
		\left(\R_+^*\right)^2
		$ such that for all $\eps$ less than $\eps_0$, it holds
		for all $t \in \R_+$,
		\[
		\|
		\left(
		v\,-\,\mathcal{V}
		\right)^{i}
		\,
		\left(
		\mu^\eps
		\,-\,
		\mu
		\right)(t)
		\,
		\|_{\mathcal{H}^k
			\left(
			m^-
			\right)}
		\,\leq\,
		C_i\,e^{C_i\,
			\left(
			t
			\,+\,
			\eps\,e^{C_i\,t}
			\right)}\,
		\left(\,
		\eps^{\frac{i}{2}\,+\,\frac{1}{4}}
		\,+\,
		e^{
			\,
			-\,\alpha_*\frac{t}{\eps}\,
		}
		\,
		\eps^{-\,\frac{1}{2}}\,
		\right)\,,  
		\]
		where the limit $\mu$ is given by
		\[
		\mu
		\,=\,
		\mathcal{M}_{
			\rho_0\,
			\left|
			\theta^\eps
			\right|^{-2}
		}
		\left(
		v\,-\,\cV
		\right)
		\otimes
		\ols{\mu}\,.
		\]
		\item\label{item 2 th 22} Suppose assumption \eqref{hypf0:mu H k + 1} with $k=1$, assumption \eqref{continuite L 2 bar mu} with $k=0$ and 
		\begin{equation*}
			\left\|\,
			\mathcal{U}_0
			\,-\,
			\mathcal{U}_0^\eps\,
			\right\|_{L^{\infty}(K)}
			\,+\,
			\|\,
			\rho_0
			-
			\rho_0^\eps\,
			\|_{L^{\infty}(K)}\,+\,
			\|\,\ols{\mu}_0^\eps\,-\,\ols{\mu}_0\,\|_{
				\scH^0(\ols{m}^+)}
			\,
			\underset{\eps \rightarrow 0}{=}
			\,
			O\left(\,\eps
			\,\sqrt{\left|\,\ln{\eps}\,\right|}
			\,\right)
			\,.
		\end{equation*}
		There exists 
		$\ds
		(C\,,\,\eps_0)\in 
		\left(\R_+^*\right)^2
		$ such that for all $\eps$ less than $\eps_0$, it holds
		\begin{equation*}
			\left.
			\|\,\ols{\mu}^\eps(t)\,-\,\ols{\mu}(t)\,\|_{\scH^{0}(\ols{m}^-)}
			\,\leq\,
			C\,
			e^{Ct}\,
			\eps\,\sqrt{
				\left|\,\ln{\eps}\,\right|
			}\,,\quad \forall\,t \in \R_+  \,
			.
			\right.\\[0,7em]
		\end{equation*}
	\end{enumerate}
\end{theorem}
This result is a straightforward consequence of
Theorem \ref{th:2} and the convergence estimates for
the macroscopic quantities given by item \eqref{cv macro q} Proposition
\ref{th:preliminary}. We postpone the proof to Section \ref{proof
	21} and make a few comments. On the one hand, we achieve pointwise in time
convergence estimates, which is  an improvement in comparison to our
result in the $L^1$ setting. This is made possible thanks to the
regularity results obtained for $\nu^\eps$, which we were not able
to obtain in the $L^1$ setting. On the other hand, we recover the optimal convergence rate for the marginal $\ols{\mu}^\eps$ of $\mu^\eps$ towards the limit $\ols{\mu}$, up to a logarithmic correction. The logarithmic correction arises due to the fact that we do not consider well prepared initial data (see Proposition \ref{estimee:bar nu bar nu eps} for more details). In the  statement \eqref{cv mu eps H 0}, we prove convergence with rate $\displaystyle
O(\,\eps^{i}\,)
$ for all $i$. This is specific to the structure of the weighted $L^2$ spaces in this result.

\subsection{Useful estimates}
Before proving our main results, we
remind here uniform estimates with respect to $\eps$,  already established in \cite{BF}, for the moments of $\mu^\eps$ and  for the relative energy given by
\begin{equation*}
	\left\{
	\begin{array}{l}
		\ds M_{q}
		\left[\,\mu^\eps\,
		\right](t,\bx)
		\,:=\,\int_{\R^2}
		|\bu|^{q}
		\,
		\mu^\eps(t,\bx,\bu)\,\dD \bu\,,
		\\[1.1em]
		\ds D_{q}
		\left[\,\mu^\eps\,
		\right](t,\bx)\,:=\, 
		\int_{\R^2} |v- \cV^\eps(t,\bx)|^{q} \,
		\mu^\eps(t,\bx,\bu)\,\dD \bu\,,
	\end{array}\right.
\end{equation*}
where $q\geq 2$. 
\begin{proposition}
	\label{th:preliminary}
	Under assumptions \eqref{hyp1:N}-\eqref{hyp2:N} on the drift $N$, \eqref{hyp2:psi} on $\Psi$, \eqref{hyp:rho0}-\eqref{hyp2:f0} on the initial conditions $\mu^\eps_0$ consider the unique solutions $\mu^\eps$ and
	$
	\left(
	\mathcal{V},\,
	\bar{\mu}
	\right)
	$ to \eqref{kinetic:eq} and \eqref{macro:eq}. Furthermore, define the initial macroscopic error as
	\[
	\mathcal{E}_{\mathrm{mac}}
	\,=\,
	\left\|\,
	\mathcal{U}_0
	\,-\,
	\mathcal{U}_0^\eps\,
	\right\|_{L^{\infty}(K)}
	\,+\,
	\|\,
	\rho_0
	-
	\rho_0^\eps\,
	\|_{L^{\infty}(K)}\,.
	\]
	There exists 
	$\ds (C\,,\,\eps_0)\, \in \,
	\left(
	\R^+_*
	\right)^2
	$
	such that
	\begin{enumerate}
		\item\label{cv macro q} for all $\eps \leq \eps_0$, it holds
		\[
		\left\|\,
		\mathcal{U}(t)
		\,-\,
		\mathcal{U}^\eps(t)\,
		\right\|_{L^{\infty}(K)}
		\,\leq\,
		C\,
		\min
		{
			\left(\,
			e^{C\,t}
			\left(\,
			\mathcal{E}_{\mathrm{mac}}\,
			+\,
			\eps\,
			\right)\,,\,
			1\,
			\right)}\,,
		\quad\quad
		\forall\,t \in \R^+\,,
		\]
		where $\cU^\eps$ and $\cU$ are respectively given by \eqref{macro:q} and \eqref{macro:eq}.\\
		\item\label{estimate moment mu} For all $\eps >0$ and all $q$ in 
		$
		\ds
		[2,\,2p]$ it holds
		\[
		M_{q}[\,\mu^\eps\,](t,\bx)
		\,\,\leq\,\,
		C\,, \quad\quad\forall \,(t,\bx)\, \in\,\R^+\times K,
		\]
		where exponent $p$ is given in assumption \eqref{hyp2:N}. In particular, $\cU^\eps$ is uniformly bounded with respect to both 
		$\ds(t\,,\,\bx) \in \R^+\times K$ and
		$\eps$.\\
		\item\label{estimate rel energy mu} For all $\eps >0$ and all $q$ in 
		$
		\ds
		[\,2,\,2p\,]$ it holds
		\[
		D_q[\,\mu^\eps\,](t,\bx)
		\,\,\leq\,\,
		C \,\left[\,\exp \left(-q\,m_*\,\frac{t}{\eps}\, \right)
		\,+\,
		\eps^{\frac{q}{2}}\,\right], \quad\forall (t,\bx)\in\R^+\times K\,.
		\]
		\item\label{estimate error} For all $\eps >0$ we have
		\[
		\left|\,\mathcal{E}(\,\mu^\eps
		\left(t\,,\,\bx\,,\,\cdot\,
		\right)\,)\,\right| 
		\,\leq\, 
		C\,\left[\,\exp \left(-2\,m_*\,\frac{t}{\eps}\, \right)
		\,+\,
		\eps\,\right]\,,\quad\forall\,
		(t\,,\,\bx) \in \R^+ \times\,K\,,
		\]
		where $\ds \mathcal{E}$ is defined by \eqref{error}.
	\end{enumerate}
\end{proposition}
The proof of this result can be found in \cite{BF}. More precisely,
we refer to \cite[Proposition 4.4]{BF}  for the proof of \eqref{cv
	macro q},  \cite[Proposition 3.1]{BF}   for the proof of
\eqref{estimate moment mu} , \cite[Proposition 3.3]{BF}   for the
proof of \eqref{estimate rel energy mu} and  \cite[Proposition
3.5]{BF}  for the proof of \eqref{estimate error}.

\section{Convergence analysis in $L^1$}\label{sec1b}
In this section, we prove Theorem \ref{th1} which ensures the
convergence of $\nu^\eps$ towards the asymptotic profile
$\cM_{\rho_0^\eps}\otimes \bar{\nu}$ in a $L^1$
setting. In order to explain our argument, we
outline the main steps of our approach on a simplified example : the diffusive limit for the kinetic Fokker-Planck equation.
We consider  the asymptotic limit $\eps\rightarrow 0$ of the following linear kinetic Fokker-Planck equation
\[
\partial_t\,f^\eps
\,+\,
\frac{1}{\eps}\,
\bv\cdot\nabla_{\bx}\,f^\eps
\,=\,
\frac{1}{\eps^2}\,
\nabla_{\bv}\cdot
\left(\,
\bv\,f^\eps\,+\,
\nabla_\bv\,f^\eps\,
\right)\,,
\]
where $(\bx,\bv)$ lie in the phase space $\R^d\times\R^d$. In this context, the challenge consists in proving that as $\eps$ vanishes, it holds
\[
f^\eps\left(t\,,\,\bx\,,\,\bv\right)
\underset{\eps\rightarrow 0}{\sim}
\mathcal{M}(\bv)\otimes\rho(t,\bx)\,,
\]
where $\rho$ is a solution to the heat equation 
\[
\partial_t\,\rho\,=\,
\Delta_{\bx}\,\rho\,,
\]
and where $\cM$ stands for the standard Maxwellian distribution over $\R^d$. Relying on a rather classical  free energy estimate, it is possible to prove that $f^\eps$ converges to the following local equilibrium of the Fokker-Planck operator
\[
\cM \otimes \rho^\eps\,,
\]
where the spatial density of particles $\rho^\eps$ is defined by
\[
\rho^\eps\,=\,
\int_{\R^d}f^\eps\,\dD\bv\,.
\]

Then, the difficulty lies in proving that the spatial density of particles $\rho^\eps$ converges to $\rho$.
The convergence analysis is made intricate by the transport operator, which keeps us from obtaining a closed equation on $\rho^\eps$
\[
\partial_t\,\rho^\eps
\,+\,
\frac{1}{\eps}\,
\nabla_{\bx}\cdot
\int_{\R^d}
\bv\,f^\eps\,\dD\,\bv\,=\,0\,.
\]
To overcome this difficulty, our strategy consists in considering the following re-scaled quantity
\[
\pi^\eps(t,\bx)\,=\,
\int_{\R^d}f^\eps
\left(t,\bx-\eps\,\bv,\bv
\right)
\,\dD\bv\,.
\]
On this simplified example, the advantage of considering $\pi^\eps$ instead of $\rho^\eps$ is straightforward as it turns out that $\pi^\eps$ is an exact solution of the limiting equation. Indeed, changing variables in the equation on $f^\eps$ and integrating with respect to $\bv$, we obtain
\[
\partial_t\,\pi^\eps\,=\,
\Delta_{\bx}\,\pi^\eps\,.
\]
Therefore, the convergence analysis comes down to proving that $\pi^\eps$ is close to $\rho^\eps$. It is possible to achieve this final step taking advantage of the following estimate
\[
\left\|\,
\rho^\epsilon\,-\,
\pi^\eps
\,\right\|_{L^1
	\left(
	\R^{d}
	\right)
}
\,\leq\,
\cA\,+\,\cB\,,
\]
where $\cA$ and $\cB$ are defined as follows
\begin{equation*}
	\left\{
	\begin{array}{l}
		\displaystyle  \cA 
		\,=\,
		\left\|
		\,
		\mathcal{M}
		\otimes
		\tau_{
			-\eps\,\bv
		}\,
		\rho^\eps
		\,-\,
		\tau_{
			-\eps\,\bv
		}\,
		f^\eps
		\,
		\right\|_{L^1
			\left(
			\R^{2d}
			\right)
		}
		\,
		,\\[1.1em]
		\displaystyle  \cB \,=\,
		\int_{\R^{d}}
		\mathcal{M}(\Tilde{\bv})\,
		\left\|
		\,
		f^\eps
		\,-\,
		\tau_{
			-\eps\,\Tilde{\bv}
		}\,
		f^\eps
		\,
		\right\|_{
			L^1
			\left(
			\R^{2d}
			\right)
		}\,\dD \Tilde{\bv}
		\,,
	\end{array}
	\right.
\end{equation*}
and where $\ds\tau_{
	\bx_0
}$ stands for the translation of vector $\bx_0$ with respect to the $\bx$-variable. To estimate $\cA$, we use the first step, which ensures that $\ds f^\eps$ is close to 
$\mathcal{M}
\otimes
\rho^\eps$. Then, to estimate $\cB$, it is sufficient to prove equicontinuity estimates for $f^\eps$, that is
\[
\left\|
\,
f^\eps
\,-\,
\tau_{
	\bx_0
}\,
f^\eps
\,
\right\|_{
	L^1
	\left(
	\R^{2d}
	\right)
}
\,\lesssim\,
\left|\,
\bx_0\,
\right|\,.
\]
In the forthcoming analysis, we adapt this argument in our context.

\subsection{\textit{A priori} estimates}

The main object of this section consists in deriving equicontinuity estimates for the sequence of solutions 
$\displaystyle
\left(
\nu^\eps
\right)_{\eps\,>\,0}
$ to equation \eqref{nu:eq}. To obtain this result, we make use of the following key result
\begin{lemma}\label{abstract lemma rel ent}
	Consider 
	$\delta$ in 
	$\ds
	\{
	0\,,\,1
	\}
	$
	and smooth solutions $f$ and $g$ to the following equations
	\\[-0.3em]
	\begin{equation*}
		\left\{
		\begin{array}{ll}
			\displaystyle  \partial_t\, f
			\,+\,
			\mathrm{div}_{\by}
			\left[
			\,\mathbf{a}
			\left(
			t\,,\,\by\,,\,\xi
			\right)
			\,f\, \right]
			\,+\,
			\lambda(t)\,
			\mathrm{div}_{\xi}
			\left[
			\,
			\left(\mathbf{b}_1
			\,+\,
			\mathbf{b}_3
			\right)
			\left(
			t\,,\,\by\,,\,\xi
			\right)
			\,f\, \right]
			
			\,=\,
			\lambda(t)^2\,
			\Delta_{\xi}\,
			f\,,\\[0.7em]
			\displaystyle  
			
			\partial_t\, g
			\,+\,
			\mathrm{div}_{\by}
			\left[
			\,\mathbf{a}
			\left(
			t\,,\,\by\,,\,\xi
			\right)
			\,g\, \right]
			\,+\,
			\lambda(t)\,
			\mathrm{div}_{\xi}
			\left[
			\,\mathbf{b}_2
			\left(
			t\,,\,\by\,,\,\xi
			\right)
			\,g\, \right]
			\,=\,
			\delta\,\lambda(t)^2\,
			\Delta_{\xi}\,
			g\,.
		\end{array}
		\right.\\[0.6em]
	\end{equation*}
	set on the phase space  
	$
	\ds
	\left(t,
	\by,\xi
	\right)
	\in
	\R^+\times
	\R^{d_1}\times\R^{d_2}
	$, with 
	$\ds d_1 \,\geq\, 0$ and
	$\ds d_2 \,\geq\, 1$, where 
	\[ 
	\left(
	\mathbf{a}
	\,:\,\R_+\times\R^{d_1}\times\R^{d_2}\,\rightarrow\,\R^{d_1}
	\right)\;\;
	\mathrm{and}\;\;
	\left(
	\mathbf{b}_{i}
	\,:\,\R_+\times\R^{d_1}\times\R^{d_2}\,\rightarrow\,\R^{d_2}
	\right)\,,
	\quad
	i\in\{1\,,\,2\,,\,3\}\,,
	\]
	are given vector fields and where $\lambda$ is a positive valued function. Suppose that $f$ and $g$ have positive values and are normalized as follows
	\[
	\int_{\R^{d_1+d_2}} f\,\dD\by\,\dD \xi\,=\,\int_{\R^{d_1+d_2}} g\,\dD\by\,\dD \xi\,=\,1\,.
	\]
	Then it holds for all time $t\geq0$
	\begin{equation}\label{estimee2:abstrat lemma}
		\left\|\,
		f(t)\,-\,
		g(t)\,
		\right\|_{L^1
			\left(
			\R^{d_1+d_2}
			\right)}
		\,\leq\,
		2\,\sqrt{2}\,
		\left(
		\left\|\,
		f_0\,-\,
		g_0\,
		\right\|_{L^1
			\left(
			\R^{d_1+d_2}
			\right)}^{\frac{1}{2}}
		\,+\,
		\left(
		\int_0^t
		\cR(s)\,\dD s\,
		\right)^{\frac{1}{2}}
		\right)\,,
	\end{equation}
	where $\cR$ is defined as
	\begin{equation*}
		\cR(t)
		\,=\,
		\int_{\R^{d_1+d_2}}
		\left(\,
		\frac{1}{4}\,
		\left|\,
		\mathbf{b}_1\,
		\,-\,
		\mathbf{b}_2\,
		\right|^2 f
		\,+\,
		\lambda
		\,
		\left|\,
		\mathrm{div}_{\xi}
		\left[\,
		\mathbf{b}_3
		\,g
		\,+\,
		\left(\delta
		\,-\,1\right)\,\lambda\,
		\nabla_{\xi}\,
		g\,
		\,\right]\,\right|
		\right)(t\,,\,\by\,,\,\xi)\,
		\dD\by\,\dD\xi
		\,.
	\end{equation*}
\end{lemma}
We postpone the proof of this result to Appendix \ref{sec:rel ent}. Thanks to the latter lemma, we prove the following equicontinuity estimate for solutions to \eqref{nu:eq}
\begin{proposition}\label{equicontinuity}
	Consider a sequence
	$
	\left(\nu^\eps
	\right)_{\eps\,>\,0}
	$
	of smooth solutions to equation \eqref{nu:eq} whose initial conditions meet assumption \eqref{hyp1 nu L 1}. There exists a positive constant $C$ independent of $\eps$ such that for all $\eps\,>\,0$, it holds
	\[
	\left\|\,
	\nu^\eps(t,\bx)\,-\,
	\tau_{w_0}\,\nu^\eps(t,\bx)\,
	\right\|_{
		L^1
		\left(
		\R^2
		\right)
	}
	\,\leq\,
	C
	\left(\,
	\left|\,e^{b\,t}\,w_0\,\right|
	\,+\,
	\left|\,e^{b\,t}\,w_0\,\right|^{\frac{1}{2}}\,
	\right)\,,\quad\quad
	\forall\,
	\left(t,\bx,w_0\right)
	\in \R^+\times K\times\R
	\,,
	\]
	where  $C$ is explicitly given by
	\[
	C
	\,=\,
	\sqrt{
		\max{
			\left(
			8\,m_1\,,\,
			1\,/\,b
			\right)
		}
	}\,,
	\]
	with $m_1$ defined in assumption \eqref{hyp1 nu L 1}.
\end{proposition}

\begin{proof}
	We fix some $\bx$ in $K$, some positive $\eps$ and consider some $w_0$ in $\R$. Then we  define a re-scaled version $f$ of $\nu^\eps$
	\[
	f
	\left(
	t\,,\,w\,,\,v
	\right)
	\,=\,
	e^{-b\,t}\,
	\nu^{\eps}
	\left(
	t\,,\,\bx\,,\,v\,,\,e^{-b\,t}\,w
	\right)\,.
	\]
	We compute the equation solved by $f$ performing the change of variable
	\begin{equation}\label{change var h}
		w\,\mapsto\, e^{-b\,t}\,w
	\end{equation}
	in equation \eqref{nu:eq}, this yields
	\begin{equation*}
		\partial_t\, f
		\,+\,
		\partial_{w}
		\left[
		e^{b\,t}\,
		A_0
		\left(
		\theta^\eps\,v\,,\,0
		\right)
		f\,
		\right]
		\,+\,
		\frac{1}{
			\theta^\eps}\,
		\partial_{v}
		\left[
		\,B^\eps_0
		\left(
		t\,,\,\bx\,,\,\theta^\eps\,v\,,\,
		e^{-b\,t}\,w
		\right)
		f\,
		\right]
		\,=\,
		\frac{1}{
			\left|
			\theta^\eps
			\right|^2}\,
		\cF_{\rho_0^\eps}
		\left[\,
		f\,
		\right]\,
		,
	\end{equation*}
	where $A_0$ and $B^\eps_0$ are given by \eqref{def:b0}. Then, we define  
	$\ds
	g\,:=\,
	\tau_{w_0}\,f$, which solves the following equation
	\begin{equation*}
		\partial_t\, g
		\,+\,
		\partial_{w}
		\left[
		e^{b\,t}\,
		A_0
		\left(
		\theta^\eps\,v\,,\,0
		\right)
		g\,
		\right]
		\,+\,
		\frac{1}{
			\theta^\eps}\,
		\partial_{v}
		\left[
		\,B^\eps_0
		\left(
		t\,,\,\bx\,,\,\theta^\eps\,v\,,\,
		e^{-b\,t}\,(w+w_0)
		\right)
		g\,
		\right]
		\,=\,
		\frac{1}{
			\left|
			\theta^\eps
			\right|^2}\,
		\cF_{\rho_0^\eps}
		\left[\,
		g\,
		\right]\,
		.
	\end{equation*}
	Thanks to the change of variable \eqref{change var h},  coefficients inside the $w$-derivatives in the equations on $f$ and $g$ are the same. Hence, we can apply Lemma \ref{abstract lemma rel ent} to $f$ and $g$ with the following parameters 
	\begin{equation*}
		\left\{
		\begin{array}{l}
			\displaystyle  \left(\,\delta\,,\,\lambda\,
			\right)\,=\,\ds
			\left(\,1\,,\,
			1\,/\,
			\theta^\eps
			\,
			\right)\,,\\[0.7em]
			\displaystyle  
			\mathbf{a}\,(t\,,\,w\,,\,v)
			\,=\,\ds
			e^{b\,t}\,
			A_0
			\left(
			\theta^\eps\,v\,,\,0
			\right)
			\,,\\[0.7em]
			\displaystyle  
			\mathbf{b}_1\,(t\,,\,w\,,\,v)
			\,=\,\ds
			\,B^\eps_0
			\left(
			t\,,\,\bx\,,\,\theta^\eps\,v\,,\,
			e^{-b\,t}\,w
			\right)
			\,-\,
			\frac{1}{
				\theta^\eps}
			\,\rho_0^\eps(\bx)\,v
			\,,\\[0.8em]
			\displaystyle  \left(\,\mathbf{b}_2\,,\,\mathbf{b}_3\,
			\right)\,=\,\ds
			\left(\,\tau_{w_0}\,\mathbf{b}_1\,,\,0\,
			\right)\,.
		\end{array}
		\right.\\[0,3em]
	\end{equation*}
	According to \eqref{estimee2:abstrat lemma} in Lemma \ref{abstract lemma rel ent}, it holds
	\begin{equation*}
		\left\|\,
		f(t)\,-\,
		g(t)\,
		\right\|_{L^1
			\left(
			\R^{2}
			\right)}
		\,\leq\,
		2\,\sqrt{2}\,
		\left\|\,
		f_0\,-\,
		g_0\,
		\right\|_{L^1
			\left(
			\R^{2}
			\right)}^{\frac{1}{2}}
		\,+\,
		\left(
		\frac{1
			\,-\,
			e^{-2\,b\,t}
		}{
			b
		}\,
		\right)^{\frac{1}{2}}
		\left|
		w_0
		\right|
		\,,
		\quad
		\forall\,t\in\R^+\,.
	\end{equation*}
	Therefore, according to assumption \eqref{hyp1 nu L 1}, we obtain the result after inverting the change of variable \eqref{change var h} and taking the supremum over all $\bx$ in $K$.\\
\end{proof}

We conclude this section providing regularity estimates for the limiting distribution $\ols{\nu}$ with respect to the adaptation variable, which solves \eqref{bar nu:eq}. 
The proof for this result is mainly computational since we have an explicit formula for the solutions to equation \eqref{bar nu:eq}.
\begin{proposition}\label{lemme:bar nu L 1}
	Consider some $\ols{\nu}_0$ satisfying assumption \eqref{hyp bar nu 0 L 1}.
	The solution $\ols{\nu}$ to equation \eqref{bar nu:eq} with initial condition 
	$
	\ds\ols{\nu}_0
	$ verifies
	\[
	\left\|\,
	\ols{\nu}(t)\,
	\right\|_
	{
		L^{\infty}
		\left(
		K\,,\,
		W^{2\,,\,1}
		\left(
		\R
		\right)
		\right)
	}
	\,\leq\,
	\exp{
		\left(
		2
		\,b\,t
		\right)}\,
	\left\|\,
	\ols{\nu}_0\,
	\right\|_
	{
		L^{\infty}
		\left(
		K\,,\,
		W^{2\,,\,1}
		\left(
		\R
		\right)
		\right)
	}\,,
	\quad\forall \,t \in\R^+\,,
	\]
	and
	\[
	\left\|\,
	w\,
	\partial_w\,
	\ols{\nu}(t)\,
	\right\|_
	{
		L^{\infty}
		\left(
		K\,,\,
		L^1
		\left(
		\R
		\right)
		\right)
	}
	\,=\,
	\left\|\,
	w\,
	\partial_w\,
	\ols{\nu}_0\,
	\right\|_
	{
		L^{\infty}
		\left(
		K\,,\,
		L^1
		\left(
		\R
		\right)
		\right)
	}\,,
	\quad\forall \,t \in\R^+\,.
	\]
\end{proposition}
\begin{proof}
	Since $\ols{\nu}$ solves \eqref{bar nu:eq}, it is given by the following formula
	\[
	\ols{\nu}(t,\,\bx\,,\,w)
	\,\,=\,\,
	e^{bt}\,
	\ols{\nu}_0
	\left(\bx\,,\,e^{bt}\,w
	\right)\,
	, 
	\quad\forall \,(t,\bx)\, \in\,\R^+\times K\,.
	\]
	Consequently, we easily obtain the expected result.
\end{proof}

We are now ready to prove the first convergence result on $\nu^\eps$.
\subsection{Proof of Theorem
	\ref{th1}
}
The proof is divided in three steps. First, we prove that the solution $\nu^\eps$ to \eqref{nu:eq} converges to the local equilibrium 
\[
\cM_{\rho_0^\eps}\otimes \bar{\nu}^\eps\,,
\]
thanks to a free energy estimate. Then,
as in the example developed at the beginning of Section \ref{sec1b}, we introduce an intermediate
quantity $\ols{g}^\eps$, which converges to the
solution $\ols{\nu}$ to equation \eqref{bar nu:eq}. At last, we prove
that $\ols{g}^\eps$ is close to $\ols{\nu}^\eps$ thanks to the
equicontinuity estimate given in Proposition \ref{equicontinuity} and therefore conclude that the marginal $\ols{\nu}^\eps$ converges towards $\ols{\nu}$.

\subsubsection{Convergence of $\nu^\eps$ towards
	$\,
	\cM_{\rho_0^\eps}\otimes \bar{\nu}^\eps\,
	$:
	free energy estimate
}

\noindent In this section, we investigate the time evolution of the free energy along the trajectories of equation \eqref{nu:eq}. It is defined for all $(t,\bx) \in\R^+\times K$ as
\begin{equation*}
	E\left[\,\nu^\eps(t,\bx)\, \right]\,=\,\int_{\R^2}
	\nu^\eps(t,\bx,\bu)
	\,\ln{
		\left(
		\frac{
			\nu^\eps(t,\bx,\bu)
		}{
			\mathcal{M}_{\rho_0^\eps(\bx)}(v)
		}
		\right)
	}
	\,\dD\bu\,.
\end{equation*}
More precisely, our interest lies in its decay rate, which is given by the Fisher information 
\begin{equation*}
	I\left[\,\nu^\eps(t,\bx)\,|\,\mathcal{M}_{\rho_0^\eps(\bx)}\,\right] \,:=\, \int_{\R^2}\left|\partial_v\ln{\left(\frac{\nu^\eps(t,\bx,\bu)}
		{\mathcal{M}_{\rho_0^\eps(\bx)}(v)}\right)}\right|^2\nu^\eps(t,\bx,\bu)\,\dD\bu\,.
\end{equation*}
The reason for our interest is that the latter quantity controls the following relative entropy 
\[
H
\left[\,
\nu^\eps(t,\bx)\,|\,
\mathcal{M}_{\rho_0^\eps}(\bx)
\otimes \ols{\nu}^\eps(t,\bx)\,
\right]
\,=\,
\int_{\R^2}
\nu^\eps(t,\bx,\bu)
\,\ln{
	\left(
	\frac{
		\nu^\eps(t,\bx,\bu)
	}{
		\cM_{\rho_0^\eps}\otimes \bar{\nu}^\eps(t,\bx,\bu)
	}
	\right)
}
\,\dD\bu\,, 
\]
which itself controls the $L^1$-distance between $\nu^\eps$ and $
\cM_{\rho_0^\eps}\otimes \bar{\nu}^\eps
$. This allows to deduce the following result
\begin{proposition}
	\label{estimee rel ent nu M}
	Under assumptions \eqref{hyp1:N}-\eqref{hyp2:N} on the drift $N$ and \eqref{hyp2:psi} on the interaction kernel $\Psi$, consider a sequence of solutions $(\mu^\eps)_{\eps>0}$ to \eqref{kinetic:eq} with initial conditions satisfying assumptions \eqref{hyp:rho0}-\eqref{hyp2:f0}
	and \eqref{hyp2 nu L
		1}. Then for all $\eps\leq 1\,$, it holds
	\[
	\left\|\,
	\nu^\eps\,-\,
	\mathcal{M}_{\rho_0^\eps}
	\otimes
	\ols{\nu}^\eps
	\,
	\right\|_{
		L^{\infty}\left(K\,,\,L^1\left([\,0,\,t\,]\times \R^2\right)\right)}
	\,\leq\,
	\sqrt{\eps}\left(
	2\,m_2\,\sqrt{t}\,
	\,+\,
	C\,(t+1)\right)\,,
	\quad
	\forall\,
	t\,\geq\,0\,,
	\]
	where $m_2$ is given in assumption \eqref{hyp2 nu L 1}. In this result, the constant $C$ only depends on $m_*$,~$m_p$ and $\ols{m}_p$ (see assumptions \eqref{hyp:rho0}-\eqref{hyp2:f0}) and the data of the problem $N$,~$\Psi$~and~$A_0$.
\end{proposition}
\begin{proof}
	All along this proof, we choose some 
	$\ds \bx$ lying in $K$ and we omit the dependence with respect to 
	$\ds
	\left(
	t\,,\,\bx
	\right)
	$ when the context is clear. We compute the time derivative of $\ds
	E
	\left[\,
	\nu^\eps\,
	\right]
	$ multiplying equation \eqref{nu:eq} by
	$
	\ds
	\ln{
		\left(
		\nu^\eps\,/\,
		\mathcal{M}_{\rho_0^\eps}
		\right)
	}
	$. After integrating by part the stiffer term, it yields
	\[
	\frac{\dD}{\dD t}\,
	E
	\left[\,
	\nu^\eps\,
	\right]
	\,+\,
	\frac{1}{\left|
		\theta^\eps
		\right|^2}\,
	I
	\left[\,
	\nu^\eps\,|\,
	\mathcal{M}_{\rho_0^\eps}\,
	\right]
	\,=\,
	\cA\,,
	\]
	where $\cA$ is given by
	\[
	\cA\,=\,
	-\,\int_{\R^2}
	\mathrm{div}_{\bu}
	\left[\,
	\mathbf{b}^\eps_0\,
	\nu^\eps\,
	\right]
	\,
	\ln{
		\left(
		\frac{\nu^\eps}{
			\mathcal{M}_{\rho_0^\eps}
		}
		\right)
	}\,\dD \bu\,.
	\]
	After an integration by part, $\cA$ rewrites as follows
	\[
	\cA
	\,=\,
	\frac{1}{\theta^\eps}\,
	\int_{\R^2}
	B^\eps_0
	\left(\,t\,,\,\bx\,,\,
	\theta^\eps\, v,\, w\,
	\right)
	\;
	\partial_{v}
	\left[\,
	\ln{
		\left(
		\frac{\nu^\eps}{
			\mathcal{M}_{\rho_0^\eps}
		}
		\right)
	}\,
	\right]\,
	\nu^\eps\,
	\,\dD \bu
	\,+\,b
	\,,
	\]
	where $B^\eps_0$ is given by \eqref{def:b0}. According to items
	\eqref{estimate moment mu} and \eqref{estimate error} in Proposition \ref{th:preliminary} , $\mathcal{V}^\eps$ and $\mathcal{E}(\mu^\eps)$ are uniformly bounded with respect to both $(t,\bx) \in \R^+\times K$ and $\eps\,>\,0$. Furthermore, according to assumptions \eqref{hyp2:psi} and \eqref{hyp:rho0} on $\Psi$ and $\rho_0^\eps$, $\Psi*_r\rho^\eps_0$ is uniformly bounded with respect to both $\bx \in K$ and $\eps\,>\,0$. Consequently, applying Young's inequality, assumption \eqref{hyp2:N} and since $N$ is locally Lipschitz, we obtain
	\[
	\cA
	\,\leq\,
	\frac{1}{2\,\left|
		\theta^\eps
		\right|^2}\,
	I
	\left[\,
	\nu^\eps\,|\,
	\mathcal{M}_{\rho_0^\eps}\,
	\right]
	\,+\,
	C
	\left(
	1
	\,+\,
	\int_{\R^2}
	\left(\,
	\left|\,
	\theta^\eps\,v\,
	\right|^{2p}
	\,+\,
	w^2\,
	\right)
	\,
	\nu^\eps\,
	\,\dD \bu
	\right),
	\]
	for some positive constant $C$ only depending on $m_*$, $m_p$, $\ols{m}_p$ and the data of the problem: $N$, $A$ and $\Psi$. Then we invert the change of variable \eqref{change:var} in the integral in the right-hand side of the latter inequality and apply item \eqref{estimate moment mu} in Proposition \ref{th:preliminary}. In the end, it yields
	\[
	\cA
	\,\leq\,
	\frac{1}{2\,\left|
		\theta^\eps
		\right|^2}\,
	I
	\left[\,
	\nu^\eps\,|\,
	\mathcal{M}_{\rho_0^\eps}\,
	\right]
	\,+\,
	C\,.
	\]
	Consequently, we end up with the following differential inequality
	\[
	\frac{\dD}{\dD t}\,
	E
	\left[\,
	\nu^\eps\,
	\right]
	\,+\,
	\frac{1}{2\,\left|
		\theta^\eps
		\right|^2}\,
	I
	\left[\,
	\nu^\eps\,|\,
	\mathcal{M}_{\rho_0^\eps}\,
	\right]
	\,\leq\,
	C\,,
	\]
	Then we substitute the Fisher information with the relative entropy in the latter inequality according to the Gaussian logarithmic Sobolev inequality, which reads as follows (see \cite{Fathi/Indrei/Ledoux})
	\[
	2\,
	H
	\left[\,
	\nu^\eps(t,\bx)\,|\,
	\mathcal{M}_{\rho_0^\eps}
	\otimes \ols{\nu}^\eps(t,\bx)\,
	\right]
	\,\leq\,
	I
	\left[\,
	\nu^\eps(t,\bx)\,|\,
	\mathcal{M}_{\rho_0^\eps(\bx)}\,
	\right]\,,
	\]
	and we integrate between $0$ and $t$ to get
	\[
	\int_0^t
	\frac{1}{
		\left|
		\theta^\eps(s)
		\right|^2
	}\,
	H
	\left[\,
	\nu^\eps(s,\bx)\,|\,
	\mathcal{M}_{\rho_0^\eps}
	\otimes \ols{\nu}^\eps(s,\bx)\,
	\right]
	\,
	\dD s
	\,\leq\,
	E
	\left[\,
	\nu^\eps_0(\bx)\,
	\right]
	-
	E
	\left[\,
	\nu^\eps(t,\bx)\,
	\right]
	\,+\,
	C\,t\,.
	\]
	In the latter inequality, we bound $-E
	\left[\,
	\nu^\eps(t,\bx)\,
	\right]$ thanks to the following estimate, obtained using Jensen's inequality
	\[
	-E
	\left[\,
	\nu^\eps(t,\bx)\,
	\right]
	\,\leq\,-
	\int_{\R^2}
	\nu^\eps(t,\bx,\bu)\,\ln{\left(\cM_{1}(\cW^\eps+w)\right)}\,\dD \bu
	\,.
	\]
	In the right hand side of the latter inequality, we replace $\nu^\eps$ with $\mu^\eps$ according to \eqref{def:nueps} and invert the change of variable \eqref{change:var}, this yields
	\[
	-E
	\left[\,
	\nu^\eps(t,\bx)\,
	\right]
	\,\leq\,
	\frac{1}{2}
	\int_{\R^2}
	\mu^\eps(t,\bx,\bu)\,(\ln{(2\pi)}+|w|^2)\,\dD \bu
	\,,
	\]
	which, after applying item \eqref{estimate moment mu} in Proposition \ref{th:preliminary} to estimate the latter right hand side, ensures
	\[
	\int_0^t
	\frac{1}{
		\left|
		\theta^\eps(s)
		\right|^2
	}\,
	H
	\left[\,
	\nu^\eps(s,\bx)\,|\,
	\mathcal{M}_{\rho_0^\eps}
	\otimes \ols{\nu}^\eps(s,\bx)\,
	\right]
	\,
	\dD s
	\,\leq\,
	E
	\left[\,
	\nu^\eps_0(\bx)\,
	\right]
	\,+\,
	C\,(t+1)\,.
	\]
	To estimate $E
	\left[\,
	\nu^\eps_0(\bx)\,
	\right]$ in the latter inequality, we replace $\nu^\eps_0$ with $\mu^\eps_0$ according to \eqref{def:nueps} and invert the change of variable \eqref{change:var} at time $t=0$
	\[
	E
	\left[\,
	\nu^\eps_0(\bx)\,
	\right]
	=
	H
	\left[\,
	\nu^\eps_0(\bx)\,
	\right]
	+\frac{1}{2}
	\int_{\R^2}
	\mu^\eps_0(\bx,\bu)\,(\rho_0^\eps(\bx)|\cV^\eps_0(\bx)-v|^2+\ln{(2\pi)}-\ln{(\rho_0^\eps(\bx))})\,\dD \bu\,.
	\]
	Then, we bound $H
	\left[\,
	\nu^\eps_0(\bx)\,
	\right]$, $\rho_0^\eps(\bx)$
	and moments of $\mu^\eps_0$ thanks to assumptions \eqref{hyp2 nu L 1}, \eqref{hyp:rho0} and \eqref{hyp1:f0} respectively, which yields
	\[
	\int_0^t
	\frac{1}{
		\left|
		\theta^\eps(s)
		\right|^2
	}\,
	H
	\left[\,
	\nu^\eps(s,\bx)\,|\,
	\mathcal{M}_{\rho_0^\eps}
	\otimes \ols{\nu}^\eps(s,\bx)\,
	\right]
	\,
	\dD s
	\,\leq\,
	m_2^2
	\,+\,
	C\,(t+1)\,.
	\]
	To estimate the left hand side in the latter relation, we use the explicit formula \eqref{expression for theta eps} for $\theta^\eps$, which ensures that as long as $\eps$ is less than $1$ and $s$ is greater than 
	$
	\ds
	T^\eps
	$, where $T^\eps$ is given by
	\[
	T^\eps
	\,=\,
	\frac{\eps}{2\,m_*}\,
	\left|\,
	\ln{
		\left(
		\eps
		\right)}\,
	\right|
	\,,
	\]
	we have
	\[
	\frac{1}{2\,\eps}
	\,\leq\,
	\frac{1}{\left|
		\theta^\eps(s)
		\right|^{2}}
	\,.
	\]
	Consequently, we obtain 
	\[
	\int_{T^\eps}^t
	H
	\left[\,
	\nu^\eps(s
	,\bx)\,|\,
	\mathcal{M}_{\rho_0^\eps}
	\otimes
	\ols{\nu}^\eps(s
	,\bx)
	\,
	\right]
	\,
	\dD s
	\,\leq\,
	2\,\eps\,
	m_2^2
	\,+\,
	C\,\eps\,(t\,+\,1)\,.
	\]
	Then, we substitute the relative entropy with the $L^1$-norm according to Csizár-Kullback inequality
	\[
	\left\|\,
	\nu^\eps(s,\bx)\,-\,
	\mathcal{M}_{\rho_0^\eps}
	\otimes \ols{\nu}^\eps(s,\bx)\,
	\right\|^2_{L^1
		\left(
		\R^2
		\right)
	}
	\,\leq\,
	2\,
	H
	\left[\,
	\nu^\eps(s,\bx)\,|\,
	\mathcal{M}_{\rho_0^\eps}
	\otimes \ols{\nu}^\eps(s,\bx)\,
	\right]\,,
	\]
	and take the supremum over all $\bx$ in $K$. After taking the square root, it yields
	\[
	\sup_{\bx \in K}
	\int_{T^\eps}^t
	\left\|\,
	\nu^\eps\,-\,
	\mathcal{M}_{\rho_0^\eps}
	\otimes
	\ols{\nu}^\eps
	\,
	\right\|_{L^1\left(\R^2\right)}(s,\bx)
	\,
	\dD s
	\,\leq\,
	2\,\sqrt{\eps\,t}\,
	m_2
	\,+\,
	C\,\sqrt{\eps}\,(t+1)\,,
	\quad
	\forall\,
	t\,\geq\,0\,,
	\]
	To conclude,
	we notice that since equation \eqref{nu:eq} is conservative, it holds
	\[
	\sup_{\bx \in K}
	\int^{T^\eps}_0
	\left\|\,
	\nu^\eps\,-\,
	\mathcal{M}_{\rho_0^\eps}
	\otimes
	\ols{\nu}^\eps
	\,
	\right\|_{L^1\left(\R^2\right)}(s,\bx)
	\,
	\dD s
	\,\leq\,
	2\,T^\eps
	\,\leq\,
	C\,\sqrt{\eps}\,.
	\]
	We sum the last two estimates to obtain the result.
\end{proof}

\subsubsection{Convergence of
	$\ols{g}^\eps$ towards
	$\, \ols{\nu}$}\label{sec:cv bar g eps to bar nu}
As in the example developed at the beginning of this section, we consider the following re-scaled version $g^\eps$ of $\nu^\eps$
\[
\nu^\eps
\left(
t\,,\,\bx\,,\,v\,,\,w
\right)
\,=\,
g^\eps
\left(
t\,,\,\bx\,,\,v\,,\,w\,+\,
\gamma^{\eps}(t,\bx)\,v
\right)\,,
\]
where $\gamma^\eps$ is given by
\[
\gamma^\eps
\left(t\,,\,\bx\right)
\,=\,
\frac{a\,\eps}{\rho_0^\eps
	(\bx)}\,\theta^\eps
\left(t\,,\,\bx\right)\,.
\]
Operating the following change of variable in equation \eqref{nu:eq}
\begin{equation}\label{change:var g}
	(t\,,\,v\,,\,w)\mapsto \left(t\,,\,v\,,\,
	w\,+\,
	\gamma^\eps\,v
	\right)\,,
\end{equation}
and integrating the equation with respect to $v$, the equation on the marginal 
$\displaystyle\ols{g}^{\eps}$ of 
$\displaystyle g^{\eps}$ defined as
\[
\ols{g}^{\eps}
\left(
t\,,\,\bx\,,\,w
\right)
\,=\,
\int_{\R}
g^\eps
\left(
t\,,\,\bx\,,\,v\,,\,w
\right)
\,\dD v\,,
\]
reads as follows
\begin{equation}\label{bar g eps:eq}
	\ds
	\partial_t\, \ols{g}^\eps
	\,+\,
	\frac{a\,\eps}{\rho_0^\eps}
	\,
	\partial_{w}
	\left[\,
	\int_{\R}
	\left(
	B^\eps_0
	\left(
	t,\,\bx,\,\theta^\eps\,v,\,
	w-\gamma^\eps\,v
	\right)
	+b\,\theta^\eps\,v
	\right)
	g^\eps\,\dD v\,
	\right]
	\,-\,
	\left(
	\frac{a\,\eps}{\rho_0^\eps}
	\right)^2
	\,
	\partial_{w}^{\,2}
	\,\ols{g}^{\eps}
	\,=\,
	\partial_{w}
	\left[\,
	b\,w\,\ols{g}^{\eps}\,
	\right]
	\,,
\end{equation}
where $\ds B^\eps_0$ is defined by \eqref{def:b0}. As in our example, the equation on $\ols{g}^\eps$ is consistent with the limiting equation \eqref{bar nu:eq} as $\eps$ vanishes, this enables to prove that $\ols{g}^\eps$ converges towards $\ols{\nu}$
\begin{proposition}\label{cv: bar g eps/bar nu}
	Under assumptions \eqref{hyp1:N}-\eqref{hyp2:N} on the drift $N$ and \eqref{hyp2:psi} on the interaction kernel $\Psi$, consider a sequence of solutions $(\mu^\eps)_{\eps\,>\,0}$ to \eqref{kinetic:eq} with initial conditions satisfying assumption \eqref{hyp:rho0}-\eqref{hyp2:f0} and the solution $\ols{\nu}$ to equation \eqref{bar nu:eq} with initial condition 
	$
	\ds\ols{\nu}_0
	$ satisfying assumption \eqref{hyp bar nu 0 L 1}. There exists a positive constant $C$ independent of $\eps$ such that for all $\eps$ less than $1$, it holds
	\[
	\left\|\,
	\ols{g}^\eps(t)
	\,-\,
	\ols{\nu}(t)
	\right\|_{L^{\infty}_{\bx}
		L^{1}_{w}}
	\,\leq\,
	2\,\sqrt{2}\,
	\left\|\,
	\ols{g}^\eps_0
	\,-\,
	\ols{\nu}_0
	\right\|^{\frac{1}{2}}_{L^{\infty}_{\bx}
		L^{1}_{w}}
	\,+\,
	C\,e^{b\,t}\,\sqrt{\eps}
	\,,\quad
	\forall\,t\in \R^+\,.
	\]
	In this result, the constant $C$ only depends on $m_*$,~$m_p$~and~$\ols{m}_p$ (see assumptions \eqref{hyp:rho0}-\eqref{hyp2:f0}) and the data of the problem $\ols{\nu}_0$, $N$, $\Psi$ and $A_0$.
\end{proposition}
\begin{proof}
	All along this proof, we choose some 
	$\ds \bx$ lying in $K$ and some positive $\eps$; we omit the dependence with respect to 
	$\ds
	\left(
	t,\bx
	\right)
	$ when the context is clear. Since $\ols{\nu}$ and $\ols{g}^\eps$ solve respectively equations \eqref{bar nu:eq} and \eqref{bar g eps:eq}, Lemma \ref{abstract lemma rel ent} applies with the following parameters 
	\begin{equation*}
		\left\{
		\begin{array}{l}
			\displaystyle  \left(\,
			d_1\,,\,d_2\,,\,
			\delta\,,\,\lambda\,
			\right)\,=\,\ds
			\left(\,
			0\,,\,1\,,\,
			0\,,\,
			a\,\eps\,/\,\rho_0^\eps
			\,
			\right)\,,\\[0.7em]
			\displaystyle  
			\mathbf{b}_2\,(t\,,\,w)
			\,=\,\ds
			-\,
			\frac{\rho_0^\eps}{a\,\eps}
			\,b\,w
			\,,\\[0.8em]
			\displaystyle  
			\mathbf{b}_1\,(t\,,\,w)
			\,=\,\ds
			\mathbf{b}_2\,(t\,,\,w)
			\,+\,
			\int_{\R}
			\left(
			B^\eps_0
			\left(
			t,\,\bx,\,\theta^\eps\,v,\,
			-\gamma^\eps\,v
			\right)
			+b\,\theta^\eps\,v
			\right)
			\frac{g^\eps}{
				\ols{g}^\eps
			}\,\dD v\,
			\,,\\[0.8em]
			\displaystyle  \mathbf{b}_3\,(t\,,\,w)\,=\,\ds
			-\,w\,,
		\end{array}
		\right.
	\end{equation*}
	where $B^\eps_0$ is given by \eqref{def:b0}.
	According to \eqref{estimee2:abstrat lemma} in Lemma \ref{abstract lemma rel ent}, it holds
	\begin{equation*}
		\left\|\,
		\ols{g}^\eps(t)\,-\,
		\ols{\nu}(t)\,
		\right\|_{L^1
			\left(
			\R
			\right)}
		\,\leq\,
		2\,\sqrt{2}\,
		\left(
		\left\|\,
		\ols{g}^\eps_0\,-\,
		\ols{\nu}_0\,
		\right\|_{L^1
			\left(
			\R
			\right)}^{\frac{1}{2}}
		\,+\,
		\left(
		\int_0^t
		\cR_1(s)
		\,+\,
		\cR_2(s)
		\,\dD s\,
		\right)^{\frac{1}{2}}
		\right)
		\,,
		\quad
		\forall\,t\in\R^+\,,
	\end{equation*}
	where $\cR_1$ and $\cR_2$ are given by
	\begin{equation*}
		\left\{
		\begin{array}{l}
			\displaystyle  \cR_1(t)
			\,=\,
			\frac{1}{4}\int_{\R}
			\left|\,
			\int_{\R}
			\left(
			B^\eps_0
			\left(
			t\,,\,\bx\,,\,\theta^{\eps}\,v\,,\,-\,\gamma^{\eps}\,v
			\right)
			\,+\,
			b\,\theta^\eps\,v\,
			\right)
			\frac{g^\eps}{\ols{g}^\eps}\,
			\dD v\,
			\right|^2
			\ols{g}^\eps
			\dD w\,
			,\\[1.5em]
			\displaystyle  \cR_2(t) \,=\,
			\frac{a\,\eps}{\rho_0^\eps}
			\int_{\R}
			\,
			\left|\,
			\partial_w
			\left[\,
			w
			\,\ols{\nu}\,\right]\right|
			\,+\,
			\frac{a\,\eps}{\rho_0^\eps}\,
			\left|\,
			\partial^2_w\,
			\ols{\nu}\,\right|\,\dD w\,.
		\end{array}
		\right.\\[0.8em]
	\end{equation*}
	
	We estimate $\cR_1$ according to Jensen's inequality
	\[
	\cR_1(t)
	\,\leq\,
	\frac{1}{4}\int_{\R^2}
	\left|
	B^\eps_0
	\left(
	t\,,\,\bx\,,\,\theta^{\eps}\,v\,,\,-\,\gamma^{\eps}\,v
	\right)
	\,+\,
	b\,\theta^\eps\,v
	\right|^2
	g^\eps\,\dD v\,
	\dD w\,.
	\]
	Then we bound $B^\eps_0$: on the one hand $\mathcal{V}^\eps$ is uniformly bounded with respect to both $(t,\bx) \in \R^+\times K$ and $\eps\,>\,0$ according \eqref{estimate moment mu} in Proposition \ref{th:preliminary}, on the other hand according to assumptions \eqref{hyp2:psi} and \eqref{hyp:rho0} on $\Psi$ and $\rho_0^\eps$, $\Psi*_r\rho^\eps_0$ is uniformly bounded with respect to both $\bx \in K$ and $\eps\,>\,0$. Consequently, applying Young's inequality, assumption \eqref{hyp2:N} and since $N$ is locally Lipschitz, we obtain
	\[
	\cR_1(t) \,\leq\,
	C\,
	\int_{\R^2}
	\left(
	\left|\,\theta^\eps\,v\,
	\right|^{2p}
	\,+\,
	\left|\,\theta^\eps\,v\,
	\right|^2
	\,+\,
	\left|\,\mathcal{E}
	\left(
	\mu^\eps
	\right)\,
	\right|^2
	\right)
	g^\eps
	\dD \bu\,
	\,,
	\]
	as long as $\eps$ is less than $1$ to ensure that 
	$\gamma^\eps$ given by \eqref{change:var g} is less than 
	$\ds
	a\,\,\theta^\eps\,/\,m_*
	\,$.
	We invert the changes of variables \eqref{change:var g} and \eqref{change:var} and apply items \eqref{estimate rel energy mu} and \eqref{estimate error} in Proposition \ref{th:preliminary}, this yields
	\[
	\cR_1(t) \,\leq\,
	C\,
	\left(
	e^{-2\rho_0^\eps(\bx) t/\eps}
	+
	\eps\right)
	\,.
	\]    
	Then to estimate $\cR_2$, we apply Proposition \ref{lemme:bar nu L 1} to bound $\ols{\nu}$ and assumption \eqref{hyp:rho0} to lower bound $\rho_0^\eps$, it yields
	\[
	\cR_2(t) \,\leq\,
	C\,\eps\,e^{2\,b\,t}
	\,.
	\]
	We gather the former computations and take the supremum over all $\bx$ in $K$: it yields the expected result.
\end{proof}

\subsubsection{Convergence of
	$\ols{\nu}^\eps$ towards
	$\,
	\ols{\nu}
	$}

\noindent In this section, we gather the result from the last steps to deduce that $\ols{\nu}^\eps$ converges towards
$\,
\ols{\nu}\,
$. 

\begin{proposition}\label{cv:bar nu eps bar nu}
	Under the assumptions of Theorem \ref{th1}, there exists a positive constant $C$ independent of $\eps$ such that for all $\eps$ less than $1$, it holds
	\[
	\left\|
	\,\ols{\nu}^\eps
	\,-\,
	\ols{\nu}\,
	\right\|_{
		L^{\infty}\left(K\,,\,L^1\left([\,0,\,t\,] \times \R\right)\right)
	}
	\,\leq\,
	2\,\sqrt{2}
	\,t\,
	\left\|\,
	\ols{\nu}^\eps_0
	\,-\,
	\ols{\nu}_0\,
	\right\|^{\frac{1}{2}}_{L^{\infty}_{\bx}
		L^{1}_{w}}
	\,+\,
	2\,\sqrt{\eps\,t}\,
	m_2
	\,+\,
	C\,e^{b\,t}\,\sqrt{\eps}\,,
	\quad
	\forall\,
	t\,\in\,\R^+\,.
	\]
	In this result, the constant $C$ only depends on $m_1$, $m_*$,~$m_p$~and~$\ols{m}_p$ (see assumptions \eqref{hyp:rho0}-\eqref{hyp2:f0}) and the data of the problem $\ols{\nu}_0$, $N$, $\Psi$ and $A_0$.
\end{proposition}

\begin{proof} We choose some $\bx\in K$ and for all $t\geq 0$, we consider the following triangular inequality
	\[
	\left\|
	\,\ols{\nu}^\eps
	\,-\,
	\ols{\nu}\,
	\right\|_{
		L^{\infty}\left(K\,,\,L^1\left([\,0,\,t\,] \times \R\right)\right)}
	\,\leq\,
	\sup_{\bx\in K}
	\int_0^t
	\cA(s,\bx)\,+\cB(s,\bx)\,\dD s\,,
	\]
	where $\cA$ and $\cB$ are given by
	\begin{equation*}
		\left\{
		\begin{array}{l}
			\displaystyle  \cA(s,\bx) \,=\,\left\|
			\,\ols{\nu}^\eps(s,\bx)
			\,-\,
			\ols{g}^{\eps}(s,\bx)\,
			\right\|_{L^1(\R)
			}\,
			,\\[0.8em]
			\displaystyle  \cB(s,\bx) \,=\,
			\left\|
			\,\ols{g}^\eps(s,\bx)
			\,-\,
			\ols{\nu}(s,\bx)\,
			\right\|_{
				L^1(\R)
			}
			\,.
		\end{array}
		\right.\\[0.8em]
	\end{equation*}
	We estimate $\cB$ applying Proposition \ref{cv: bar g eps/bar nu}, which ensures 
	\[
	\cB(s,\bx)
	\,\leq\,
	2\,\sqrt{2}\,
	\left\|\,
	\ols{g}^\eps_0
	\,-\,
	\ols{\nu}_0
	\right\|^{\frac{1}{2}}_{L^{\infty}_{\bx}
		L^{1}_{w}}
	\,+\,
	C\,e^{b\,s}\,\sqrt{\eps}
	\,,\quad\forall(s,\bx)\in\R^+\times K\,.
	\]
	To bound $\left\|\,
		\ols{g}^\eps_0
		\,-\,
		\ols{\nu}_0
		\right\|_{L^{\infty}_{\bx}
			L^{1}_{w}}$, we first apply the following triangular inequality
		\[
		\left\|\,
		\ols{g}^\eps_0
		\,-\,
		\ols{\nu}_0
		\right\|_{L^{\infty}_{\bx}
			L^{1}_{w}}
		\leq
		\left\|\,
		\ols{g}^\eps_0
		\,-\,
		\ols{\nu}^\eps_0
		\right\|_{L^{\infty}_{\bx}
			L^{1}_{w}}
		+
		\left\|\,
		\ols{\nu}^\eps_0
		\,-\,
		\ols{\nu}_0
		\right\|_{L^{\infty}_{\bx}
			L^{1}_{w}}\,,
		\]
		and then estimate $\left\|\,
		\ols{g}^\eps_0
		\,-\,
		\ols{\nu}^\eps_0
		\right\|_{L^{\infty}_{\bx}
			L^{1}_{w}}$ replacing $g^\eps_0$ with $\nu^\eps_0$  in the definition of $\bar{g}^\eps_0$ according to the change of variable \eqref{change:var g}, that is
		\[
		\left\|\,
		\ols{g}^\eps_0
		\,-\,
		\ols{\nu}^\eps_0
		\right\|_{L^{\infty}_{\bx}
			L^{1}_{w}}
		\,=\,
		\sup_{\bx \in K}
		\int_{\R}
		\left|
		\int_{\R}
		\nu^\eps_0\left(\bx,v,w-\gamma^\eps_0\,v\right) - \nu^\eps_0\left(\bx,v,w\right)\dD v\,
		\right|\dD w\,.
		\]
		Applying the integral triangle inequality in the latter relation, we deduce
		\[
		\left\|\,
		\ols{g}^\eps_0
		\,-\,
		\ols{\nu}^\eps_0
		\right\|_{L^{\infty}_{\bx}
			L^{1}_{w}}
		\,\leq\,
		\sup_{\bx \in K}
		\left\|
		\,
		\nu^\eps_0
		\,-\,
		\tau_{-
			\gamma^\eps_0 v
		}\,
		\nu^\eps_0
		\,
		\right\|_{
			L^1
			\left(
			\R^2
			\right)
		}\,.
		\]
		Since $\gamma^\eps_0$ is given by \eqref{change:var g} and according to assumption \eqref{hyp:rho0} on $\rho_0^\eps$, it holds $|\gamma^\eps_0|\leq C\,\eps$. Hence, we deduce 
		\[
		\left\|\,
		\ols{g}^\eps_0
		\,-\,
		\ols{\nu}^\eps_0
		\right\|_{L^{\infty}_{\bx}
			L^{1}_{w}}
		\,\leq\,C\,\eps
		\sup_{\bx \in K}
		\frac{1}{\left|\gamma^\eps_0\right|}
		\left\|
		\,
		\nu^\eps_0
		\,-\,
		\tau_{-
			\gamma^\eps_0 v
		}\,
		\nu^\eps_0
		\,
		\right\|_{
			L^1
			\left(
			\R^2
			\right)
		}\,.
		\]
		Applying assumption \eqref{hyp1 nu L 1} to bound the right hand side in the latter estimate, we obtain
		\[
		\left\|\,
		\ols{g}^\eps_0
		\,-\,
		\ols{\nu}^\eps_0
		\right\|_{L^{\infty}_{\bx}
			L^{1}_{w}}
		\,\leq\,C\,\eps\,.
		\]
	Gathering the latter computations, we deduce
	\[
	\cB(s,\bx)
	\,\leq\,2\,\sqrt{2}
	\left\|\,
	\ols{\nu}^\eps_0
	\,-\,
	\ols{\nu}_0
	\right\|^{\frac{1}{2}}_{L^{\infty}_{\bx}
		L^{1}_{w}}
	\,+\,
	C\,e^{b\,s}\,\sqrt{\eps}
	\,.
	\]
	We integrate the latter estimate between $0$ and $t$ and take the supremum over all $\bx\in K$,  it yields
	\[
	\sup_{\bx \in K}
	\int_0^t
	\cB(s,\bx)\,\dD s
	\,\leq\,2\,\sqrt{2}\,t
	\left\|\,
	\ols{\nu}^\eps_0
	\,-\,
	\ols{\nu}_0
	\right\|^{\frac{1}{2}}_{L^{\infty}_{\bx}
		L^{1}_{w}} 
	\,+\,
	C\,e^{b\,t}\,\sqrt{\eps}
	\,.
	\]
	To estimate $\cA$, we replace $g^\eps$ with $\nu^\eps$ in the definition of $\bar{g}^\eps$ according to the change of variable \eqref{change:var g}, that is
	\[
	\cA(s,\bx)
	\,=\,
	\int_{\R}
	\left|
	\int_{\R}
	\nu^\eps\left(s,\bx,v,w\right)\,-\,
	\nu^\eps\left(s,\bx,v,w-\gamma^\eps\,v\right) \dD v
	\right|\dD w\,,
	\]
	and then consider the following decomposition
	\[
	\cA(s,\bx)
	\,\leq\,
	\cA_{1}(s,\bx)\,+\,\cA_{2}(s,\bx)\,,
	\]
	where $\cA_{1}$ and $\cA_{2}$ are defined as follows
	\begin{equation*}
		\left\{
		\begin{array}{l}
			\displaystyle  \cA_1(s,\bx) \,=\,
			\int_{\R}
			\left|
			\int_{\R^2}
			\cM_{\rho^\eps_0}(\Tilde{v})
			\left(
			\nu^\eps(s,\bx,v,w)-\nu^\eps(s,\bx, v,w-\gamma^\eps \Tilde{v})\right)\,\dD \Tilde{v}\,\dD v
			\right|\,\dD w
			,\\[1.2em]
			\displaystyle  \cA_2(s,\bx)
			\,=\,
			\int_{\R}
			\left|
			\int_{\R}
			\cM_{\rho^\eps_0}(v)\ols{\nu}^\eps(s,\bx,w-\gamma^\eps v)-
			\nu^\eps(s,\bx,v,w-\gamma^\eps v)
			\,\dD v
			\right|\,\dD w
			\,.
		\end{array}
		\right.
	\end{equation*}
	Applying the integral triangle inequality in the latter relations, we obtain
	\begin{equation*}
		\left\{
		\begin{array}{l}
			\displaystyle  \cA_1(s,\bx) \,\leq\,
			\int_{\R}
			\mathcal{M}_{\rho_0^\eps}(\Tilde{v})\,
			\left\|
			\,
			\nu^\eps
			\,-\,
			\tau_{-
				\left(
				\gamma^\eps\,\Tilde{v}
				\right)
			}\,
			\nu^\eps
			\,
			\right\|_{
				L^1
				\left(
				\R^2
				\right)
			}\,(s,\bx)\,\dD \Tilde{v}\,
			,\\[1.3em]
			\displaystyle  \cA_2(s,\bx)
			\,\leq
			\,
			\int_{\R^2}\left|
			\mathcal{M}_{\rho_0^\eps}
			\otimes
			\ols{\nu}^\eps
			\,-\,
			\nu^\eps\right|(s,\bx,v,w-\gamma^\eps v)\,\dD\bu
			\,.
		\end{array}
		\right.
	\end{equation*}
	To estimate $\cA_2$, we first perform the change of variable $w\leftarrow w-\gamma^\eps v$ in the right hand side of the latter inequality, this yields
	\[
	\cA_2(s,\bx)
	\,\leq\,
	\left\|
	\,
	\mathcal{M}_{\rho_0^\eps}
	\otimes
	\ols{\nu}^\eps(s,\bx)
	\,-\,
	\nu^\eps(s,\bx)
	\,
	\right\|_{
		L^1\left(\R^2\right)
	}
	\,.
	\]
	Integrating the latter estimate from $0$ to $t$, taking the supremum over all $\bx \in K$ and applying Proposition \ref{estimee rel ent nu M} to the right hand side, we deduce
	\[
	\int_0^t
	\cA_2(s,\bx)
	\,
	\dD s
	\,\leq\,2\,\sqrt{\eps\,t}\,
	m_2
	\,+\,
	C\,\sqrt{\eps}\,(t+1)\,.
	\]
	To estimate $\cA_1$, we apply Proposition \ref{equicontinuity} which yields
	\[
	\cA_1(s,\bx)\,\leq\,C\sqrt{\eps}\,e^{b\,s}\,.
	\]
	After integrating the latter estimate and taking the supremum over all $\bx\in K$, we deduce
	\[
	\sup_{\bx \in K}
	\int_0^t
	\cA_1(s,\bx)\,\dD s
	\,\leq\,C\sqrt{\eps}\,e^{b\,t}
	\,.
	\]
	We obtain the result gathering the former estimates.
\end{proof}
Theorem \ref{th1} is obtained taking the sum between the estimates in Propositions \ref{estimee rel ent nu M} and \ref{cv:bar nu eps bar nu}.

\section{Convergence analysis in weighted $L^2$ spaces}\label{sec2}
In this section, we derive convergence estimates for $\mu^\eps$ in a weighted $L^2$ functional framework. 
We take advantage of the variational structure of $L^2$ spaces in order to derive uniform regularity estimates for $\mu^\eps$. These key estimates are the object of the following section
\subsection{\textit{A priori} estimates}\label{Estimates for strong convergence}
The main purpose of this section is to propagate the $\scH^{k}$-norms along the trajectories of equation \eqref{nu:eq} uniformly with respect to $\eps$. 
We outline the strategy in the case of the $\scH^0$-norm. 
Its time derivative along the trajectories of equation \eqref{nu:eq} is obtained multiplying \eqref{nu:eq} by $\ds\nu^\eps\,m ^\eps$ and integrating with respect to $\bu$
\begin{equation*}
	\frac{1}{2}
	\frac{\dD}{\dD t}\,
	\|\,\nu^\eps\,\|_{L^2\left(m ^\eps\right)}^2
	\,=\,
	\frac{1}{
		\left(
		\theta^\eps
		\right)^2
	}
	\left\langle\,
	\cF_{\rho_0^\eps}
	\left[
	\,\nu^\eps\,
	\right]\,,\,
	\nu^\eps\,
	\right\rangle_{L^2\left(m ^\eps\right)}
	\,-\,
	\left\langle\,
	\mathrm{div}_{\bu}
	\left[\,
	\mathbf{b}^\eps_0\,
	\nu^\eps\,
	\right]\,,\,
	\nu^\eps\,
	\right\rangle_{L^2\left(m ^\eps\right)}
	\,.
\end{equation*}

We first point out that according to the following lemma, the term associated to the Fokker-Planck operator is dissipative and is consequently a helping term in the upcoming analysis
\begin{lemma}
	For all $\bx$ in $K$, it holds 
	\[
	\left\langle\,
	\cF_{\rho_0^\eps(\bx)}
	\left[
	\,\nu\,
	\right]\,,\,
	\nu\,
	\right\rangle_{L^2\left( m^\eps_{\bx} \right)}
	\,\,=\,\,
	-\,
	\cD_{\rho_0^\eps(\bx)}
	\left[\,\nu\,
	\right]\,\leq\,0\,,
	\] 
	for all
	$
	\displaystyle
	\nu\, 
	\in\,H^1
	\left(
	m^\eps_{\bx}
	\right)
	$, where the dissipation 
	$\ds \mathcal{D}_{\rho_0^\eps}$
	is given by
	\begin{equation*}
		\cD_{\rho_0^\eps(\bx)}
		\left[\,\nu\,
		\right]
		\,=\,
		\int_{\R^2}\,
		\left|
		\partial_v
		\left(\,\nu\,m^\eps_{\bx}\,
		\right)
		\right|^2\,
		\left(
		m^\eps_{\bx}\right)^{-1}\,\dD \bu\,\geq\,0\,.
	\end{equation*}
\end{lemma}
\begin{proof}
	The Fokker-Planck operator rewrites as follows
	\begin{equation*}
		\cF_{\rho_0^\eps(\bx)}
		\left[\,
		\nu\,
		\right]
		\,
		=\,
		\partial_v
		\left[\,
		\left(
		m^\eps_{\bx}\right)^{-1}\,
		\partial_v
		\left(
		\nu\,
		m^\eps_{\bx}
		\right)\,
		\right].
	\end{equation*}
	Consequently, the result is obtained integrating
	$\ds
	\cF_{\rho_0^\eps(\bx)}
	\left[\,
	\nu\,
	\right]
	$
	against $\nu \,m ^\eps$ with respect to $\bu$ and then integrating by part with respect to $v$.\\
\end{proof}
Therefore, the main challenge is to control the contribution of the transport operator 
$\displaystyle
\mathrm{div}_{\bu}\,\mathbf{b}^\eps_0
$ with the dissipation $\ds\mathcal{D}_{\rho_0^\eps}$ brought by the
Fokker-Planck operator, which is done in the following
lemma.
\begin{lemma}\label{lemme technique L 2 m}
	Under assumptions \eqref{hyp1:N}-\eqref{hyp2:N} and \eqref{hyp3:N} on
	the drift $N$, \eqref{hyp2:psi} on the interaction kernel $\Psi$,
	consider a sequence of solutions $(\mu^\eps)_{\eps\,>\,0}$ to
	\eqref{kinetic:eq} with initial conditions satisfying assumptions
	\eqref{hyp:rho0}-\eqref{hyp2:f0}. Then, for any $\alpha$ greater than $\displaystyle 1/(2b\kappa)$, there exists a constant $C>0$
	such that  for all $\eps>0$, we have
	\[
	\left.\\[0,5em]
	-\,
	\left\langle\,
	\mathrm{div}_{\bu}
	\left[\,
	\mathbf{b}^\eps_0\,
	\nu\,
	\right]\,,\,
	\nu\,
	\right\rangle_{L^2\left( m^\eps_{\bx} \right)}
	\,\,\leq\,\,
	\frac{
		\alpha
	}{
		\left(
		\theta^\eps
		\right)^2
	}\,
	\cD_{\rho_0^\eps(\bx)}
	\left[\,\nu\,
	\right]
	\,
	+\,
	C\,
	\|\,\nu\,\|_{L^2(m^\eps_{\bx})}^2\,
	,
	\right.\\[0,3em]
	\]
	for all
	$
	\displaystyle
	(t\,,\,\bx)\, 
	\in\,\R^+
	\times K 
	$ and all
	$
	\displaystyle
	\nu\, 
	\in\,H^1
	\left(
	m^\eps_{\bx}
	\right)
	$, where $\kappa$ appears in the definition \eqref{def L 2 m} of $m^\eps_\bx$.
\end{lemma}
Before getting into the heart of the proof, we point out that as long as the latter lemma holds with some $\alpha$ less than $1$, the sum of the estimates in the two latter Lemmas yield
\begin{equation*}
	\frac{1}{2}
	\frac{\dD}{\dD t}\,
	\|\,\nu^\eps\,\|_{L^2\left(m ^\eps\right)}^2
	\,\leq\,
	C\,
	\|\,\nu^\eps\,\|_{L^2\left(m ^\eps\right)}^2\,,
\end{equation*}
which ensures that the $\scH^0$-norm is propagated along the curves of \eqref{nu:eq} uniformly with respect to $\eps$. We follow the exact same strategy in order to propagate the $\scH^k$-norms when $k$ is not $0\,$: see Proposition \ref{estime L 2 m g eps} for more details. Moreover, we emphasize that the constraint \eqref{condition kappa} on $\kappa$ in Theorem \ref{th:2} arises from the lower bound on $\alpha$ in Lemma \ref{lemme technique L 2 m}.
\begin{remark}
	Due to the structure of the space $L^2\left(m ^\eps\right)$, we added the confining assumption \eqref{hyp3:N} on the drift $N$ to Theorem \ref{th:2}. Our proof of Lemma \ref{lemme technique L 2 m} crucially relies on this assumption ; it is the only time that we use it as well.
\end{remark}
\begin{proof}
	All along this proof, we consider some $\eps\,>\,0$ and some $\displaystyle(t,\,\bx)$ in $\R_+\times K$; we omit the dependence with respect to $\bx$ when the context is clear. Furthermore, we choose some $\nu$ in $H^1
	\left(
	m^\eps_{\bx}
	\right)
	$. Since $\mathbf{b^\eps_0}$ is given by \eqref{def:b0}, we have
	\begin{equation*}
		-\,
		\left\langle\,
		\mathrm{div}_{\bu}
		\left[\,
		\mathbf{b}^\eps_0\,
		\nu\,
		\right]\,,\,
		\nu\,
		\right\rangle_{L^2\left(m ^\eps\right)}
		\,=\,
		\mathcal{A}_1 + \mathcal{A}_2 + \mathcal{A}_3\,,
	\end{equation*}
	where
	\begin{equation*}
		\left\{
		\begin{array}{l}
			\displaystyle  \cA_{1} \,=\,
			\frac{1}{\theta^\eps}
			\int_{\R^2}
			\partial_v \left[\,
			w\,\nu\,
			\right]\,
			\nu\,m ^\eps\,\dD \bu 
			\,-\,\int_{\R^2}
			\partial_w
			\left[\,
			A_0
			\left(\theta^\eps v,\,w
			\right)\nu\,
			\right]\,\nu\,m ^\eps\, \dD \bu\,
			,\\[1.1em]
			\displaystyle \cA_{2}
			\,=\,
			-\,
			\frac{1}{\theta^\eps}
			\int_{\R^2}
			\partial_v
			\left[\,
			\left(
			N(\mathcal{V}^\eps
			+
			\theta^\eps v)
			\,-\,
			N(\mathcal{V}^\eps)
			\right)
			\nu\,\right]\,
			\nu\,m ^\eps\, \dD \bu\,,\\[1.1em]
			\displaystyle \cA_{3} \,=\,
			\int_{\R^2}
			\partial_v
			\left[\,
			\left(
			v\,\Psi*_r\rho^\eps_0
			\,+\,
			\mathcal{E}(\mu^\eps)
			\left(\theta^\eps\right)^{-1}
			\right)
			\nu\,\right]\,\nu\,
			m ^\eps\,\dD \bu\,
			.
		\end{array}
		\right.\\[1em]
	\end{equation*}
	To estimate $\cA_1$, we take advantage of the confining properties of
	$A_0$. When it comes to $\cA_2$, the estimate relies on the confining
	properties of the non-linear drift $N$. The last term  $\cA_3$ gathers the lower order terms and adds no difficulty.
	
	We start with $\cA_{1}$, which
	rewrites as follows after exact computations and an integration by part,
	\[
	\cA_{1}
	\,=\,-
	\int_{\R^2}
	w\,\nu \left(m ^\eps\right)^{1/2}\,
	\frac{1}{\theta^\eps}\,
	\partial_v \left[\,
	\nu\,m ^\eps\,
	\right]
	\left(
	m ^\eps
	\right)^{-1/2}\dD \bu\,+\,
	\frac{1}{2}\,\int_{\R^2}
	\left(
	\kappa w
	A_0
	\left(\theta^\eps v,w
	\right)
	-
	\partial_w\,A_0
	\right)
	\left|
	\nu
	\right|^2 m ^\eps\, \dD \bu
	\,.
	\]
	According to the definition of $A_0$ and applying Young's inequality, we obtain
	\[
	\cA_{1}
	\,\leq\,
	\frac{1}{
		2\,\eta_2
		\left(
		\theta^\eps
		\right)^2
	}\,
	\cD_{\rho_0^\eps}
	\left[\,\nu\,
	\right]
	+
	\int_{\R^2}
	\left(
	\frac{C\kappa^2}{\eta_1}
	\left|
	\theta^\eps v
	\right|^2
	+
	\left(
	C\eta_1
	+
	\frac{\eta_2-b\kappa}{2}
	\right)
	w^2
	\right)
	\left|
	\nu
	\right|^2m ^\eps\, \dD \bu
	\,
	+\,
	C
	\|
	\nu
	\|^2_{L^2\left(m ^\eps\right)}\,
	,
	\]
	for all positive
	$
	\displaystyle 
	\eta_1
	$
	and 
	$
	\displaystyle 
	\eta_2
	$
	and for some positive constant $C$.
	We set 
	$
	\displaystyle
	\alpha_{-}
	\,=\,
	(\alpha\,+\,1/(2b\kappa))/2
	$, 
	$\displaystyle
	\eta_2 = 
	1/(2\alpha_{-})
	$ and 
	$
	\displaystyle
	\eta_1\,
	=\,
	\left(
	b\kappa
	\,-\,
	\eta_2
	\right)/(2C)
	$ which is positive according to the condition on $\alpha$ in Lemma \ref{lemme technique L 2 m}. With this choice, we have $\ds C \eta_1
	+(
	\eta_2
	\,-\,
	b\kappa)/2\,=\,0$ and consequently, we obtain
	\begin{equation}\label{A 1}
		\cA_1
		\, \leq \,
		\frac{\alpha_{-}}{
			\left(
			\theta^\eps
			\right)^2
		}\,
		\cD_{\rho_0^\eps}
		\left[\,\nu\,
		\right]
		\,+\,C
		\int_{\R^2}
		\left|
		\theta^\eps v
		\right|^2
		\left|
		\nu
		\right|^2\,m ^\eps\, \dD \bu
		\,
		+\,
		C
		\|
		\nu
		\|^2_{L^2\left(m ^\eps\right)}\,,
	\end{equation}
	for some positive constant $C$ only depending on $\displaystyle A_0$, $\kappa$ and $\displaystyle \alpha$. 

	To estimate $\cA_2$, we take advantage of the super-linear decay of $N$ (see assumption \eqref{hyp1:N}) in order to control the terms growing at most linearly. We emphasize that the decaying property of $N$ is prescribed at infinity. Consequently, it may not have confining property on bounded sets. Hence, the main point here consists in isolating the domain where $N$ decays super linearly.\\
	After some exact computations and an integration by part, $\cA_2$ rewrites
	\[
	\cA_{2}
	\,=\,
	\frac{1}{2}\,\int_{\R^2}
	\left[\,
	\frac{\rho_0^\eps}{\theta^\eps}
	\,
	v\,
	\left(
	N(\mathcal{V}^\eps
	+
	\theta^\eps v)
	\,-\,
	N(\mathcal{V}^\eps)
	\right)
	\,-\,
	N'
	\left(
	\mathcal{V}^\eps
	+
	\theta^\eps v
	\right)\,
	\right]
	\,\left|
	\nu
	\right|^2\,m ^\eps\, \dD \bu\,.
	\]
	We consider some $R>0$ and split the former expression in three different parts
	\[
	\cA_{2}
	\,=\,
	\cA_{21}
	\,+\,
	\cA_{22}
	\,+\,
	\cA_{23}\,,
	\]
	where
	\begin{equation*}
		\left\{
		\begin{array}{l}
			\displaystyle \cA_{21} \,=\,
			\frac{\rho_0^\eps}{2\theta^\eps}\,\int_{\R^2}
			\mathds{1}_{|\theta^\eps v
				| \,>\, R}\,
			v\,
			\left(
			N(\mathcal{V}^\eps
			+
			\theta^\eps v)
			\,-\,
			N(\mathcal{V}^\eps)
			\right)
			\,\left|
			\nu
			\right|^2\,m ^\eps\, \dD \bu\,
			,\\[1.5em]
			\displaystyle  \cA_{22} \,=\,
			\frac{1}{2}\,\int_{\R^2}
			\mathds{1}_{|\theta^\eps v
				| \,\leq\, R}
			\left[\,
			\frac{\rho_0^\eps}{\theta^\eps}\,v
			\left(
			N(\mathcal{V}^\eps
			+
			\theta^\eps v)
			\,-\,
			N(\mathcal{V}^\eps)
			\right)
			\,-\,
			N'
			\left(
			\mathcal{V}^\eps
			+
			\theta^\eps v
			\right)
			\right]
			\,\left|
			\nu
			\right|^2\,m ^\eps\, \dD \bu\,
			,\\[1.5em]
			\displaystyle
			\cA_{23}
			\,=\,-
			\frac{1}{2}\,\int_{\R^2}
			\mathds{1}_{|\theta^\eps v
				| > R}\,
			N'
			\left(
			\mathcal{V}^\eps
			+
			\theta^\eps v
			\right)
			\,\left|
			\nu
			\right|^2\,m ^\eps\, \dD \bu
			\,.\\[0,8em]
		\end{array}
		\right.
	\end{equation*}
	The first term $\cA_{21}$ corresponds to the the contribution of $N$ on the domain where it decays super-linearly. Consequently, $\cA_{21}$ is non positive for $R$ great enough. We take advantage of the helping term $\cA_{21}$ to control $\cA_{22}$, which corresponds to the contribution of $N$ on bounded sets. We estimate $\cA_3$ taking advantage of the confining term $\cA_{21}$ coupled with the confining assumption \eqref{hyp3:N} on $N$.\\
	
	Let us estimate $\cA_{21}$. According to item \eqref{estimate moment mu} in Proposition \ref{th:preliminary}, $\mathcal{V}^\eps$ is uniformly bounded with respect to both
	$
	\left(
	t,\,\bx
	\right)
	\in
	\R^+
	\times 
	K
	$ and $\eps$. Therefore, since $N$ is continuous, we bound $N(\cV^\eps)$ by a constant in the following expression 
	\[
	\mathds{1}_{|\theta^\eps v
		| >R}\,
	\theta^\eps v\,
	\left(
	N(\mathcal{V}^\eps
	+
	\theta^\eps v)
	-
	N(\mathcal{V}^\eps)
	\right)
	\,\leq\,
	\mathds{1}_{|\theta^\eps v
		| > R}
	\left(
	\left|
	\mathcal{V}^\eps
	+
	\theta^\eps v
	\right|^2
	\omega(\mathcal{V}^\eps
	+
	\theta^\eps v)\,
	\frac{\theta^\eps v}{\mathcal{V}^\eps
		+
		\theta^\eps v}
	+
	C|\theta^\eps v|\,
	\right),
	\]
	where $\omega$ is given in assumption \eqref{hyp1:N} and where the constant $C$ depends on both $N$ and the uniform upper bound on $\mathcal{V}^\eps$. Since $\cV^\eps$ is uniformly bounded, there exists $R$ large enough such that 
	\[
	\frac{1}{2}\;\leq\;
	\frac{\theta^\eps v}{\mathcal{V}^\eps
		+
		\theta^\eps v}\,,
	\]
	for all $|\theta^\eps v|>R$. Furthermore, since $\cV^\eps$ is uniformly bounded and according to assumption \eqref{hyp1:N}, there exists $R$ large enough such that 
	\[
	\mathds{1}_{|\theta^\eps v
		| \,>\, R}\,
	\left(
	\frac{1}{2}
	\left|
	\mathcal{V}^\eps
	+
	\theta^\eps v
	\right|^2
	\omega(\mathcal{V}^\eps
	+
	\theta^\eps v)\
	\,+\,
	C\,|\theta^\eps v|\,
	\right)\,\leq\,0\,.
	\]
	From now on, \textbf{we fix} $R$ such that the latter two relations hold true. For such $R$, it holds
	\[
	\cA_{21} \,\leq\,
	\int_{\R^2}
	\left(
	\frac{m_*}{4}\,
	\mathds{1}_{|\theta^\eps v
		| \,>\, R}\,
	\left|
	\mathcal{V}^\eps
	+
	\theta^\eps v
	\right|^2
	\omega(\mathcal{V}^\eps
	+
	\theta^\eps v)
	\,+\,
	C\,|\theta^\eps v|\,
	\right)\,\left|
	\nu
	\right|^2\,m ^\eps\, \dD \bu\,,
	\]
	where we used that $\theta^\eps \leq 1$ and where $m_*$ is the lower bound of $\rho_0^\eps$ given by assumption \eqref{hyp:rho0}. We note that the radius $R$ depends only on $N$ and the uniform bound on  $\displaystyle|\mathcal{V}^\eps|$.  Furthermore, we introduce the following notation
	\[
	\mathcal{N}
	\,=\,
	\int_{\R^2}
	\mathds{1}_{|\theta^\eps v
		| \,>\, R}\,
	\left|
	\mathcal{V}^\eps
	+
	\theta^\eps v
	\right|^2
	\omega(\mathcal{V}^\eps
	+
	\theta^\eps v)
	\,\left|
	\nu
	\right|^2\,m ^\eps\, \dD \bu
	\,\leq\,
	0\quad
	\text{when}\quad
	R\,\gg\,1
	\,.
	\]
	The term $\mathcal{N}$ corresponds to the contribution of $N$ on the domain where it has super-linear decaying properties and according to assumption \eqref{hyp1:N}, it is non positive for $R$ sufficiently large. Therefore, we use $\mathcal{N}$ to control the contribution of the other terms.
	With this notation, the estimate on $\cA_{21}$ rewrites
	\[
	\cA_{21} \,\leq\,
	C\int_{\R^2}
	|\theta^\eps v|\,\left|
	\nu
	\right|^2\,m ^\eps\, \dD \bu
	\,+\,
	\frac{m_*}{4}\,
	\mathcal{N}\,
	,
	\]
	where $C$ and $R$ only depend on $N$ and the uniform bound on  $\displaystyle|\mathcal{V}^\eps|$.\\
	
	We turn to $\cA_{22}$. Since $N$ has $\scC^1$ regularity and relying item \eqref{estimate moment mu} in Proposition \ref{th:preliminary} , which ensures that $\mathcal{V}^\eps$ stays uniformly bounded, we obtain
	\[
	\cA_{22} \,\leq\,C
	\,\frac{\rho_0^\eps}{2}
	\,
	\int_{\R^2}
	|v|^2
	\,\left|
	\nu
	\right|^2\,m ^\eps\, \dD \bu
	\,+\,
	C
	\|
	\nu
	\|^2_{L^2\left(m ^\eps\right)}
	\,,
	\]
	where $C$ is a positive constant which may depend on $m_*$, $N$ and the uniform bound on $|\mathcal{V}^\eps|$.
	We estimate the quadratic term in $v$ in the latter inequality thanks to the following relation 
	\[
	\frac{1}{2}
	\int_{\R^2}
	\left(
	\rho_0^\eps\,
	|v|^2
	\,-\,
	1
	\right)
	\,\left|
	\nu
	\right|^2\,m ^\eps\, \dD \bu
	\,=\,
	\int_{\R^2}
	v\,\nu\,
	\partial_v
	\left(
	\nu\,m ^\eps
	\right)\, \dD \bu\,,
	\]
	which is obtained after exact computations and an integration by part in the right hand side of the latter equality. We apply Young's inequality to the former relation and in the end it yields
	\begin{equation*}
		\cA_{22}
		\, \leq \,
		\frac{\eta}{
			\left(
			\theta^\eps
			\right)^2
		}\,
		\cD_{\rho_0^\eps}
		\left[\,\nu\,
		\right]
		\,+\,
		C
		\left(
		\frac{1}{\eta}
		\int_{\R^2}
		\left|
		\theta^\eps v
		\right|^2
		\left|
		\nu
		\right|^2\,m ^\eps\, \dD \bu
		\,
		+\,
		\|
		\nu
		\|^2_{L^2\left(m ^\eps\right)}
		\right)\,,
	\end{equation*}
	for all positive $\eta$ and for some positive constant $C$ depending on $m_*$, $N$ and the uniform bound on $|\mathcal{V}^\eps|$.

	We estimate the last term $\cA_{23}$ taking advantage of the confining properties corresponding to $\mathcal{N}$. Indeed we have
	\begin{equation*}
		\cA_{23}
		\,+\,
		\frac{m_*}{8}\,
		\mathcal{N}
		\,=\, 
		\int_{\R^2}
		\mathds{1}_{|\theta^\eps v
			| > R}
		\left(
		\frac{m_*}{8}\,
		\left|
		v'
		\right|^2
		\omega(v')
		-
		\frac{1}{2}
		N'
		\left(v'
		\right)
		\right)\left|
		\nu
		\right|^2\,m ^\eps\, \dD \bu\,,
	\end{equation*}
	where we used the shorthand notation 
	$\displaystyle
	v' \,=\,
	\mathcal{V}^\eps
	+
	\theta^\eps v$.
	Hence, according to assumption \eqref{hyp3:N}, we deduce
	\begin{equation*}
		\cA_{23}
		\,+\,
		\frac{m_*}{8}\,
		\mathcal{N}
		\,\leq\, 
		C\|
		\nu
		\|_{L^2\left(m ^\eps\right)}^2\,,
	\end{equation*}
	for some positive constant $C$ depending on $m_*$ and $N$. Gathering these computations, we obtain
	\begin{equation}\label{A 2}
		\cA_{2}
		\,\leq\, 
		\frac{\eta}{
			\left(
			\theta^\eps
			\right)^2
		}\,
		\cD_{\rho_0^\eps}
		\left[\,\nu\,
		\right]
		\,+\,
		C\left(
		\frac{1}{\eta}
		+1\right)
		\int_{\R^2}
		\left|\theta^\eps v\right|^2
		\left|
		\nu
		\right|^2\,m ^\eps\, \dD \bu
		\,+\,
		C\|
		\nu
		\|_{L^2\left(m ^\eps\right)}^2
		\,+\,
		\frac{m_*}{8}\,
		\mathcal{N}\,,
	\end{equation}
	for all $\eta\,>\,0$ and
	where $C$ is a positive constant which may depend on $m_*$, $R$, $N$ and the uniform bound on $|\mathcal{V}^\eps|$.\\
	
	We turn to $\cA_3$, which gathers terms of lower-order. We integrate by part and apply Young's inequality. It yields
	\begin{equation*}
		\cA_{3} \,\leq\,
		\frac{\eta}{
			\left(
			\theta^\eps
			\right)^2
		}\,
		\cD_{\rho_0^\eps}
		\left[\,\nu\,
		\right]
		\,+\,
		\frac{
			1
		}{2\eta}\,
		\left(
		\left|
		\Psi*_r\rho^\eps_0
		\right|^2\,
		\int_{\R^2}
		\left|
		\theta^\eps v
		\right|^2
		\left|
		\nu
		\right|^2\,m ^\eps\, \dD \bu
		\,
		+\,
		\left|
		\mathcal{E}(\mu^\eps)
		\right|^2\,
		\|
		\nu
		\|^2_{L^2\left(m ^\eps\right)}
		\right),
	\end{equation*}
	for all positive $\eta$.
	According to item \eqref{estimate error} in Proposition \ref{th:preliminary} , $\mathcal{E}(\mu^\eps)$ is uniformly bounded with respect to both $(t,\bx) \in \R^+\times K$ and $\eps\,>\,0$. Furthermore, according to assumptions \eqref{hyp2:psi} and \eqref{hyp:rho0} on $\Psi$ and $\rho_0^\eps$, $\Psi*_r\rho^\eps_0$ is uniformly bounded with respect to both $\bx \in K$ and $\eps\,>\,0$. Consequently, we obtain
	\begin{equation}\label{A 3}
		\cA_{3} \,\leq\,
		\frac{\eta}{
			\left(
			\theta^\eps
			\right)^2
		}\,
		\cD_{\rho_0^\eps}
		\left[\,\nu\,
		\right]
		\,+\,
		\frac{
			C
		}{\eta}\,
		\left(
		\int_{\R^2}
		\left|
		\theta^\eps v
		\right|^2
		\left|
		\nu
		\right|^2\,m ^\eps\, \dD \bu
		\,
		+\,
		\|
		\nu
		\|^2_{L^2\left(m ^\eps\right)}
		\right),
	\end{equation}
	for some positive constant $C$ which may depend on $m_*$ (see assumption \eqref{hyp:rho0}), $m_p$ and $\ols{m}_p$ (see assumptions \eqref{hyp1:f0} and \eqref{hyp2:f0}) and the data of our problem $N$, $\Psi$ and $A_0$.\\
	
	Gathering estimates \eqref{A 1}, \eqref{A 2} and \eqref{A 3}, it yields
	\begin{equation*}
		-
		\left\langle\,
		\mathrm{div}_{\bu}
		\left[\,
		\mathbf{b}^\eps_0\,
		\nu\,
		\right]\,,\,
		\nu\,
		\right\rangle
		\,\leq\,
		\frac{
			\alpha_{-}+2\eta
		}{
			\left(
			\theta^\eps
			\right)^2
		}\,
		\cD_{\rho_0^\eps}
		\left[\,\nu\,
		\right]
		+\,
		C
		\left(
		1
		+
		\frac{
			1
		}{\eta}
		\right)
		\int_{\R^2}
		\left(
		\left|
		\theta^\eps v
		\right|^2
		+1
		\right)
		\left|
		\nu
		\right|^2\,m ^\eps\, \dD \bu
		\,+\,
		\frac{m_*}{8}\,
		\mathcal{N}\,
		,
	\end{equation*}
	for all positive $\eta$. Hence, we choose $2\eta \,=\,
	\alpha-\alpha_{-}$. Therefore, replacing $\mathcal{N}$ by its definition, the former estimate rewrites
	\begin{equation*}
		-
		\left\langle\,
		\mathrm{div}_{\bu}
		\left[\,
		\mathbf{b}^\eps_0\,
		\nu\,
		\right],\,
		\nu\,
		\right\rangle
		\,\leq\,
		\frac{
			\alpha
		}{
			\left(
			\theta^\eps
			\right)^2
		}\,
		\cD_{\rho_0^\eps}
		\left[\,\nu\,
		\right]
		+
		\int_{\R^2}
		\left(
		C\left(
		\left|
		\theta^\eps v
		\right|^2
		+ 1
		\right)
		+
		\mathds{1}_{|\theta^\eps v
			| > R}\,
		\frac{m_*}{8}\,
		\left|
		v'
		\right|^2
		\omega(v')
		\right)
		\left|
		\nu
		\right|^2 m ^\eps \dD \bu\,,
	\end{equation*}
	where we used the shorthand notation 
	$\displaystyle
	v' \,=\,
	\mathcal{V}^\eps
	+
	\theta^\eps v$.
	To conclude, we estimate the right-hand side in the latter inequality applying assumption \eqref{hyp1:N} on $N$. Since $\mathcal{V}^\eps$ is uniformly bounded, we obtain
	\begin{equation*}
		-
		\left\langle\,
		\mathrm{div}_{\bu}
		\left[\,
		\mathbf{b}^\eps_0\,
		\nu\,
		\right]\,,\,
		\nu\,
		\right\rangle_{L^2\left(m ^\eps\right)}
		\,\leq\,
		\frac{
			\alpha
		}{
			\left(
			\theta^\eps
			\right)^2
		}\,
		\cD_{\rho_0^\eps}
		\left[\,\nu\,
		\right]
		\,+\,
		C
		\|\nu\|_{L^2\left(m ^\eps\right)}^2\,
		,
	\end{equation*}
	for some constant $C$ only depending on $\alpha$, $\kappa$, $m_*$, $m_p$, $\ols{m}_p$ and the data of the problem: $N$,~$A_0$~and~$\Psi$.
\end{proof}
We also mention the following general result, which may be interpreted as a Poincar\'e inequality in the functional space $L^2
\left(
m^\eps
\right)$
\begin{lemma}\label{P ineq}
	For all $\bx \in K$ and all function $\nu$ in $\ds H^1_{w}
	\left(
	m^\eps_{\bx}
	\right)$, hold the following estimates
	\begin{equation*}
		\left\|\,
		\nu\,
		\right\|_{
			L^2
			\left(
			m^\eps_{\bx}
			\right)
		}
		\,\leq\,
		\frac{1}{\sqrt{\kappa}}
		\left\|
		\,\partial_{w}\,
		\nu\,
		\right\|_{
			L^2
			\left(
			m^\eps_{\bx}
			\right)
		}\quad
		\mathrm{and}\quad\;
		\left\|\,
		w\,
		\nu\,
		\right\|_{
			L^2
			\left(
			m ^\eps_{\bx}
			\right)
		}
		\,\leq\,
		\frac{2}{\kappa}
		\left\|
		\,\partial_{w}\,
		\nu\,
		\right\|_{
			L^2
			\left(
			m ^\eps_{\bx}
			\right)
		}\,.
	\end{equation*}
\end{lemma}
\begin{proof}
	The proof relies on the following relation
	\[
	\frac{1}{2}
	\int_{\R^2}
	\,
	\left(
	1\,+\,
	\kappa\,w^2
	\right)
	\left|
	\,\nu\,
	\right|^2\,m ^\eps_{\bx}(\bu)\,\dD\bu
	\,=\,
	-
	\int_{\R^2}
	\,w
	\,\nu\,
	\partial_w\,\nu \,m ^\eps_{\bx}(\bu)\,\dD\bu\,,
	\]
	which is obtained after an integration by part in the right-hand side of the equality. From the latter relation we obtain the result applying Young's inequality in the right-hand side for the first estimate and Cauchy-Schwarz inequality for the second one.\\
\end{proof}
From Lemma \ref{lemme technique L 2 m}, we deduce regularity estimates for the solution $\nu^\eps$ to equation \eqref{nu:eq}. The main challenge consists in propagating the $\scH^0$-norm. Then we easily adapt our analysis to the case of the $\scH^k$-norms, when $k$ is greater than $0$. Indeed, the $w$-derivatives of $\nu^\eps$ solve equation \eqref{nu:eq} with additional source terms which we are able to control with the dissipation brought by the Fokker-Planck operator. More precisely, equation \eqref{nu:eq} on $\nu^\eps$ reads as follows
\begin{equation*}
	\partial_t\, \nu^\eps
	\,+\,
	\scA^\eps
	\left[
	\,\nu^\eps\, \right]
	\,=\,
	0\,,
\end{equation*}
where the operator $\scA^\eps$ is given by
\begin{equation}\label{op A}
	\scA^\eps
	\left[
	\,\nu^\eps\, \right]
	\,=\,
	\mathrm{div}_{\bu}
	\left[
	\,\mathbf{b}^\eps_0\,\nu^\eps\, \right]
	\,-\,
	\frac{1}{
		(\theta^\eps)^2}\,
	\cF_{\rho_0^\eps}
	\left[\,
	\nu^\eps\,
	\right].
\end{equation}
With this notation, the equations on the $w$-derivatives read as follows
\begin{equation}\label{eq:h}
	\partial_t\, 
	h^\eps
	\,+\,
	\scA^\eps
	\left[
	\,h^\eps\, \right]
	\,=\,
	\frac{1}{\theta^\eps}
	\,
	\partial_v\,
	\nu^\eps
	\,+\,
	b
	\,
	h^\eps
	\,
	,
\end{equation}
where $\ds h^\eps\,=\,\partial_w\,\nu^\eps$, and
\begin{equation}\label{eq:g}
	\partial_t\, 
	g^\eps
	\,+\,
	\scA^\eps
	\left[
	\,g^\eps\, \right]
	\,=\,
	\frac{2}{\theta^\eps}
	\,
	\partial_v\,
	h^\eps
	\,+\,
	2\,b
	\,
	g^\eps
	\,
	,
\end{equation}
where $g^\eps$ is given by $
\displaystyle
\partial_{w}^{\,2}\,\nu^\eps
$.
\begin{proposition}\label{estime L 2 m g eps}
	Under assumptions \eqref{hyp1:N}-\eqref{hyp2:N} and \eqref{hyp3:N} on
	the drift $N$, \eqref{hyp2:psi} on the interaction kernel $\Psi$,
	consider a sequence of smooth solutions $(\mu^\eps)_{\eps\,>\,0}$ to
	\eqref{kinetic:eq} with initial conditions satisfying assumptions
	\eqref{hyp:rho0}-\eqref{hyp2:f0} and \eqref{hyp4:f0} with an exponent
	$\kappa$ greater than $1/(2b)$. Then, there exists a
	constant $C>0$, such that, for~all~$\eps\,>\,0$, we have
	\[
	\left.\\[0,5em]
	\left\|\,
	\partial_{w}^{\,k}\,
	\nu^\eps(t\,,\,\bx)\,
	\right\|_{L^2\left(m^\eps_{\bx}\right)}
	\,\,\leq\,\,
	e^{Ct}\,
	\left\|\,\partial_{w}^{\,k}\,
	\nu^\eps_0(\bx)\,
	\right\|
	_{L^2
		\left(
		m^\eps_{\bx}\right)}\,
	, 
	\quad\forall \,(t,\bx)\, \in\,\R^+\times K\,,
	\right.\\[0,3em]
	\]
	for all $k$ in $\left\{0,1,2\right\}$.
\end{proposition}
\begin{proof}
	We start with $k = 0$. We compute the time derivative of 
	$
	\displaystyle
	\|\,\nu^\eps\,\|^2_{L^2\left(m ^\eps\right)}$ multiplying equation \eqref{nu:eq} by 
	$
	\displaystyle 
	\nu^\eps \, m
	$ and integrating with respect to $\bu$. After integrating by part the stiffer term, we obtain
	\begin{equation*}
		\frac{1}{2}
		\frac{\dD}{\dD t}\,
		\|\,\nu^\eps\,\|_{L^2\left(m ^\eps\right)}^2
		\,+\,
		\frac{1}{
			\left(
			\theta^\eps
			\right)^2
		}\,
		\cD_{\rho_0^\eps}
		\left[\,\nu^\eps\,
		\right]
		\,=\,
		-
		\left\langle\,
		\mathrm{div}_{\bu}
		\left[\,
		\mathbf{b}^\eps_0\,
		\nu^\eps\,
		\right]\,,\,
		\nu^\eps\,
		\right\rangle_{L^2\left(m ^\eps\right)}\,,
	\end{equation*}
	for all $\eps\,>\,0$ and all 
	$\displaystyle
	\left(
	t,\, \bx
	\right)
	\in 
	\R_+ \times K$. Since $\kappa$ is greater than $1/(2b)$, we apply Lemma \ref{lemme technique L 2 m} with $\alpha\,=\,1\,$. This leads to the following inequality
	\begin{equation*}
		\frac{\dD}{\dD t}\,
		\|\,\nu^\eps\,\|_{L^2\left(m ^\eps\right)}^2
		\,\leq\,
		C\,
		\|\,\nu^\eps\,\|_{L^2\left(m ^\eps\right)}^2\,
		,
	\end{equation*}
	for some constant $C$ only depending on $\kappa$, $m_*$, $m_p$, $\ols{m}_p$ and on the data of the problem: $N$, $A_0$ and $\Psi$. According to Gronwall's lemma, it yields
	\[
	\left.\\[0,5em]
	\left\|\,
	\nu^\eps(t\,,\,\bx)\,
	\right\|_{L^2\left(m^\eps_{\bx}\right)}
	\,\,\leq\,\,
	e^{Ct}\,
	\left\|\,
	\nu^\eps(0\,,\,\bx)\,
	\right\|
	_{L^2
		\left(
		m^\eps_{\bx}\right)}\,
	, 
	\quad\forall \,(t,\bx)\, \in\,\R^+\times K\,.
	\right.
	\]
	
	Let us now treat the case
	$
	\ds
	k\,=\,1$. We write $\ds h^\eps\,=\,\partial_w\,\nu^\eps$. We compute the derivative of
	$
	\displaystyle
	\|\,
	h^\eps\,\|^2_{L^2\left(m ^\eps\right)}$ multiplying equation \eqref{eq:h} by 
	$
	\ds
	h^\eps\, m^\eps$
	and integrating with respect to $\bu$. After integrating by part the stiffer term, we obtain
	\begin{equation*}
		\frac{1}{2}
		\frac{\dD}{\dD t}\,
		\|\,h^\eps\,
		\|_{L^2\left(m ^\eps\right)}^2
		\,+\,
		\frac{1}{
			\left(
			\theta^\eps
			\right)^2
		}\,
		\cD_{\rho_0^\eps}
		\left[\,
		h^\eps
		\,
		\right]
		\,=\,
		-
		\left\langle\,
		\mathrm{div}_{\bu}
		\left[\,
		\mathbf{b}^\eps_0\,
		h^\eps\,
		\right]\,,\,
		h^\eps\,
		\right\rangle_{L^2\left(m ^\eps\right)}
		\,+\,b\,
		\|\,
		h^\eps\,
		\|_{L^2\left(m ^\eps\right)}^2
		\,+\,
		\cB\,,
	\end{equation*}
	for all $\eps\,>\,0$ and all 
	$\displaystyle
	\left(
	t,\, \bx
	\right)
	\in 
	\R_+ \times K$, where $\cB$ is given by
	\[
	\cB\,=\,
	\frac{1}{\theta^\eps}
	\int_{\R^2}
	\partial_v\,\nu^\eps\,
	h^\eps
	\,m ^\eps\, \dD \bu\,.
	\]
	We estimate $\cB$ integrating by part and applying Young's inequality. It yields
	\[
	\cB\,\leq\,
	\frac{C}{\eta}\,
	\|\,
	\nu^\eps\,
	\|_{L^2\left(m ^\eps\right)}^2
	+\,
	\frac{\eta}{
		\left(
		\theta^\eps
		\right)^2
	}\,
	\cD_{\rho_0^\eps}
	\left[\,
	h^\eps\,
	\right]
	\,.
	\]
	for some positive constant $C$ and
	for all positive $\eta$. Then we apply Lemma \ref{P ineq}, which yields
	\[
	\cB\,\leq\,
	\frac{C}{\eta}\,
	\|\,
	h^\eps\,
	\|_{L^2\left(m ^\eps\right)}^2
	+\,
	\frac{\eta}{
		\left(
		\theta^\eps
		\right)^2
	}\,
	\cD_{\rho_0^\eps}
	\left[\,
	h^\eps\,
	\right]
	\,,
	\]
	and conclude this step following the same method as in the former step of the proof.\\
	
	The last case 
	$\ds k
	=
	2$ relies on the same arguments as the former step. Indeed, equation \eqref{eq:g} on $\partial^2_w\, \nu^\eps$  is the same as equation \eqref{eq:h} on
	$\partial_w\, \nu^\eps$
	up to a constant. Consequently, we skip the details and conclude this proof.\\
\end{proof} 
Due to the cross terms between the $v$ and $w$ variables in equation \eqref{kinetic:eq}, we are led to estimate mixed quantities of the form 
$
\ds
w^{k_1}\,\partial_{w}^{\,k_2}\,
\nu^\eps
$. These estimates are easily obtained from Proposition \ref{estime L 2 m g eps} and Lemma \ref{P ineq}
\begin{corollary}\label{estime mixed quantities}
	Under the assumptions of Proposition \ref{estime L 2 m g eps}, we consider 
	$\ds(k\,,\,k_1\,,\,k_2)$ in $\N^{\,3}$ such that 
	$\ds
	k_1\,+\,k_2\,=\,k$ and
	$\ds
	k\,\leq\,2$.
	there exists a
	constant $C>0$, such that, for~all~$\eps\,>\,0$, we have
	\[
	\left.\\[0,5em]
	\left\|\,
	\left(w^{k_1}\,\partial_{w}^{\,k_2}\right)\,
	\nu^\eps(t\,,\,\bx)\,
	\right\|_{L^2\left(m^\eps_{\bx}\right)}
	\,\,\leq\,\,
	\left(
	\frac{2}{\kappa}
	\right)^{k_1}\,e^{Ct}\,
	\left\|\,\partial_{w}^{\,
		k}\,
	\nu^\eps_0(\bx)\,
	\right\|
	_{L^2
		\left(
		m^\eps_{\bx}\right)}\,
	, 
	\quad\forall \,(t,\bx)\, \in\,\R^+\times K\,,
	\right.\\[0,3em]
	\]
\end{corollary}
\begin{proof}
	We consider 
	$\ds(k\,,\,k_1\,,\,k_2)$ in $\N^{\,3}$ such that 
	$\ds
	k_1\,+\,k_2\,=\,k$ and
	$\ds
	k\,\leq\,2$ and point out that according to Lemma \ref{P ineq}, we have
	\[
	\left.\\[0,5em]
	\left\|\,
	\left(w^{k_1}\,\partial_{w}^{\,k_2}\right)\,
	\nu^\eps(t\,,\,\bx)\,
	\right\|_{L^2\left(m^\eps_{\bx}\right)}
	\,\,\leq\,\,
	\left(
	\frac{2}{\kappa}
	\right)^{k_1}\,
	\left\|\,\partial_{w}^{\,k
	}\,
	\nu^\eps(t\,,\,\bx)\,
	\right\|_{L^2\left(m^\eps_{\bx}\right)}\,
	.
	\right.\\[0,3em]
	\]
	Consequently, we obtain the result applying Proposition \ref{estime L 2 m g eps}.\\
\end{proof}

We conclude this section with providing regularity estimates for the limiting distribution $\ols{\nu}$ with respect to the adaptation variable, which solves \eqref{bar nu:eq}. 
The proof for this result is mainly computational since we have an explicit formula for the solutions to equation \eqref{bar nu:eq}.
\begin{lemma}\label{lemme:bar nu}
	Consider some index $k$ lying in $\{0\,,\,1\}$ and some $\ols{\nu}_0$ lying in 
	$
	\displaystyle
	\scH^{k}
	\left(
	\ols{m}
	\right)
	$. The solution $\ols{\nu}$ to equation \eqref{bar nu:eq} with initial condition 
	$
	\ds\ols{\nu}_0
	$ verifies
	\[
	\left\|\,
	\ols{\nu}(t)\,
	\right\|_
	{
		\scH^{k}
		\left(
		\ols{m}
		\right)
	}
	\,\leq\,
	\exp{
		\left(
		\left(k\,+\,\frac{1}{2}
		\right)
		\,b\,t
		\right)}\,
	\left\|\,
	\ols{\nu}_0\,
	\right\|_
	{
		\scH^{k}
		\left(
		\ols{m}
		\right)
	}\,,
	\quad\forall \,t \in\R^+\,.
	\]
\end{lemma}
\begin{proof}
	Since $\ols{\nu}$ solves \eqref{bar nu:eq}, it is given by the following formula
	\[
	\ols{\nu}_{t,\,\bx}(w)
	\,\,=\,\,
	e^{bt}\,
	\ols{\nu}_{0,\,\bx}
	\left(\,e^{bt}\,w
	\right)\,
	, 
	\quad\forall \,(t,\bx)\, \in\,\R^+\times K\,.
	\]
	Consequently, we easily obtain the expected result.
\end{proof}
\subsection{Proof of Theorem \ref{th:2}}\label{sec:proof th21}
In the forthcoming analysis we quantify the convergence of $\nu^\eps$ towards the asymptotic profile $\nu$ given by
\[
\nu\,=\,
\cM_{\rho_0^\eps}\otimes \bar{\nu}\,,
\]
in the functional spaces $\ds
\scH^{k}
\left(m ^\eps
\right)$. We introduce the orthogonal projection of $\nu^\eps$ onto the space of function with marginal 
$
\ds
\cM_{\rho_0^\eps}
$ with respect to the voltage variable
\[
\Pi\,\nu^\eps
\,=\,
\cM_{\rho_0^\eps}\otimes \bar{\nu}^\eps\,.
\]
Furthermore, we consider the orthogonal component $\ds\,\nu^\eps_{\bot}\,$ of $\nu^\eps$ with respect to the latter projection
\[
\nu^\eps_{\bot}
\,=\,
\nu^\eps
\,-\,
\Pi\,\nu^\eps\,.
\]
With these notations we have 
\[
\left\|\,
\nu^\eps
\,-\,
\nu\,
\right\|_{\scH^{k}
	\left(
	m^\eps
	\right)}^2
\,=\,
\left\|\,
\nu^\eps_{\bot}\,
\right\|_{\scH^{k}
	\left(
	m^\eps
	\right)}^2
\,+\,
\left\|\,
\ols{\nu}^\eps
\,-\,
\ols{\nu}\,
\right\|_{\scH^{k}
	\left(
	\ols{m}
	\right)}^2\,,
\]
for $k$ in 
$\ds 
\left\{
0\,,\,1
\right\}$.
Therefore, we prove that $\ds\,\nu^\eps_{\bot}\,$ and
$\ds\,
\ols{\nu}^\eps
\,-\,
\ols{\nu}\,$ vanish
as $\eps$ goes to zero in both 
$\ds \scH^0$ and
$\ds \scH^1$.\\

\subsubsection{
	Estimates for
	$\,
	\nu^\eps_{\bot}\,
	$}

\noindent Our strategy relies on the same arguments as the ones we developed in the former section to prove Proposition \ref{estime L 2 m g eps}. Indeed, the equation satisfied by 
$
\nu^\eps_{\bot}$ is the same equation as equation \eqref{nu:eq} solved by $\nu^\eps$ with additional source terms. It reads as follows
\begin{equation}\label{nu orth:eq}
	\partial_t\, \nu^\eps_{\bot}
	\,+\,
	\scA^\eps
	\left[
	\,\nu^\eps_{\bot}\, \right]
	\,=\,
	\scS
	\left[
	\,\nu^\eps\,,\,\Pi\,\nu^\eps\,
	\right]\,,
\end{equation}
where the operator $\scA^\eps$ is given by
\eqref{op A} and the source terms are given by 
\begin{equation*}
	\scS^\eps
	\left[\,
	\nu^\eps\,,\,
	\Pi\,\nu^\eps\,
	\right]
	\,=\,
	\partial_{w}
	\left[\,
	a\,\theta^\eps
	\int
	v\,\nu^\eps\,\dD v\,
	\mathcal{M}_{\rho_0^\eps}
	\,-\,
	b\,w \,\Pi\,\nu^\eps\,\right]
	\,-\,
	\mathrm{div}_{\bu}
	\left[
	\,\mathbf{b}^\eps_0\,
	\Pi\,\nu^\eps
	\, \right]\,.
\end{equation*}
Consequently, our strategy consists in estimating the source terms using the regularity estimates provided by Proposition \ref{estime L 2 m g eps}. Then, we adapt our analysis to the case of 
$\displaystyle \partial_w\,\nu^\eps_{\bot}$, which we write
$\ds h^\eps_{\bot}
\,=\,
\partial_w\,\nu^\eps_{\bot}$ and which solves again the same equation up to extra terms that add no difficulty
\begin{equation}\label{h orth:eq}
	\partial_t\, h^\eps_{\bot}
	\,+\,
	\scA^\eps
	\left[
	\,h^\eps_{\bot}\, \right]
	\,=\,
	\scS
	\left[
	\,h^\eps\,,\,\Pi\,h^\eps\,
	\right]
	\,+\,
	b
	\,
	h^\eps_{\bot}
	\,+\,
	\frac{1}{\theta^\eps}
	\,
	\partial_v\,
	\nu^\eps
	\,
	,
\end{equation}
where we used the notation 
$\displaystyle
h^\eps\,=\,
\partial_w\,\nu^\eps$.
\begin{proposition}\label{estimee:nu orth}
	Under assumptions \eqref{hyp1:N}-\eqref{hyp2:N} and \eqref{hyp3:N} on the drift $N$ and \eqref{hyp2:psi} on the interaction kernel $\Psi$, consider a sequence of solutions $(\mu^\eps)_{\eps>0}$ to \eqref{kinetic:eq} with initial conditions satisfying assumptions \eqref{hyp:rho0}-\eqref{hyp2:f0} and \eqref{hyp4:f0} with an index $k$ in $\{0\,,\,1\}$ and an exponent $\kappa$ greater than $1/(2b)$. Then,
	for all $\eps$ between $0$ and $1$
	and any
	$\alpha_*$ lying in $\displaystyle 
	\left(\,
	0\,,\,1-(2b\kappa)^{-1}\,
	\right)\,
	$, there exists a constant $C>0$, independent of $\eps$, such that
		\begin{equation*}
			\left.\\[0,3em]
			\|\nu^\eps_{\bot}(t)\|_{
				\scH^{k}
				\left(
				m^\eps
				\right)
			}
			\,\leq\,
			e^{Ct}\,\left\|\,
			\nu^\eps_0\,
			\right\|_
			{
				\scH^{k\,+\,1}
				\left(
				m^\eps
				\right)
			}
			\,
			\left(
			\,C\,
			\sqrt{\eps}
			\,+\,
			\min\left\{
			\,1\,,\,
			e^{
				\,
				-\,\alpha_*t\,/\,\eps\,
			}
			\,
			\eps^{-\,\alpha_*\,/\,
				(2\,m_*)}
			\right\}
			\right)
			\,,
			\quad\forall \;t \in\R^+\,.
			\right.\\[0,3em]
		\end{equation*}
	
\end{proposition}
\begin{proof}
	We first treat the case $k=0$
	.
	All along this step of the proof, we consider some $\eps\,>\,0$ and some $\displaystyle(t,\,\bx)$ in $\R_+\times K$; we omit the dependence with respect to $(t,\bx)$ when the context is clear. We compute the time derivative of 
	$
	\displaystyle
	\|\,\nu^\eps_{\bot}\,\|^2_{L^2\left(m ^\eps\right)}$ multiplying equation \eqref{nu orth:eq} by 
	$
	\displaystyle 
	\nu^\eps_{\bot} \, m^\eps
	$ and integrating with respect to $\bu$
	\begin{equation*}
		\left.\\[0,3em]
		\frac{1}{2}
		\frac{\dD}{\dD t}\,
		\|\nu^\eps_{\bot}\|_{L^2\left(m ^\eps\right)}^2
		\,=\,
		\left\langle\,
		\scS^\eps
		\left[\,
		\nu^\eps\,,\,
		\Pi\,\nu^\eps\,
		\right]\,,\,
		\nu^\eps_{\bot}\,
		\right\rangle_{L^2\left(m ^\eps\right)}
		\,-\,
		\left\langle\,
		\scA^\eps
		\left[\,
		\nu^\eps_{\bot}\,
		\right]\,,\,
		\nu^\eps_{\bot}\,
		\right\rangle_{L^2\left(m ^\eps\right)}
		\,,
		\right.\\[0,3em]
	\end{equation*}
	and we split the contribution of the source terms as follows
	\[
	\left\langle\,
	\scS^\eps
	\left[\,
	\nu^\eps\,,\,
	\Pi\,\nu^\eps\,
	\right]\,,\,
	\nu^\eps_{\bot}\,
	\right\rangle_{L^2\left(m ^\eps\right)}
	\,=\,
	\mathcal{A}_1 
	\,+\,
	\mathcal{A}_2
	\,,
	\]
	where 
	the terms $\cA_1$ and $\cA_2$ are given by\\
	\begin{equation*}
		\left\{
		\begin{array}{l}
			\displaystyle  \cA_{1} \,=\,
			-\,
			\frac{1}{\theta^\eps}
			\int_{\R^2}
			\partial_v
			\left[\,
			B^\eps_0
			\left(
			t\,,\,\bx\,,\,\theta^{\eps}\,v\,,\,w
			\right)
			\Pi\,\nu^\eps\,\right]\,
			\nu^\eps_{\bot}\,m ^\eps\, \dD \bu\,
			,\\[1.1em]
			\displaystyle  \cA_{2} \,=\,
			-
			a\,\theta^\eps\,\int_{\R^2}
			v\,
			\partial_w
			\left[\,
			\Pi\,\nu^\eps\,
			\right]\,
			\nu^\eps_{\bot}\,
			m ^\eps\, \dD \bu\,
			,
		\end{array}
		\right.\\[0.8em]
	\end{equation*}
	where $B^\eps_0$ is given by \eqref{def:b0}.\\
	
	Let us estimate $\cA_1$. After an integration by part, this term rewrites as follows
	\[
	\cA_1
	\,=\,
	\int_{\R^2}
	B^\eps_0
	\left(
	t\,,\,\bx\,,\,\theta^{\eps}\,v\,,\,w
	\right)\,
	\Pi\,\nu^\eps\,
	\left(m ^\eps\right)^{\frac{1}{2}}\,
	\frac{1}{\theta^\eps}\,
	\partial_v
	\left[\,
	\nu^\eps_{\bot}\,m ^\eps
	\,\right]\,\left(m ^\eps\right)^{-\frac{1}{2}}\, \dD \bu\,.\]
	According to items \eqref{estimate moment mu} and \eqref{estimate error} in Proposition  \ref{th:preliminary} , $\mathcal{V}^\eps$ and 
	$
	\ds
	\mathcal{E}
	\left(
	\mu^\eps
	\right)
	$
	are uniformly bounded with respect to both
	$
	\left(
	t,\,\bx
	\right)
	\in
	\R^+
	\times 
	K
	$ and $\eps$. Moreover, according to assumptions \eqref{hyp2:psi} and \eqref{hyp:rho0}, 
	$
	\ds
	\Psi*_r\rho_0^\eps
	$ stays uniformly bounded with respect to both
	$
	\bx
	\in 
	K
	$ and $\eps$ as well.
	Consequently, applying Young's inequality to the former relation and using assumption \eqref{hyp2:N}, which ensures that $N$ has polynomial growth, we obtain
	\[
	\cA_1
	\,\leq\,
	\frac{\eta}{
		\left(
		\theta^{\eps}
		\right)^2}\,
	\cD_{\rho_0^\eps}
	\left[\,
	\nu^\eps_{\bot}\,
	\right]
	\,+\,
	\frac{C}{\eta}\,
	\int_{\R^2}
	\left(\,
	\left|
	\theta^\eps v
	\right|^{2p}
	\,+\,|w|^2
	\,+\,1\,
	\right)
	\,
	\left|\ols{\nu}^\eps
	\right|^2\,
	\mathcal{M}_{\rho_0^\eps}\,
	\ols{m}\,
	\dD \bu\,,\]
	for all positive $\eta$ and for some positive constant $C$ which only depends on $m_*$,~$m_p$~and~$\ols{m}_p$ and the data of the problem $N$,~$\Psi$~and~$A_0$. Taking advantage of the properties of the Maxwellian
	$\ds
	\mathcal{M}_{\rho_0^\eps}
	$ and since $\rho^\eps_0$ meets assumption \eqref{hyp:rho0}, the latter estimate simplifies into
	\[
	\cA_1
	\,\leq\,
	\frac{\eta}{
		\left(
		\theta^{\eps}
		\right)^2}\,
	\cD_{\rho_0^\eps}
	\left[\,
	\nu^\eps_{\bot}\,
	\right]
	\,+\,
	\frac{C}{\eta}
	\left(\,
	\left\|\,
	\ols{\nu}^\eps\,
	\right\|^2_{L^2(\ols{m})}
	\,+\,
	\left\|\,
	w\,
	\ols{\nu}^\eps\,
	\right\|^2_{L^2(\ols{m})}\,
	\right).\]
	Furthermore, according to Jensen's inequality, it holds
	\[
	\left\|\,
	\ols{\nu}^\eps\,
	\right\|^2_{L^2(\ols{m})}
	\,+\,
	\left\|\,
	w\,
	\ols{\nu}^\eps\,
	\right\|^2_{L^2(\ols{m})}
	\,\leq\,
	\left\|\,
	\nu^\eps\,
	\right\|^2_{L^2\left(m ^\eps\right)}
	\,+\,
	\left\|\,
	w\,
	\nu^\eps\,
	\right\|^2_{L^2\left(m ^\eps\right)}\,.
	\]
	Therefore, we apply Corollary \ref{estime mixed quantities} and obtain the following estimate for $\cA_1$
	\[
	\cA_1
	\,\leq\,
	\frac{\eta}{
		\left(
		\theta^{\eps}
		\right)^2}\,
	\cD_{\rho_0^\eps}
	\left[\,
	\nu^\eps_{\bot}\,
	\right]
	\,+\,
	\frac{C}{\eta}\,
	e^{Ct}\,
	\left\|\,
	\nu^\eps_0\,
	\right\|^2_
	{
		H^{1}_{w}
		\left(
		m^\eps
		\right)
	}\,.\]

	To estimate $\cA_2$, we apply the Cauchy-Schwarz inequality, use the properties of the Maxwellian 
	$\ds
	\cM_{\rho_0^\eps}
	$ and assumption \eqref{hyp:rho0}. It yields
	\[
	\cA_{2}
	\,\leq\,
	C\,
	\left(
	\left\|\,
	\partial_w\,
	\ols{\nu}^\eps\,
	\right\|^2_
	{
		L^{2}
		\left(
		m^\eps
		\right)
	}
	\,+\,
	\left\|\,
	\nu^\eps_{\bot}\,
	\right\|^2_
	{
		L^{2}
		\left(
		m^\eps
		\right)
	}
	\right)\,.
	\]
	According to the same remark as in the former step, it holds
	\[
	\left\|\,
	\partial_w\,
	\ols{\nu}^\eps\,
	\right\|^2_
	{
		L^{2}
		\left(
		\ols{m}
		\right)
	}
	\,+\,
	\left\|\,
	\nu^\eps_{\bot}\,
	\right\|^2_
	{
		L^{2}
		\left(
		m^\eps
		\right)
	}
	\,\leq\,
	\left\|\,
	\nu^\eps\,
	\right\|^2_{L^2\left(m ^\eps\right)}
	\,+\,
	\left\|\,
	\partial_w\,
	\nu^\eps\,
	\right\|^2_{L^2\left(m ^\eps\right)}\,,
	\]
	hence, applying Proposition \ref{estime L 2 m g eps}, we obtain
	\[
	\cA_{2}
	\,\leq\,
	C\,
	e^{Ct}\,
	\left\|\,
	\nu^\eps_0\,
	\right\|^2_
	{
		H^{1}_{w}
		\left(
		m^\eps
		\right)
	}\,.\\[0.5em]
	\]
	
	To evaluate the contribution of $\scA^\eps$ we replace it  by its definition \eqref{op A} and integrate by part the stiffer term. It yields
	\[
	-\,
	\left\langle\,
	\scA^\eps
	\left[\,
	\nu^\eps_{\bot}\,
	\right]\,,\,
	\nu^\eps_{\bot}\,
	\right\rangle_{L^2\left(m ^\eps\right)}
	\,=\,
	-\,
	\frac{1}{
		\left(
		\theta^\eps
		\right)^2}
	\cD_{\rho_0^\eps}
	\left[\,\nu^\eps_{\bot}\,
	\right]
	\,-\,
	\left\langle\,
	\mathrm{div}_{\bu}
	\left[\,
	\mathbf{b}^\eps_0\,
	\nu^\eps_{\bot}\,
	\right]\,,\,
	\nu^\eps_{\bot}\,
	\right\rangle_{L^2\left(m ^\eps\right)}\,.
	\]
	In order to close the estimate, we apply Lemma \ref{lemme technique L 2 m} and Proposition \ref{estime L 2 m g eps} to control the term associated to linear transport. It yields\\
	\[
	\left.
	-
	\left\langle\,
	\mathrm{div}_{\bu}
	\left[\,
	\mathbf{b}^\eps_0\,
	\nu^\eps_{\bot}\,
	\right]\,,\,
	\nu^\eps_{\bot}\,
	\right\rangle_{L^2\left(m ^\eps\right)}
	\,\,\leq\,\,
	\frac{
		\alpha
	}{
		\left(
		\theta^\eps
		\right)^2
	}\,
	\cD_{\rho_0^\eps}
	\left[\,\nu^\eps_{\bot}\,
	\right]
	\,
	+\,
	C\,
	e^{Ct}\,
	\left\|\,
	\nu^\eps_{0} \,
	\right\|^2
	_{L^2
		\left(
		m ^\eps\right)}\,
	,
	\right.\\[0,3em]
	\]
	for all positive constant $\alpha$ greater than 
	$\ds 1/(2b\kappa)$. Consequently, the former computations lead to the following differential inequality
	\begin{equation*}
		\left.\\[0,3em]
		\frac{1}{2}
		\frac{\dD}{\dD t}\,
		\|\nu^\eps_{\bot}\|_{L^2\left(m ^\eps\right)}^2
		\,+\,
		\frac{1-\alpha-\eta}{
			\left(
			\theta^\eps
			\right)^2
		}\,
		\cD_{\rho_0^\eps}
		\left[\,\nu^\eps_{\bot}\,
		\right]
		\,\leq\,
		\frac{C}{\eta}\,
		e^{Ct}\,\left\|\,
		\nu^\eps_0\,
		\right\|^2_
		{
			H^{1}_{w}
			\left(
			m^\eps
			\right)
		}\,.
		\right.\\[0,3em]
	\end{equation*}
	Based on our assumptions, it holds
	$
	\ds
	1/(2b\kappa)<1
	$ therefore we may choose $\alpha$ and $\eta$ such that 
	$\ds\,\alpha_*\,=\,
	1-\alpha-\eta\,$ lies in 
	$\displaystyle 
	\left]\,
	0\,,\,1-(2b\kappa)^{-1}\,
	\right[\,
	$.
	Furthermore, in our context, the Gaussian-Poincar\'e inequality \cite{Dolbeault and Volzone} rewrites
	\[
	\|\nu^\eps_{\bot}\|_{L^2\left(m ^\eps\right)}^2
	\,\leq\,
	\cD_{\rho_0^\eps}
	\left[\,\nu^\eps_{\bot}\,
	\right]\,.
	\]
	According to the latter remarks, the former inequality rewrites
	\begin{equation*}
		\left.\\[0,3em]
		\frac{\dD}{\dD t}\,
		\|\nu^\eps_{\bot}\|_{L^2\left(m ^\eps\right)}^2
		\,+\,
		\frac{2\,\alpha_*}{
			\left(
			\theta^\eps
			\right)^2
		}\,
		\|\nu^\eps_{\bot}\|_{L^2\left(m ^\eps\right)}^2
		\,\leq\,
		C\,
		e^{Ct}\,\left\|\,
		\nu^\eps_0\,
		\right\|^2_
		{
			H^{1}_{w}
			\left(
			m^\eps
			\right)
		}\,.
		\right.\\[0,3em]
	\end{equation*}
	We multiply this estimate by
	\[
	\displaystyle
	\exp{
		\left(
		2\,\alpha_*\,
		\int_0^t
		\frac{1}{
			\left(
			\theta^\eps
			\right)^2
		}\,\dD s
		\right)
	}\,,
	\]
	and integrate between $0$ and $t$. In the end, we deduce the following inequality
	\begin{equation*}
		\left.\\[0,3em]
		\|\,\nu^\eps_{\bot}\,\|_{L^2\left(m ^\eps\right)}^2
		\,\leq\,
		\|\,\nu^\eps_{0}\,\|_{H^1_w\left(m ^\eps\right)}^2\,
		e^{
			-2\,\alpha_*\,
			I(t)
		}
		\left(
		1
		\,+\,
		C
		\int_0^t
		e^{Cs}\,
		e^{
			2\,\alpha_*\,
			I(s)
		}\,
		\dD s
		\right)
		\,,
		\right.\\[0,3em]
	\end{equation*}
	where $I$ is given by
	\[
	I(t)
	\,=\,
	\int_0^t
	\frac{1}{
		\left(
		\theta^\eps
		\right)^2
	}\,\dD s\,.
	\]
	Taking advantage of the ODE solved by $\theta^\eps$ (see \eqref{expression for theta eps}), we compute explicitly $I$
	\[
	I(t)\,=\,
	\frac{t}{\eps}
	\,+\,
	\frac{1}{2\,\rho_0^\eps}
	\ln{
		\left(\,
		\theta^\eps(t)^2\,
		\right)
	}\,.
	\]
	Consequently, the latter estimate rewrites
	\begin{equation*}
		\left.\\[0,3em]
		\|\,\nu^\eps_{\bot}\,\|_{L^2\left(m ^\eps\right)}^2
		\,\leq\,
		\|\,\nu^\eps_{0}\,\|_{H^1_w\left(m ^\eps\right)}^2\,
		e^{
			-2\,\alpha_*\,
			\frac{t}{\eps}}
		\left(\left(
		\theta^{\eps}(t)
		\right)^{
			-2\,\frac{\alpha_*}{\rho_0^\eps}
		}
		+
		C
		\int_0^t
		e^{Cs}\,
		e^{
			2\,\alpha_*\,
			\frac{s}{\eps}}
		\,
		\left(
		\frac{\theta^{\eps}(s)}{\theta^{\eps}(t)}
		\right)^{2\,
			\frac{\alpha_*}{\rho_0^\eps}
		}\,
		\dD s
		\right)
		\,.
		\right.\\[0,3em]
	\end{equation*}
	Then we notice that according to the explicit formula \eqref{expression for theta eps} for $\theta^\eps$, given that $\eps$ lies in 
	$(0\,,\,1)\,$, we have on the one hand
	\[
	e^{
		-2\,\alpha_*\,
		\frac{t}{\eps}}
	\,
	\left(
	\theta^{\eps}(t)
	\right)^{
		-2\,\frac{\alpha_*}{\rho_0^\eps}
	}
	\,\leq\,
	\min{
		\left(
		1\,,\,
		e^{
			-2\,\alpha_*\,
			\frac{t}{\eps}}\,
		\eps^{
			-\,
			\frac{\alpha_*}{\rho_0^\eps}}
		\right)
	}\,,
	\]
	and on the other hand
	\[
	\left(
	\frac{\theta^{\eps}(s)}
	{\theta^{\eps}(t)}
	\right)^{2\,
		\frac{\alpha_*}{\rho_0^\eps}
	}
	\,\leq\,
	\min{
		\left(
		2\,e^{
			2\,\alpha_*\,
			\frac{t-s}{\eps}}\,,\,
		C
		\left(1\,+\,
		e^{
			-2\,\alpha_*\,
			\frac{s}{\eps}}\,
		\eps^{
			-\,
			\frac{\alpha_*}{\rho_0^\eps}}
		\right)
		\right)
	}\,.
	\]
	We inject these bounds and take the supremum over all $\bx$ in $K$ in the latter estimate. In the end, we obtain the estimate for the case where $k=0$
	in Proposition \ref{estimee:nu orth}.\\
	
	We turn to the case $k=1$ in
	Proposition \ref{estimee:nu orth}. We make use of the shorthand notation $\ds h^\eps_{\bot}
	\,=\,
	\partial_w\,\nu^\eps_{\bot}$.
	We compute the time derivative of 
	$
	\displaystyle
	\|\,h^\eps_{\bot}\,\|^2_{L^2\left(m ^\eps\right)}$ multiplying equation \eqref{h orth:eq} by 
	$
	\displaystyle 
	h^\eps_{\bot} \, m^\eps
	$ and integrating with respect to $\bu$\\
	\begin{equation*}
		\frac{1}{2}
		\frac{\dD}{\dD t}\,
		\|\,h^\eps_{\bot}\,\|_{L^2\left(m ^\eps\right)}^2
		\,=\,
		\left\langle\,
		\scS^\eps
		\left[\,
		h^\eps\,,\,
		\Pi\,h^\eps\,
		\right]\,,\,
		h^\eps_{\bot}\,
		\right\rangle_{L^2\left(m ^\eps\right)}
		\,-\,
		\left\langle\,
		\scA^\eps
		\left[\,
		h^\eps_{\bot}\,
		\right]\,,\,
		h^\eps_{\bot}\,
		\right\rangle_{L^2\left(m ^\eps\right)}
		\,+\,b\,
		\|\,h^\eps_{\bot}\,\|_{L^2\left(m ^\eps\right)}^2
		\,+\,\mathcal{A}
		\,,\\[0.7em]
	\end{equation*}
	where $\cA$ is given by\\
	\begin{equation*}
		\displaystyle  \cA \,=\,
		\frac{1}{\theta^\eps}
		\int_{\R^2}
		\partial_v\,
		\nu^\eps\,
		h^\eps_{\bot}\,m ^\eps\, \dD \bu\,
		.\\[0.5em]
	\end{equation*}
	We estimate $\cA$ integrating by part with respect to $v$ and applying Young's inequality.
	After applying Proposition \ref{estime L 2 m g eps}, it yields
	\begin{equation*}
		\displaystyle  \cA \,\leq\,
		\frac{1}{4\,\eta}\,e^{Ct}\,
		\|\,
		\nu^\eps_0\,
		\|_{L^2\left(m ^\eps\right)}^2
		+\,
		\frac{\eta}{
			\left(
			\theta^\eps
			\right)^2
		}\,
		\cD_{\rho_0^\eps}
		\left[\,
		h^\eps_{\bot}\,
		\right]
		\,
		,
	\end{equation*}
	for some positive constant $C$ and all positive $\eta$. Then we follow the same argument as in the last step and obtain the expected result.\\
\end{proof}

\subsubsection{
	Estimate for
	$
	\left(\,
	\ols{\nu}^\eps
	\,-\,
	\ols{\nu}\,
	\right)$}

It solves the following equation
\begin{equation}
	\label{bar nu eps bar nu:eq}
	\left.\\[0,3em]
	\displaystyle \partial_t \left(
	\,
	\ols{\nu}^\eps
	\,-\,
	\ols{\nu}\,
	\right)
	\,-\,b\,
	\partial_{w}
	\left[\,
	w\, \left(
	\,
	\ols{\nu}^\eps
	\,-\,
	\ols{\nu}\,
	\right)\,
	\right]\,=\,
	-a\,\theta^\eps\,\partial_{w}
	\left[\,
	\int_\R v\, \nu^\eps\,\dD v\,
	\right]\,,
	\right.\\[0,3em]
\end{equation} 
obtained taking the difference between equations \eqref{bar nu:eq} and \eqref{bar:nu eps:eq}.
It is the same equation as \eqref{bar nu:eq} solved by $\ols{\nu}$ with the additional source term on the right-hand side of the latter equation. Consequently, our strategy consists in estimating the source term. We point out that since the source term is weighted by $\theta^\eps$, it sufficient to prove that it is bounded in order to obtain convergence. However, it is not hard to check that the source term cancels if we replace $\nu^\eps$ by its projection $\Pi\,\nu^\eps$
\[
\left.\\[0,3em]
\partial_{w}
\left[\,
\int_\R v\, \Pi\,\nu^\eps\,\dD v\,
\right]\,=\,0\,.
\right.\\[0,3em]
\]
Consequently, based on the estimates obtained in the first step on $\nu^\eps_{\bot}$ (see Proposition \ref{estimee:nu orth}), we expect
\[
\left.
\theta^\eps\,\partial_{w}
\left[\,
\int_\R v\, \nu^\eps\,\dD v\,
\right]\,
\underset{\eps \rightarrow 0}{=}\,
O
\left(
\eps
\right).
\right.\\[0,3em]
\]
This formal approach was already rigorously justified in a weak convergence setting in \cite{BF}. In our setting and
due to the structure of the source term, we use regularity estimates to achieve the latter convergence rate.
\begin{lemma}\label{lemme technique nu bar}
	Consider a sequence of solutions $(\mu^\eps)_{\,\eps\,>\,0}$ to \eqref{kinetic:eq} with initial conditions satisfying assumption \eqref{hyp4:f0} with an index $k$ in $\ds\{0\,,\,1\}$
	as well as the solution $\ols{\nu}$ to equation \eqref{bar nu:eq} with some initial condition $\ols{\nu}_0$ lying in
	$
	\displaystyle
	\scH^{k}(\ols{m})
	$. The following estimate holds for all positive $\eps$
	\begin{equation*}
		\|\,
		\ols{\nu}^\eps
		\,-\,
		\ols{\nu}\,\|_{H^{k}(\ols{m})}
		\,\leq\,
		e^{C\,t}
		\left(
		\|\,
		\ols{\nu}^\eps_0
		\,-\,
		\ols{\nu}_0\,\|_{H^{k}(\ols{m})}
		\,+\,
		C\,
		\int_0^t
		\,e^{-C\,s}
		\,\theta^\eps\,
		\|\,\nu^\eps_{\bot}\|_{H^{k+1}_w\left(m ^\eps\right)}
		\,\dD s\,
		\right)
		, 
	\end{equation*}
	for all $(t,\bx)\in \R^+\times K$,
	where $k$ lies in $\{0\,,\,1\}$ and
	for some positive constant $C$ only depending on $\kappa$, $m_*$ and $A$.
\end{lemma}

\begin{proof}
	We start with the case $k=0$. We consider some $\eps\,>\,0$ and some $\displaystyle(t,\,\bx)$ in $\R_+\times K$; we omit the dependence with respect to $(t,\bx)$ when the context is clear. We compute the time derivative of 
	$
	\displaystyle
	\|\,
	\ols{\nu}^\eps
	\,-\,
	\ols{\nu}\,\|^2_{L^2(\ols{m})}$ multiplying equation \eqref{bar nu eps bar nu:eq} by 
	$
	\displaystyle 
	\left(
	\,
	\ols{\nu}^\eps
	\,-\,
	\ols{\nu}\,
	\right) \, \ols{m}
	$ and integrating with respect to $w$. We integrate by part the term associated to linear transport and end up with the following relation
	\begin{equation*}
		\frac{1}{2}
		\frac{\dD}{\dD t}\,
		\|\,
		\ols{\nu}^\eps
		\,-\,
		\ols{\nu}\,\|_{L^2(\ols{m})}^2
		=
		\frac{b}{2}\,\int_{\R}
		\left(\,1\,-\,
		\kappa\,w^2\,
		\right)
		\left|
		\,
		\ols{\nu}^\eps
		\,-\,
		\ols{\nu}\,
		\right|^2\,\ols{m}\, \dD w
		\,-\,
		a\,\theta^\eps\int_{\R^2}
		v\,
		\partial_w\,
		\nu^\eps_{\bot}
		\left(
		\ols{\nu}^\eps
		\,-\,
		\ols{\nu}\right)\,\ols{m}\, \dD \bu\,
		.
	\end{equation*}
	According to Cauchy-Schwarz inequality and applying assumption $\eqref{hyp:rho0}$, the source term admits the bound
	\[
	-\,
	a\,\theta^\eps\int_{\R^2}
	v\,
	\partial_w\,
	\nu^\eps_{\bot}
	\left(
	\ols{\nu}^\eps
	\,-\,
	\ols{\nu}\right)\,\ols{m}\, \dD \bu
	\,\leq\,C\,
	\theta^\eps\,
	\|\,h^\eps_{\bot}\,\|_{L^2\left(m ^\eps\right)}
	\,
	\|\,
	\ols{\nu}^\eps
	\,-\,
	\ols{\nu}\,\|_{L^2(\ols{m})}\,,
	\]
	for some positive constant $C$ only depending on $A$ and $m_*$. Furthermore, we bound the term associated to linear transport using that the polynomial 
	$
	\ds
	1\,-\,\kappa\,w^2
	$ is upper-bounded over $\R$. Gathering the former considerations we end up with the following differential inequality
	\begin{equation*}
		\frac{1}{2}\,
		\frac{\dD}{\dD t}\,
		\|\,
		\ols{\nu}^\eps
		\,-\,
		\ols{\nu}\,\|_{L^2(\ols{m})}^2
		\,\leq\,
		C\,
		\left(
		\|\,
		\ols{\nu}^\eps
		\,-\,
		\ols{\nu}\,\|_{L^2(\ols{m})}^2
		\,+\,
		\theta^\eps\,
		\|\,\partial_w\,
		\nu^\eps_{\bot}\,\|_{L^2\left(m ^\eps\right)}
		\,
		\|\,
		\ols{\nu}^\eps
		\,-\,
		\ols{\nu}\,\|_{L^2(\ols{m})}
		\right)
		,
	\end{equation*}
	for some positive constant $C$ only depending on $\kappa$, $m_*$ and $A$. we divide the latter inequality by
	$
	\|\,
	\ols{\nu}^\eps
	\,-\,
	\ols{\nu}\,\|_{L^2(\ols{m})}
	$ and obtain
	\begin{equation*}
		\frac{\dD}{\dD t}\,
		\|\,
		\ols{\nu}^\eps
		\,-\,
		\ols{\nu}\,\|_{L^2(\ols{m})}
		\,\leq\,
		C\,
		\left(
		\|\,
		\ols{\nu}^\eps
		\,-\,
		\ols{\nu}\,\|_{L^2(\ols{m})}
		\,+\,
		\theta^\eps\,
		\|\,\partial_w\,
		\nu^\eps_{\bot}\,\|_{L^2\left(m ^\eps\right)}
		\right)
		,
	\end{equation*}
	We conclude this step applying Gronwall's Lemma.\\
	
	We treat the case $k\,=\,1$ applying the same method. Indeed, 
	$\ds
	\,
	\ols{\nu}^\eps
	\,-\,
	\ols{\nu}\,
	$ and 
	$\ds
	\,
	\partial_w
	\left(
	\ols{\nu}^\eps
	\,-\,
	\ols{\nu}
	\right)\,
	$
	solve the same equation up to an additional source term which adds no difficulty.\\
\end{proof}
\begin{proposition}\label{estimee:bar nu bar nu eps}
	Under assumptions \eqref{hyp1:N}-\eqref{hyp2:N} and \eqref{hyp3:N} on
	the drift $N$ and \eqref{hyp2:psi} on the interaction kernel $\Psi$,
	consider a sequence of solutions $(\mu^\eps)_{\eps>0}$ to
	\eqref{kinetic:eq} with initial conditions satisfying assumptions
	\eqref{hyp:rho0}-\eqref{hyp2:f0} as well as the solution $\ols{\nu}$
	to equation \eqref{bar nu:eq} with some initial condition
	$\ols{\nu}_0$ and an exponent $\kappa$ greater than $1/(2b)$. Then
	there exists a constant $C>0$,  such that for all $\eps\in(0,1)$,
	the following statements hold
	\begin{enumerate}
		\item\label{item 1 cv bar nu eps} suppose that assumptions
		\eqref{hyp4:f0}-\eqref{hyp bar nu th2} are fulfilled  with an
		index $k$ in $\ds\{0\,,\,1\}$, then for all $t \geq0$,
		\begin{equation*}
			\|\,\ols{\nu}^\eps(t)\,-\,\ols{\nu}(t)\,\|_{\scH^{\,k}(\ols{m})}
			\,\leq\,
			e^{Ct}\,
			\left(
			\|\ols{\nu}^\eps_0\,-\,\ols{\nu}_0\|_{\scH^{\,k}(\ols{m})}
			\,+\,
			C\left\|\,
			\nu^\eps_0\,
			\right\|_
			{
				\scH^{\,k+1}
				\left(
				m^\eps
				\right)
			}\sqrt{\eps}\,
			\right)\,;
		\end{equation*}
		\item\label{item 2 cv bar nu eps}supposing assumption \eqref{hyp4:f0}
		with index $k\,=\,1$ and  assumption \eqref{hyp bar nu th2} with
		index $k\,=\,0$, it holds for all $t \geq0$,
		\begin{equation*}
			\|\,\ols{\nu}^\eps(t)\,-\,\ols{\nu}(t)\,\|_{\scH^{\,0}(\ols{m})}
			\,\leq\,
			e^{Ct}\,
			\left(
			\|\ols{\nu}^\eps_0\,-\,\ols{\nu}_0\|_{\scH^{\,0}(\ols{m})}
			\,+\,
			C\,\left\|\,
			\nu^\eps_0\,
			\right\|_
			{
				\scH^{\,2}\left(m ^\eps\right)
			}\eps\,\sqrt{
				\left|\,\ln{\eps}\,\right|
				\,+\,1
			}\,
			\right)\,.
		\end{equation*}
	\end{enumerate}
\end{proposition}

\begin{proof}
	We prove item \eqref{item 2 cv bar nu eps} in the latter proposition. According to Lemma \ref{lemme technique L 2 m}, we have
	\begin{equation*}
		\|\,
		\ols{\nu}^\eps
		\,-\,
		\ols{\nu}\,\|_{L^{2}(\ols{m})}
		\,\leq\,
		e^{C\,t}
		\left(
		\|\,
		\ols{\nu}^\eps_0
		\,-\,
		\ols{\nu}_0\,\|_{L^{2}(\ols{m})}
		\,+\,C
		\int_0^t
		\,e^{-C\,s}
		\,\theta^\eps\,
		\|\,\nu^\eps_{\bot}\|_{H_w^{1}\left(m ^\eps\right)}
		\,\dD s\,
		\right)\,
		.
	\end{equation*}
	Therefore, the proof comes down to estimating the integral in the right-hand side of the latter inequality
	\begin{equation*}
		\cA\,:=\,
		\int_0^t
		\,e^{-C\,s}
		\,\theta^\eps\,
		\|\,\nu^\eps_{\bot}\|_{H^{1}_w\left(m ^\eps\right)}
		\,\dD s\,
		.
	\end{equation*}
	We apply the second estimate in Proposition \ref{estimee:nu orth} and Cauchy-Schwarz inequality. This yields
	\begin{equation*}
		\cA\,
		\leq\,
		C\,
		\left\|\,
		\nu^\eps_0\,
		\right\|_
		{
			H^{2}_{w}
			\left(
			m^\eps
			\right)
		}\,
		\left(
		\int_0^t
		\left|\theta^\eps\right|^2
		\,\dD s
		\right)^{1/2}
		\,
		\left(
		\int_0^t
		\eps
		\,+\,
		\min\left\{
		\,1\,,\,
		e^{
			\,
			-\,2\,\alpha_* \frac{s}{\eps}
		}
		\,
		\eps^{-\,\frac{\alpha_*}{m_*}}
		\right\}
		\,\dD s
		\right)^{1/2}
		\,
		.
	\end{equation*}
	Then we inject the following estimate in the latter inequality
	\[
	\min\left\{
	\,1\,,\,
	e^{
		\,
		-\,2\,\alpha_* s \,/\,\eps\,
	}
	\,
	\eps^{-\,\alpha_*\,/\,
		m_*}
	\right\}
	\,\leq\,
	\mathds{1}_{\,
		\left\{
		s
		\,\leq\,
		-\frac{1}{2m_*}\eps\ln{\eps}
		\right\}
	}
	\;+\;
	\mathds{1}_{\,
		\left\{
		s
		\,>\,
		-\frac{1}{2m_*}\eps\ln{\eps}
		\right\}
	}\;
	e^{
		\,
		-\,2\,\alpha_*\,\frac{s}{\eps}
	}
	\,
	\eps^{-\,\frac{\alpha_*}{m_*}
	}\,.
	\]
	Moreover, we use \eqref{expression for theta eps} to compute the time integral of $\ds\left|\theta^\eps\right|^2$. In the end, we obtain
	\begin{equation*}
		\cA\,
		\leq\,
		C\,
		\left\|\,
		\nu^\eps_0\,
		\right\|_
		{
			H^{2}_{w}
			\left(
			m^\eps
			\right)
		}\,\eps\,\sqrt{t\,+\,1}
		\,
		\sqrt{
			\left|\,\ln{\eps}\,\right|
			\,+\,1
		}
		\,
		.
	\end{equation*}
	Hence, we obtain the expected result taking the supremum over all $\bx$ in $K$ in the latter estimate.
	
	Item \eqref{item 1 cv bar nu eps} in Proposition \ref{estimee:bar nu bar nu eps} is obtained following the same method excepted that we estimate $\cA$ using Proposition \ref{estime L 2 m g eps} with index $k$ instead of Proposition \ref{estimee:nu orth}.\\
\end{proof}
Let us now conclude the proof of Theorem
\ref{th:2}. On the one hand, we observe that item
\eqref{item2 th21} corresponds to item \eqref{item 2 cv bar nu eps}
of Proposition \ref{estimee:bar nu bar nu eps}. On the other hand,
item \eqref{item 1 th2} is obtained by gathering the estimate in
Proposition \ref{estimee:nu orth} and item \eqref{item 1 cv bar nu
	eps} in Proposition \ref{estimee:bar nu bar nu eps}. 

\subsection{Proof of Theorem \ref{th21}}\label{proof 21}
All along this proof, we consider some $\eps_0$ small enough so that the following condition is fulfilled
\begin{equation*}
	\left\|\,\rho_0^\varepsilon
	\,-\,
	\rho_0\,
	\right\|_{L^\infty(K)}
	\,<\, m_*\,/\,2\,,
\end{equation*}
for all $\eps$ less than $\eps_0$. We omit the dependence with respect to $\ds(t,\bx,\bu) \in
\R^+\times K \times \R^2
$ when the context is clear.
We start by proving item \eqref{cv mu eps H 0} in Theorem \ref{th21}. Since the cases $k=0$ and $k=1$ are treated the same way, we only detail the case $k=0$.
We consider some integer $i$ and take some $\eps$ less than $\eps_0$. Then we decompose the error as follows
\[
\|\left(
v\,-\,\mathcal{V}
\right)^{i}
\,
\left(
\mu^\eps
\,-\,
\mu
\right)(t)
\|_{\mathcal{H}^0
	\left(
	m^-
	\right)}
\,\leq\,
\cA
\,+\,
\cB\,,
\]
where $\cA$ and $\cB$ are given by\\[0.3em]
\begin{equation*}
	\left\{
	\begin{array}{l}
		\displaystyle  \cA \,=\,
		\|
		(v\,-\,\cV)^{i}
		\,
		\left(
		\mu^\eps
		\,-\,
		\tau_{-\,
			\mathcal{U}^\eps
		}\,\circ\,
		D_{\theta^\eps}
		\left(\,
		\nu\,
		\right)
		\right)
		\,
		\|_{\mathcal{H}^0
			\left(
			m^-
			\right)}
		\,
		,\\[1.1em]
		\displaystyle \cB
		\,=\,\|
		(v\,-\,\cV)^{i}
		\,
		\left(
		\,
		\tau_{-\,
			\mathcal{U}^\eps
		}\,\circ\,
		D_{\theta^\eps}
		\left(\,
		\nu\,
		\right)
		\,-\,
		\mu
		\right)
		\,
		\|_{\mathcal{H}^0
			\left(
			m^-
			\right)
		}
		\,,
	\end{array}
	\right.\\[0.5em]
\end{equation*}
where $\nu$ is the limit of $\nu^\eps$ in Theorem \ref{th:2} and is defined by \eqref{limit nu} and where the operators 
$
\ds
\tau_{-\,
	\mathcal{U}^\eps
}
$ and 
$
\ds
D_{\theta^\eps}
$
respectively stand for the translation of vector 
$
\ds-\,
\mathcal{U}^\eps
$ with respect to the $\bu$-variable and the dilatation with parameter 
$
\ds
\left(\theta^\eps
\right)^{-1}
$ with respect to the $v$-variable, that is 
\[
\tau_{-\,
	\mathcal{U}^\eps
}\,\circ\,
D_{\theta^\eps}
\left(\,
\nu\,
\right)
\left(
t\,,\,\bx\,,\,\bu
\right)
\,=\,
\frac{1}{
	\theta
	^\eps
}\,
\nu
\left(
t\,,\,\bx\,,\,
\frac{
	v\,
	-\,
	\mathcal{V}^\eps
}{
	\theta
	^\eps
}
\,,\,
w\,
-\,
\mathcal{W}^\eps
\right)
\,.
\]
The first term $\cA$ corresponds to the convergence of the re-scaled version $\nu^\eps$ of $\mu^\eps$ towards $\nu$ whereas $\cB$ corresponds to the convergence of the macroscopic quantities.\\

We estimate $\cA$ as follows
\[
\cA
\,\leq\,
C
\left(\,
\cA_1
\,+\,
\cA_2
\right)\,,
\]
where $C$ is a positive constant which only depends on $i$ and where $\cA_1$ and $\cA_2$ are given by
\begin{equation*}
	\left\{
	\begin{array}{l}
		\displaystyle \cA_1
		\,=\,
		\left\|\,
		\left(
		v
		\,-\,
		\cV^\eps
		\right)^i
		\,
		\left(
		\,
		\mu^\eps
		\,-\,
		\tau_{-\,
			\mathcal{U}^\eps
		}\,\circ\,
		D_{\theta^\eps}
		\left(\,
		\nu\,
		\right)
		\,
		\right)\,
		\right\|_{\mathcal{H}^0
			\left(
			m^-
			\right)}
		\,,\\[1.1em]
		\displaystyle  \cA_2 \,=\,
		\left\|
		\,
		\cV
		\,-\,
		\cV^\eps
		\,
		\right\|_{L^{\infty}
			\left(
			K
			\right)}^{i}\,
		\left\|
		\,
		\mu^\eps
		\,-\,
		\tau_{-\,
			\mathcal{U}^\eps
		}\,\circ\,
		D_{\theta^\eps}
		\left(\,
		\nu\,
		\right)
		\,
		\right\|_{\mathcal{H}^0
			\left(
			m^-
			\right)}
		\,
		.
	\end{array}
	\right.\\[0.7em]
\end{equation*}
According to item \eqref{estimate moment mu} in Proposition \ref{th:preliminary} , $\mathcal{V}^\eps$ and $\mathcal{W}^\eps$ are
uniformly bounded with respect to both $(t,\bx) \in \R^+\times K$ and
$\eps\,>\,0$, and $0\leq \theta^\eps\leq 1$,  hence
$$
m^-(\bx,v, w) \, \leq\, C\, m^\eps
\left(
\bx\,,\,
\frac{v\,-\,\cV^\eps}{
	\sqrt{2}\,\theta^\eps
}\,,\,
w\,-\,\cW^\eps
\right),
$$
which yields
\[
\cA_1
\,\leq\,
C\,
\sup_{\bx \in K}
\left(\,
\int_{\R^2}
\left(
v
\,-\,
\cV^\eps
\right)^{2\,i}
\,
\left|
\,
\mu^\eps
\,-\,
\tau_{-\,
	\mathcal{U}^\eps
}\,\circ\,
D_{\theta^\eps}
\left(\,
\nu\,
\right)
\,
\right|^{2}\,
m^\eps
\left(
\bx\,,\,
\frac{v\,-\,\cV^\eps}{
	\sqrt{2}\,\theta^\eps
}\,,\,
w\,-\,\cW^\eps
\right)
\,\dD \bu
\right)^{\frac{1}{2}}
\,,
\]
for another constant $C>0$ depending only on $\kappa$, $m_*$, $m_p$ and $\ols{m}_p$ (see assumptions \eqref{hyp:rho0}-\eqref{hyp2:f0}) and on the data of the problem $N$,~$\Psi$~and~$A_0$. Then we invert the change of variable \eqref{change:var} and notice that
\[
v^{2\,i}\,
m^\eps
\left(
\bx\,,\,
\frac{v}{
	\sqrt{2}
}\,,\,
w
\right)
\,\leq\,
C\,
m^\eps
\left(
\bx\,,\,
\bu
\right)\,,
\]
for some constant $C>0$ only depending on $i$ and $m_*$. Consequently, we deduce
\[
\cA_1
\,\leq\,
C\,
\left\|\,
\left(
\theta^\eps
\right)^{i\,-\,\frac{1}{2}}\,
\right\|_{L^{\infty}
	\left(
	K
	\right)}\,
\left\|\,
\nu^\eps
\,-\,
\nu\,
\right\|_{\mathcal{H}^0
	\left(
	m^\eps
	\right)}\,.
\]
Therefore, applying Theorem \ref{th:2}, using the compatibility
assumption \eqref{th:2 compatibility assumption 1}, and thanks to the constraint 
$
\ds
\theta^\eps
\left(
t=0
\right)\,=\,1 
$, which ensures
\[
\left\|\,
\nu^\eps_{0}\,
\right\|_
{
	\scH^{1}
	\left(
	m ^\eps
	\right)
}
\,\leq\,
C\,
\left\|\,
\mu^\eps_{0}\,
\right\|_
{
	\scH^{1}
	\left(
	m^+
	\right)
}\,,
\]
for some constant $C$ depending only on $m_p$, $\kappa$ and $m_*$, we
finally get
\[
\left.
\cA_1
\,\leq\,
C\,e^{Ct}\,\eps^{-\frac{1}{4}}
\left(
\eps^{\frac{i}{2}}\,+\,
e^{-\,i\,m_*\,\frac{t}{\eps}}
\right)\,
\left(
\eps^{\frac{1}{2}}
\,+\,
e^{
	\,
	-\,\alpha_*\frac{t}{\eps}\,
}
\,
\eps^{-\,\frac{\alpha_*}{2\,m_*}
}\,
\right)
\,.
\right.
\]
Moreover, since $
\ds \alpha_*\,<\,m_*\,/\,2\,
$, we deduce
\[
\left.
\cA_1
\,\leq\,
C\,e^{Ct}\,
\left(\,
\eps^{\frac{i}{2}\,+\,\frac{1}{4}}
\,+\,
e^{
	\,
	-\,\alpha_*\frac{t}{\eps}\,
}
\,
\eps^{-\,\frac{1}{2}}\,
\right)
\,.
\right.\\[0.7em]
\]
To estimate $\cA_2$, we  apply item \eqref{cv macro q} in Proposition \ref{th:preliminary} and the compatibility assumption \eqref{th:2 compatibility assumption 1}, which ensure
\[
\left\|
\,
\cV
\,-\,
\cV^\eps
\,
\right\|_{L^{\infty}
	\left(
	K
	\right)}^{i}
\,\leq\,
C\,
e^{C\,t}
\,\eps^{i}\,.
\]
Then we follow the same method as before. In the end, we end up with the following bound for $\cA_2$
\[
\left.
\cA_2
\,\leq\,
C\,e^{Ct}\,
\left(\,
\eps^{i\,+\,\frac{1}{4}}
\,+\,
e^{
	\,
	-\,\alpha_*\frac{t}{\eps}\,
}
\,
\eps^{i\,-\,\frac{1}{2}}\,
\right)
\,.
\right.\\[0.7em]
\]
Gathering these results, we obtain the following estimate for $\cA$
\[
\left.
\cA
\,\leq\,
C\,e^{Ct}\,
\left(\,
\eps^{\frac{i}{2}\,+\,\frac{1}{4}}
\,+\,
e^{
	\,
	-\,\alpha_*\frac{t}{\eps}\,
}
\,
\eps^{-\,\frac{1}{2}}\,
\right)
\,.
\right.\\[0.7em]
\]

We turn to $\cB$. Similarly as before, we apply the triangular inequality and invert the change of variable \eqref{change:var}. 
Then we apply Proposition \ref{th:preliminary}, which yields
\[
\cB
\,\leq\,
C\,e^{C\,t}\eps^{-\frac{1}{4}}
\left(
\,\eps^{\frac{i}{2}}\,
\,+\,
e^{-\,i\,m_*\,\frac{t}{\eps}}
\right)
\,
\left\|\,
\nu
\,-\,
\tau_{\,
	\left(\,
	\frac{
		\mathcal{V}^\eps
		-\,
		\mathcal{V}
	}{\theta^\eps}\,,\,
	\mathcal{W}^\eps
	-\,
	\mathcal{W}\,
	\right)
}\,
\left(
\cM_{\rho_0}\otimes \bar{\nu}\,
\right)
\,
\right\|_{\mathcal{H}^0
	\left(
	D_{\sqrt{2}}\,\,
	\left(m^\eps\right)
	\right)}\,,
\]
where 
$
\ds D_{\sqrt{2}}\,\,
\left(m^\eps\right)
$ is a short-hand notation for
\[
D_{\sqrt{2}}\,
\left(m^\eps\right)
\left(
\bx\,,\,\bu
\right)
\,=\,
m^\eps
\left(
\bx\,,\,
\frac{v}{\sqrt{2}}\,,\,
w
\right)\,.
\]
Then we decompose the right-hand side of the latter inequality as follows
\[
\left\|\,
\nu
\,-\,
\tau_{
	\left(\,
	\frac{
		\mathcal{V}^\eps
		-\,
		\mathcal{V}
	}{\theta^\eps}\,,\,
	\mathcal{W}^\eps
	-\,
	\mathcal{W}\,
	\right)
}\,
\left(
\cM_{\rho_0}\otimes \bar{\nu}\,
\right)
\,
\right\|_{\mathcal{H}^0
	\left(
	D_{\sqrt{2}}\,\,
	\left(m^\eps\right)
	\right)}
\,\leq\,
\cB_1
\,+\,
\cB_2
\,+\,
\cB_3
\,,
\]
where $\cB_1$,
$
\cB_2$
and
$\cB_3$ are given by
\begin{equation*}
	\left\{
	\begin{array}{l}
		\displaystyle \cB_1
		\,=\,
		\left\|\,
		\ols{\nu}
		\,-\,
		\tau_{\,
			\left(\,
			\mathcal{W}^\eps
			-\,
			\mathcal{W}\,
			\right)
		}\,
		\bar{\nu}
		\,
		\right\|_{\mathcal{H}^0
			\left(
			\ols{m}
			\right)}
		\,,\\[1.1em]
		\displaystyle  \cB_2 \,=\,
		\left\|\,
		\tau_{\,
			\left(\,
			\mathcal{W}^\eps
			-\,
			\mathcal{W}\,
			\right)
		}\,
		\bar{\nu}
		\,
		\right\|_{\mathcal{H}^0
			\left(
			\ols{m}
			\right)}\,
		\left\|\,
		\cM_{\rho^\eps_0}
		\,-\,
		\tau_{
			\left(\,
			\frac{
				\mathcal{V}^\eps
				-\,
				\mathcal{V}
			}{\theta^\eps}\,
			\right)
		}\,
		\cM_{\rho^\eps_0}
		\,
		\right\|_{\mathcal{H}^0
			\left(
			\cM_{\rho^\eps_0}^{-1}
			\right)}
		\,,\\[1.6em]
		\displaystyle  \cB_3 \,=\,
		\left\|\,
		\tau_{\,
			\left(\,
			\mathcal{W}^\eps
			-\,
			\mathcal{W}\,
			\right)
		}\,
		\bar{\nu}
		\,
		\right\|_{\mathcal{H}^0
			\left(
			\ols{m}
			\right)}\,
		\left\|\,
		\cM_{\rho^\eps_0}
		\,-\,
		\cM_{\rho_0}
		\,
		\right\|_{\mathcal{H}^0
			\left(
			\cM_{\rho^\eps_0}^{-1}
			\right)}\,
		e^{\,
			\|
			\mathcal{V}^\eps
			-\,
			\mathcal{V}
			\|^2_{L^{\infty}(K)}
			\,/\,
			\left(2\,m_*\,\eps\right)\,
		}
		\,,
	\end{array}
	\right.\\[0.3em]
\end{equation*}
where we used that
\[
D_{\sqrt{2}}\left(m^\eps\right)\,\leq\,m^\eps\,,\quad\textrm{and}\quad m^\eps\,=\,\cM_{\rho^\eps_0}^{-1}\,\ols{m}\,.
\]
Since equation \eqref{bar nu:eq} is linear, 
$\ds
\ols{\nu}
\,-\,
\tau_{\,
	w_0}\,\ols{\nu}
$ also solves the equation, therefore, applying Lemma \ref{lemme:bar nu} with $w_0\,=\,
\mathcal{W}^\eps
-\,
\mathcal{W}
$, it yields
\[
\cB_1
\,\leq\,
C\,
e^{\frac{b}{2}\,t}\,
\left\|\,
\ols{\nu}_0
\,-\,
\tau_{\,
	\left(\,
	\mathcal{W}^\eps
	-\,
	\mathcal{W}\,
	\right)
}\,
\bar{\nu}_0
\,
\right\|_{\mathcal{H}^0
	\left(
	\ols{m}
	\right)}
\,.
\]
Furthermore, since 
$\ols{m}
\,\leq\,\ols{m}^+
$ and relying on assumption \eqref{continuite L 2 bar mu}, we deduce
\[
\cB_1
\,\leq\,
C\,
e^{\frac{b}{2}\,t}\,
\left\|\,
\mathcal{W}^\eps
-\,
\mathcal{W}
\,
\right\|_{L^{\infty}(K)}\,.
\]
Therefore, according to item \eqref{cv macro q} in Proposition \ref{th:preliminary} we conclude
\[
\cB_1\,\leq\,
C\,
e^{C\,t}\,
\eps\,.
\]
To estimate $\cB_2$ and $\cB_3$, we follow the same method as for $\cB_1$: we first apply the following relation
\[
\left\|\,
\cM_{\rho^\eps_0}
\,-\,
\tau_{
	\left(\,
	\frac{
		\mathcal{V}^\eps
		-\,
		\mathcal{V}
	}{\theta^\eps}\,
	\right)
}\,
\cM_{\rho^\eps_0}
\,
\right\|_{\mathcal{H}^0
	\left(
	\cM_{\rho^\eps_0}^{-1}
	\right)}
\,=\,
\left\|\,
e^{\,\rho_0^\eps\,
	\left|\,
	\frac{
		\mathcal{V}^\eps
		-\,
		\mathcal{V}
	}{\theta^\eps}\,
	\right|^2
}
\,-\,1
\,
\right\|^{\frac{1}{2}}_{L^{\infty}(K)}
\,,
\]
which ensures 
\[
\left\|\,
\cM_{\rho^\eps_0}
\,-\,
\tau_{
	\left(\,
	\frac{
		\mathcal{V}^\eps
		-\,
		\mathcal{V}
	}{\theta^\eps}\,
	\right)
}\,
\cM_{\rho^\eps_0}
\,
\right\|_{\mathcal{H}^0
	\left(
	\cM_{\rho^\eps_0}^{-1}
	\right)}
\,\leq\,
\left\|\,
e^{\,\rho_0^\eps\,
	\left|\,
	\frac{
		\mathcal{V}^\eps
		-\,
		\mathcal{V}
	}{\theta^\eps}\,
	\right|^2
}
\,\rho_0^\eps\,
\left|\,
\frac{
	\mathcal{V}^\eps
	-\,
	\mathcal{V}
}{\theta^\eps}\,
\right|^2
\right\|^{\frac{1}{2}}_{L^{\infty}(K)}
\,.
\]
Furthermore, we apply Lemma \ref{lemme:bar nu}, which ensures that
\[
\left\|\,
\tau_{\,
	\left(\,
	\mathcal{W}^\eps
	-\,
	\mathcal{W}\,
	\right)
}\,
\bar{\nu}
\,
\right\|_{\mathcal{H}^0
	\left(
	\ols{m}
	\right)}
\,\leq\,
C\,e^{C\,t}\,.
\]
Therefore, applying item \eqref{cv macro q} in Proposition \ref{th:preliminary}, we obtain
\[
\cB_2
\,\leq\,
C\,e^{C\,
	\left(t\,+\,
	\eps\,e^{Ct}\right) 
}
\,\eps^{\frac{1}{2}}
\,.
\]
Then to estimate $\cB_3$, an exact computation yields
\[
\left\|\,
\cM_{\rho^\eps_0}
\,-\,
\cM_{\rho_0}\,
\right\|_{\mathcal{H}^0
	\left(
	\cM_{\rho^\eps_0}^{-1}
	\right)}
\,=\,
\sup_{\bx \in K}
\left(
\frac{
	\left|\,
	\rho_0
	\,-\,
	\rho_0^\eps
	\,\right|^2
}{
	\,
	\sqrt{\rho_0^2
		-
		\left(
		\rho_0
		-
		\rho_0^\eps
		\right)^2
	}
	\left(
	\rho_0
	\,+\,
	\sqrt{\rho_0^2
		-
		\left(
		\rho_0
		-
		\rho_0^\eps
		\right)^2
	}
	\right)
}
\,\right)^{\frac{1}{2}}\,.
\]
Therefore, according to assumption \eqref{th:2 compatibility assumption 1} and item \eqref{cv macro q} in Proposition \ref{th:preliminary}, which ensures that
\[
e^{\,
	\|
	\mathcal{V}^\eps
	-\,
	\mathcal{V}
	\|^2_{L^{\infty}(K)}
	\,/\,
	\left(2\,m_*\,\eps\right)\,
}
\,\leq\,
e^{\,C\,e^{Ct}\,
	\eps
}\,,
\] 
this yields
\[
\cB_3\,\leq\,C\,e^{C\,
	\left(t
	\,+\,
	e^{Ct}\,
	\eps\right)}\,\eps\,.
\]
In the end, we deduce the following estimate for $\cB$
\[
\cB
\,\leq\,
C\,e^{C\,
	\left(
	t
	\,+\,
	\eps\,e^{Ct}
	\right)
}
\left(
\,\eps^{\frac{i}{2}
	+
	\frac{1}{4}}\,
\,+\,
e^{-\,i\,m_*\,\frac{t}{\eps}}
\,\eps^{\frac{1}{4}}
\right)
\,.
\]

The proof for the statement \eqref{item 2 th 22} in Theorem \ref{th21} follows the same lines as the former one excepted that we apply item \eqref{item2 th21} in Theorem \ref{th:2} instead of item \eqref{item 1 th2} to quantify the convergence of $\ols{\nu}^\eps$ towards $\ols{\nu}$. Therefore, we do not detail the proof.\\

\section{Conclusion}
This article highlights how macroscopic behavior arise in spatially extended FitzHugh-Nagumo neural networks. In the regime where strong short-range interactions dominate, we proved that the voltage distribution concentrates with Gaussian profile around an averaged value $\cV$ whereas the adaptation variable converges towards a distribution $\bar{\mu}$. The limiting quantities $\left(\cV,\bar{\mu}\right)$ solve the coupled reaction diffusion-transport system \eqref{macro:eq}. The novelty of this work is that we derive quantitative estimates ensuring strong convergence towards the Gaussian profile. More precisely, we provide two results. On the one hand, we prove convergence in a $L^1$ functional framework. Our analysis relies on a modified Boltzmann entropy (see Appendix \ref{sec:rel ent}) which is original to our knowledge. On the other hand, we prove convergence in a weighted $L^2$ setting, in which we take advantage of the variational structure to obtain regularity estimates and recover optimal convergence rates. These results complement collaborations of the author with E. Bouin and F. Filbet dedicated to the quantitative analysis of the strong short-range interaction regime (see \cite{BF} for weak convergence estimates and \cite{Blaustein_Bouin} for uniform convergence estimates). \\

A natural continuation of this work concerns the link between equation \eqref{kinetic:eq} and other popular models in neuroscience. More precisely, it is of primary interest to understand how the FitzHugh-Nagumo model relates to models based on "forced action potential" such as integrate and fire \cite{CCP,Carrillo_Gonzalez_Mar_Gualdani13,CPSS,Delarue_Inglis_Rubenthaler_Tanre15,roux,Roux_Salort_21}, voltage-conductance \cite{Perthame/Salort13,Dou_Perthame_Salort_Zhou23} and time elapsed  \cite{PPS,CCDR,CHE1,Mischler_Weng18} neural models. Indeed, Hodgkin-Huxley and FitzHugh-Nagumo neural models reproduce the spiking behavior of neurons thanks to an autonomous system of ordinary differential equations whereas their  "forced action potential" counterparts artificially enforce the spiking behavior. Therefore, the next step of our investigation is to derive "forced action potential" models as (biologically relevant) asymptotic limits of FitzHugh-Nagumo or Hodgkin-Huxley networks.\\

To conclude, our model displays structural similarities with others coming from kinetic theory such as flocking models \cite{FK,KV,KK,PS}, Vlasov-Navier-Stokes models with Brinkman force \cite{Daniel} and Vlasov-Poisson-Fokker-Planck models \cite{VPFP,El Ghani/ Masmoudi,Herda}. Hence, a natural question concerns the applicability of our methods in this wider context. We partially answered this question in \cite{Blaustein_VPFP} by applying the strategy explained at the beginning of Section \ref{sec1b} to treat the diffusive limit of the Vlasov-Poisson-Fokker-Planck model. Closer to applications in Biology, we also address the applicability of our approach to the fast adaptation regime in the run and tumble model for bacterial motion, analyzed by B. Perthame \textit{et al.} in \cite{Perthame_Tang_Vauchelet}.
\section*{Acknowledgment}
The author warmly thanks Francis Filbet whose recommendations and advises helped the completion of this work.\\
The author gratefully acknowledges the support of  ANITI (Artificial and Natural Intelligence Toulouse Institute). This project has received support from ANR ChaMaNe No: ANR-19-CE40-0024.

\appendix

\section{Proof of Lemma \ref{abstract lemma rel ent}
}\label{sec:rel ent}
Instead of estimating directly the $L^1$ norm of $f-g$, we estimate the following intermediate quantity, introduced in \cite{BF}
\begin{equation*}
	H_{1/2}\left[\,f\, |\,g\,\right]\,=\,\int_{\R^{d_1+d_2}}
	f
	\,\ln{
		\left(
		\frac{
			2\,f
		}{
			f + g
		}
		\right)
	}
	\,\dD\by \,\dD \xi \,, 
\end{equation*}
which is easily comparable to the classical $L^1$ norm thanks to the following Lemma (see \cite{BF}, Lemma $\mathrm{A}.1.$ for detailed proof)
\begin{lemma}
	\label{entropy V.S. L1}
	For any two non-negative functions
	$f$,  $g \in L^1\left(\R^{d_1+d_2}\right)$ with integral equal to one, the following estimate holds
	\begin{equation*}
		\frac{1}{8}
		\left\|\,
		f\,-\,
		g\,
		\right\|_{L^1
			\left(
			\R^{d_1+d_2}
			\right)}^2
		\,\leq\,
		H_{1/2}
		\left[\,
		f\,|\,
		g\,
		\right]
		\,\leq\,
		\left\|\,
		f
		\,-\,
		g\,
		\right\|_{L^1(\R^{d_1+d_2})}\,.
	\end{equation*}
\end{lemma}
Unlike the $L^1$ norm, $H_{1/2}$ has an explicit dissipation with respect to the Laplace operator which is given by
\begin{equation*}
	I_{1/2}\left[\,f\,|\,g\,\right] \,:=\, \int_{\R^{d_1+d_2}}\left|\nabla_{\xi}\ln{\left(\frac{2\,f}
		{
			f+g}\right)
		}\right|^2 f\,\dD\by \,\dD \xi \,.
\end{equation*}
We consider the quantity 
$\ds h\,:=\,
\left(
f+g
\right)/\,2\,
$, whose equation is obtained multiplying by $1/2$ the sum of the equations solved by $f$ and $g$, that is
\begin{equation*}
\partial_t\, h
\,+\,
\mathrm{div}_{\by}
\left[
\,\mathbf{a}
\,h\, \right]
\,+\,
\frac{\lambda(t)}{2}\,
\mathrm{div}_{\xi}
\left[\,
\left(\mathbf{b}_1
\,+\,
\mathbf{b}_3
\right)
\,f\,
+\,
\mathbf{b}_2
\,g\,\right]
\,-\,\lambda(t)^2\,
\Delta_{\xi}\,
h
\,=\,
(\delta-1)\,
\frac{\lambda(t)^2}{2}\,
\Delta_{\xi}\,
g\,.
\end{equation*}
Then, we compute the time derivative of the quantity
$
\ds
H_{1/2}
\left[\,
f\,|\,
g\,
\right]
$
integrating with respect to both $\xi$ and $\by$ the difference between the equation solved by $f$ multiplied by
$\ds
\ln{
\left(f\,/\,h
\right)
}
$
and the equation solved by $h$ multiplied by
$\ds
f\,/\,h
$. After an integration by part, it yields
\[
\frac{\dD}{\dD t}\,
H_{
1/2
}
\left[
\,
f
\,|\,
g
\,
\right]
\,+\,
\lambda(t)^2\,
I_{
1/2
}
\left[
\,
f
\,|\,
g
\,
\right]
\,=\,
\cA\,+\,\cB
\,,
\]
where $\cA$ and $\cB$ are given by
\begin{equation*}
\left\{
\begin{array}{l}
\displaystyle  \cA \,=\,
\frac{\lambda(t)}{2}\,
\int_{\R^{d_1+d_2}}
\left(\,
\mathbf{b}_1\,
\,-\,
\mathbf{b}_2\,
\right)\,
\nabla_{\xi}
\left(
\ln{
	\left(
	\frac{f}{h
	}
	\right)
}
\right)\,
\frac{g}{h}\,f
\, \dD \by\, \dD \xi \,
,\\[1.5em]
\displaystyle  \cB \,=\,
-\,\frac{\lambda(t)}{2}
\int_{\R^{d_1+d_2}}
\mathrm{div}_{\xi}
\left[\,
\mathbf{b}_3
\,g
\,+\,
\left(\delta
\,-\,1\right)\,\lambda(t)\,
\nabla_{\xi}\,
g\,
\,\right]\,
\frac{f}{h}\,
\dD\by\,\dD\xi
\,.
\end{array}
\right.\\[0.8em]
\end{equation*}
To estimate $\cA$, we notice that $\ds 
\left|\,
g\,/\,h\,
\right|\,\leq\,2\,
$ and we apply Young's inequality. This yields
\[
\cA
\,\leq\,
\lambda(t)^2\,
I_{
1/2
}
\left[
\,
f
\,|\,
g
\,
\right]
\,+\,
\frac{1}{4}\,
\int_{\R^{d_1+d_2}}
\left|\,
\mathbf{b}_1\,
\,-\,
\mathbf{b}_2\,
\right|^2 f\,
\dD\by\,\dD\xi
\,.
\]
To estimate $\cB$, we simply notice that 
$\ds 
\left|\,
f\,/\,h\,
\right|\,\leq\,2\,
$ and take the absolute value inside the integral. In the end, we obtain 
\[
\cA\,+\,\cB
\,\leq\,
\cR(t)
\,+\,
\lambda(t)^2\,
I_{
1/2
}
\left[
\,
f
\,|\,
g
\,
\right]
\,,
\]
where $\cR$ is defined as in Lemma \ref{abstract lemma rel ent}. Therefore, it holds
\begin{equation*}
\frac{\dD}{\dD t}\,
H_{1/2}
\left[\,
f(t)\,|\,g(t)\,
\right]
\,\leq\,
\cR(t)
\,,
\quad
\forall\,t\in\R^+\,.
\end{equation*}
Then, we deduce \eqref{estimee2:abstrat lemma} by integrating the latter inequality between $0$ and $t$ and applying Lemma \ref{entropy V.S. L1} in order to substitute $H_{1/2}$ with the $L^1$-norm.

\end{document}

%% file: macro.tex
\usepackage{amsthm,amsfonts,amssymb,amsmath,oldgerm}
\numberwithin{equation}{section}
\usepackage{fullpage}
\usepackage{amsmath}
\usepackage{setspace}
\usepackage{stmaryrd}
\usepackage{mathrsfs}
\usepackage{fancyhdr}
\usepackage{graphicx}
\usepackage{enumerate}
\usepackage{bm}
\usepackage{bbm}
\usepackage{dsfont}
\usepackage{multirow}
\usepackage{psfrag}
\usepackage[font=scriptsize]{caption}
\usepackage[font=scriptsize]{subcaption}
\usepackage{listings}
\usepackage{tikz}
\usepackage{empheq}
\usepackage{cases}
\usetikzlibrary{matrix} 
\usepackage[framemethod=TikZ]{mdframed}
\mdfdefinestyle{MyFrame}{backgroundcolor=gray!20!white}
\usepackage[makeroom]{cancel}

\usepackage[top=20mm,bottom=20mm,left=20mm,right=20mm,a4paper]{geometry}

\usepackage{hyperref,bookmark}

\usepackage[normalem]{ulem}
\definecolor{violet}{rgb}{0.580,0.,0.827}

\usepackage[textwidth= \marginparwidth,textsize=tiny]{todonotes}

\hypersetup{
    colorlinks,          
    citecolor=blue,        
    filecolor=blue,      
    urlcolor=blue,           
    linkcolor=blue,
}

\newcommand{\LLL}[1]{{\left\vert\kern-0.25ex\left\vert\kern-0.25ex\left\vert #1 
    \right\vert\kern-0.25ex\right\vert\kern-0.25ex\right\vert}}

\newcommand{\ols}[1]{\mskip.5\thinmuskip\overline{\mskip-.5\thinmuskip {#1} \mskip-.5\thinmuskip}\mskip.5\thinmuskip} 

\newcommand\dD{\mathrm{d}}

\def\eps{\varepsilon }

\newcommand\br{\begin{remark}}
\newcommand\er{\end{remark}}
\newcommand\bp{\begin{pmatrix}}
\newcommand\ep{\end{pmatrix}}
\newcommand{\be}{\begin{equation}}
\newcommand{\ee}{\end{equation}}
\newcommand\ba{\begin{equation}\begin{aligned}}
\newcommand\ea{\end{aligned}\end{equation}}
\newcommand\ds{\displaystyle}

\newcommand{\beg}{\begin{example}}
\newcommand{\eeg}{\end{exaplem}}
\newcommand{\bpr}{\begin{proposition}}
\newcommand{\epr}{\end{proposition}}
\newcommand{\bt}{\begin{theorem}}
\newcommand{\et}{\end{theorem}}
\newcommand{\bc}{\begin{corollary}}
\newcommand{\ec}{\end{corollary}}
\newcommand{\bl}{\begin{lemma}}
\newcommand{\el}{\end{lemma}}
\newcommand{\bd}{\begin{definition}}
\newcommand{\ed}{\end{definition}}
\newcommand{\brs}{\begin{remarks}}
\newcommand{\ers}{\end{remarks}}

\newtheorem{theorem}{Theorem}[section]
\newtheorem{proposition}[theorem]{Proposition}
\newtheorem{corollary}[theorem]{Corollary}
\newtheorem{lemma}[theorem]{Lemma}
\newtheorem{remark}[theorem]{Remark}
\newtheorem{definition}[theorem]{Definition}

\newtheorem{example}[theorem]{Example}

\newcommand{\N}{{\mathbb N}}

\newcommand{\R}{{\mathbb R}}

\newcommand\bx{{\bm x}}

\newcommand{\by}{\bm y}

\newcommand{\bu}{\bm u}
\newcommand\bv{{\bm v}}

\newcommand\cA{{\mathcal A}}
\newcommand\cB{{\mathcal B}}

\newcommand\cD{{\mathcal D}}
\newcommand\cE{{\mathcal E}}
\newcommand\cF{{\mathcal F}}

\newcommand\cM{{\mathcal M}}

\newcommand\cR{{\mathcal R}}

\newcommand\cU{{\mathcal U}}
\newcommand\cV{{\mathcal V}}
\newcommand\cW{{\mathcal W}}

\newcommand\scA{{\mathscr A}}
\newcommand\scC{{\mathscr C}}

\newcommand\scH{{\mathscr H}}

\newcommand\scS{{\mathscr S}}

